\documentclass[reqno,11pt,graphicx,]{amsart}
\usepackage[colorlinks,linkcolor=blue,anchorcolor=yellow,citecolor=red,urlcolor=green]{hyperref} 
\usepackage{amsmath,amscd,amsthm,amsfonts,appendix,tikz}
\usepackage{amssymb,mathrsfs,xcolor,bbm,dsfont,verbatim}

\usepackage[OT2,OT1]{fontenc}
\usepackage{caption}
\usepackage{longtable}
\usepackage[all]{xy}
\numberwithin{equation}{section}
\allowdisplaybreaks[4]
\theoremstyle{plain}

\newcommand{\N}{{\mathbb N}}
\newcommand{\NN}{{\mathcal N}}
\newcommand{\n}{{\mathscr N}}

\newcommand{\Gg}{{\mathbf G}}

\newcommand{\Z}{{\mathbb Z}}
\newcommand{\II}{{\mathbb I}}
\newcommand{\R}{{\mathbb R}}

\newcommand{\Q}{{\mathbb Q}}
\newcommand{\QQ}{{\mathcal Q}}

\newcommand{\B}{{\mathcal B}}
\newcommand{\F}{{\mathbb F}}

\newcommand{\one}{{\mathbf 1}}
\newcommand{\two}{{\mathbf 2}}
\newcommand{\three}{{\mathbf 3}}
\newcommand{\A}{{\mathscr A}}
\newcommand{\T}{{\mathscr T}}

\newcommand{\tT}{{\mathbf t}}
\newcommand{\iI}{{\mathbf i}}

\newcommand{\eL}{{\mathscr L}}

\allowdisplaybreaks

\newtheorem{theorem}{Theorem}[section]
\newtheorem{proposition}[theorem]{Proposition}
\newtheorem{lemma}[theorem]{Lemma}
\newtheorem{corollary}[theorem]{Corollary}
\newtheorem{remark}[theorem]{Remark}
\newtheorem{definition}[theorem]{Definition}

\begin{document}
	\title{Almost sure dimensional properties for the spectrum and the density of states of Sturmian Hamiltonians}
	
	\author{Jie CAO}
	\address[J. Cao]{Chern Institute of Mathematics and LPMC, Nankai University, Tianjin 300071, P. R. China.}
	\email{caojie@nankai.edu.cn}
	\author{Yanhui QU}
	\address[Y.H. QU]{Department of Mathematical Science, Tsinghua University, Beijing 100084, P. R. China.}
	\email{yhqu@tsinghua.edu.cn}

	\begin{abstract}
		In this paper, we find a full Lebesgue measure set of frequencies $\hat \II\subset [0,1]\setminus \Q$ such that for any $(\alpha,\lambda)\in \hat \II\times [24,\infty)$, the Hausdorff and box dimensions of the spectrum of the Sturmian Hamiltonian $H_{\alpha,\lambda,\theta}$ coincide and are independent of $\alpha$. Denote the common value by $D(\lambda)$, we show that $D(\lambda)$ satisfies a Bowen's type formula, and  is locally Lipschitz. We also obtain the exact asymptotic behavior of $D(\lambda)$ as $\lambda$ tends to $ \infty.$ This considerably improves the result of Damanik and Gorodetski (Comm. Math. Phys. 337, 2015). We also show that for any $(\alpha,\lambda)\in \hat \II\times [24,\infty)$, the density of states measure of $H_{\alpha,\lambda,\theta}$ is exact-dimensional;  its Hausdorff and packing dimensions coincide and are independent of $\alpha$. Denote the common value by $d(\lambda)$, we  show that $d(\lambda)$ satisfies a Young's type formula, and is locally Lipschitz. We also obtain the exact asymptotic behavior of $d(\lambda)$ as $\lambda$ tends to $ \infty.$ During the course of study, we also answer or partially answer several questions in the same paper of Damanik and Gorodetski.
	\end{abstract}
	
	\maketitle	
	\section{Introduction}

	The {\it Sturmian Hamiltonian} is a discrete Schr\"odinger operator defined on $\ell^2(\Z)$ with Sturmian potential:
	\begin{equation*}
		(H_{\alpha,\lambda,\theta}\psi)_n:=\psi_{n+1}+\psi_{n-1}+\lambda \chi_{[1-\alpha,1)}(n\alpha+\theta\pmod 1)\psi_n,
	\end{equation*}
	where $\alpha\in [0,1]\setminus \Q$ is the {\it frequency}, $\lambda >0$ is the {\it coupling constant} and $\theta\in [0,1)$ is the {\it phase} (where $\chi_A$ is the indicator function of the set $A$). It is well-known that the spectrum of $H_{\alpha,\lambda,\theta}$ is independent of $\theta$ (see \cite{BIST}).
	We denote the spectrum by $\Sigma_{\alpha,\lambda}$.
	
	Another spectral object  is the so-called {\it density of states measure (DOS)} $\NN_{\alpha,\lambda}$ supported on $\Sigma_{\alpha,\lambda}$, which is defined by
	\begin{equation}\label{def-dos}
		\int_{\Sigma_{\alpha,\lambda}} g(E)d \NN_{\alpha,\lambda}(E):=\int_{\mathbb T}\langle \delta_0,g(H_{\alpha,\lambda,\theta})\delta_0\rangle d\theta, \ \ \ \forall g\in C(\Sigma_{\alpha,\lambda}).
	\end{equation}
	
	In this paper, we will study the fractal dimensional properties  of $\Sigma_{\alpha,\lambda}$
	and $\NN_{\alpha,\lambda}$ for typical $\alpha\in [0,1]\setminus \Q$. We also study the exact-dimensional property of $\NN_{\alpha,\lambda}$.
	
	We recall some notions, especially the definition of exact-dimensionality  from fractal geometry (see for example \cite{Fal97}).

Given a subset $A$ of a metric space $X$, we denote the {\it Hausdorff,  lower-box, upper-box} and {\it box dimensions} of $A$ by
\begin{equation*}
  \dim_HA,\ \  \underline\dim_BA,\ \ \overline \dim_BA,\ \  \dim_B A,
\end{equation*}
respectively.

Assume  $\mu$ is a finite Borel measure supported  on a  metric space $X$.   Fix $x\in X$, we define the {\it  lower} and {\it upper } local dimensions of $\mu$ at $x$ as
	\begin{equation}\label{def-loc-dim}
		\underline{d}_\mu(x):=\liminf_{r\to0}\frac{\log \mu(B(x,r))}{\log r}\ \ \ \text{ and }\ \ \ \overline{d}_\mu(x):=\limsup_{r\to0}\frac{\log \mu(B(x,r))}{\log r}.
	\end{equation}
	If  $\underline{d}_\mu(x)=\overline{d}_\mu(x)$, we say that the {\it local  dimension} of   $\mu$ at $x$ exists and denote it by  $d_\mu(x)$.

	The {\it Hausdorff} and {\it packing dimensions} of $\mu$ are defined as
	\begin{equation}\label{dim-meas}
		\begin{cases}
\dim_H\mu:=\sup\{s: \underline{d}_\mu(x)\ge s \text{ for  } \mu \text{ a.e. }x\in X\},\\
\dim_P\mu:=\sup\{s: \overline{d}_\mu(x)\ge s \text{ for  } \mu \text{ a.e. }x\in X\}.
		\end{cases}
	\end{equation}
	
	If there exists a constant $d$ such that $\underline{d}_\mu(x)=d\  (\overline{d}_\mu(x)=d)$ for $\mu$ a.e. $x\in X$, then necessarily $\dim_H\mu=d$ ($\dim_P\mu=d$). In this case we say that  $\mu$ is  {\it  exact lower- (upper-)dimensional}. If there exists a constant $d$ such that ${d}_\mu(x)=d$ for $\mu$ a.e. $x\in X$, then necessarily $\dim_H\mu=\dim_P\mu=d$. In this case we say that  $\mu$ is {\it exact-dimensional} (for more details, see \cite[Chapter 10]{Fal97}).

	\subsection{Background and previous results}\
	
	Sturmian Hamiltonians are popular models of one-dimensional quasicrystals and have been extensively studied since the pioneer works \cite{BIST,KKT,OPRSS}, see for example the survey papers \cite{Da07,Da17,DEG15} and references therein. In the following, we give a brief survey on the known  results  for Sturmian Hamiltonians, with an emphasis on the dimensional properties of the spectra.
	
	The most prominent model among Sturmian Hamiltonians is the {\it Fibonacci Hamiltonian}, for  which the frequency is taken to be  the golden number $\alpha_1:=(\sqrt{5}+1)/2.$ This model was  introduced by physicists  to model the quasicrystal, see  \cite{KKT,OPRSS}.  Casdagli \cite{Ca} first studied this model mathematically. He introduced the notion of pseudospectrum $B_\infty(\lambda)$ of the operator and showed that $B_\infty(\lambda)$ is a Cantor set of Lebesgue measure zero for $\lambda\ge 8$. Later, S\"ut\"o \cite{Su87,Su} showed that for any $\lambda>0,$ the pseudospectrum $B_\infty(\lambda)$ is the spectrum of the operator and the related spectral measure  is purely singular continuous. After that, the Fibonacci Hamiltonian has been extensively studied. Killip, Kiselev and Last \cite{KKL} studied the dynamical upper bounds on wavepacket spreading of Fibonacci Hamiltonian. The arguments they used inspired  later works on estimates of fractal dimensions of the spectra.
Now the spectral picture of Fibonacci Hamiltonian is quite complete, see \cite{C,DEGT,DG,DG2,DG3,DG4,JL,P,R}, especially \cite{DGY} for detail.   We summarize the results which are related to our paper in  Theorem A. To state it, we need to introduce the trace map dynamics related to Fibonacci Hamiltonian.
	
	Define the {\it Fibonacci trace map} $\mathbf T: \R^3\to \R^3$ as
	\begin{equation*}
		\mathbf T(x,y,z):=(2xy-z,x,y).
	\end{equation*}
	It is known that the function $G(x,y,z):=x^2+y^2+z^2-2xyz-1$
	is invariant under $\mathbf T$. Hence for any $\lambda>0$, the map $\mathbf T$ preserves the cubic surface
	\begin{equation*}
		S_\lambda:=\{(x,y,z)\in \R^3: G(x,y,z)=\lambda^2/4\}.
	\end{equation*}
	Write $\mathbf T_\lambda:=\mathbf T|_{S_\lambda}$ and let $\Lambda_\lambda$ be the set of points in $S_\lambda$ with bounded $\mathbf T_\lambda$-orbits. It is known that $\Lambda_\lambda$ is the non-wandering set of $\mathbf T_\lambda$ and is a locally maximal compact transitive hyperbolic set of $\mathbf T_\lambda$, see \cite{C,Ca,DG,Mei}.\\

	\noindent {\bf Theorem A}(\cite{C,DEGT,DG,DG2,DG3,DG4,DGY,JL,P})\  {\it
		Assume $\lambda>0$. Then the following hold:
		
		(1) The spectrum $\Sigma_{\alpha_1,\lambda}$ satisfies
		\begin{equation*}
			\dim_H \Sigma_{\alpha_1,\lambda}=\dim_B \Sigma_{\alpha_1,\lambda}=:D(\alpha_1,\lambda).
		\end{equation*}
		
		(2) $D(\alpha_1,\lambda)$ satisfies Bowen's formula: $D(\alpha_1,\lambda)$ is the unique zero of the  pressure function $P(t\psi_\lambda)$, where $\psi_\lambda$ is the geometric potential defined by
		$$\psi_\lambda(x):=-\log\|D\mathbf T_\lambda(x)|_{E^u_x}\|,$$
 where $E^u_x$ is the unstable subspace of the tangent space $T_xS_\lambda$.
		
		(3) The function $D(\alpha_1,\cdot)$ is analytic on $(0,\infty)$ and
		\begin{equation}\label{asym-golden-spectra}
			\lim_{\lambda\to 0} D(\alpha_1,\lambda)=1;\ \ \ \lim_{\lambda\to \infty} D(\alpha_1,\lambda)\log \lambda=\log (1+\sqrt{2}).
		\end{equation}
		
		(4) The DOS $\NN_{\alpha_1,\lambda}$ is exact-dimensional and consequently
		\begin{equation*}
			\dim_H \NN_{\alpha_1,\lambda}=\dim_P \NN_{\alpha_1,\lambda}=:d(\alpha_1,\lambda).
		\end{equation*}
		
		(5) $d(\alpha_1,\lambda)$ satisfies Young's  formula:
		\begin{equation}\label{dim-dos-golden}
			d(\alpha_1,\lambda)=\dim_H \mu_{\lambda,\max}=\frac{\log \alpha_1}{{\rm Lyap}^u\mu_{\lambda,\max}},
		\end{equation}
		where  $\mu_{\lambda,\max}$ is the measure of maximal entropy of $\mathbf T_\lambda$, and $\log\alpha_1, \ { \rm Lyap}^u\mu_{\lambda,\max}$ are the  entropy and the unstable Lyapunov exponent of $\mu_{\lambda,\max}$, respectively.		
		
		(6) The function $d(\alpha_1,\cdot)$ is analytic on $(0,\infty)$ and
		\begin{equation}\label{asym-golden-dos}
			\lim_{\lambda\to 0} d(\alpha_1,\lambda)=1;\ \ \ \lim_{\lambda\to \infty} d(\alpha_1,\lambda)\log \lambda=\frac{5+\sqrt{5}}{4}\log \alpha_1.
		\end{equation}
	}
	
See \cite{Bowen,Walters} for the definition of topological pressure and more generally the theory of thermodynamical formalism. See \cite{Young} for the original Young's dimension formula.

Now we turn to the spectral properties of general Sturmian Hamiltonians.
The first work on Sturmian Hamiltonians is \cite{BIST}. In this work, Bellissard et. al. laid the foundation for all the future studies on  Sturmian Hamiltonians. Among other things, they showed that for all $\lambda>0$, the spectrum $\Sigma_{\alpha,\lambda}$ is of Lebesgue measure zero. This motivates the study on the fractal dimensions of the spectrum.
   Based on \cite{BIST}, Raymond \cite{R}  showed that for $\lambda>4$, the spectrum $\Sigma_{\alpha,\lambda}$ has a natural covering structure. Using this structure, he could show that all the gaps of the spectrum  predicted by the gap labelling theory are open.
    See also \cite{BBBRT} for a new form of \cite{R}.
    In a  very recent work \cite{BBL}, Band, Beckus and Loewy  extended the above result to all $\lambda>0$ and solved the dry ten Martini problem for Sturmian Hamiltonians.

	The results for Fibonacci Hamiltonian are generalized to Sturmian Hamiltonians with ``nice" frequencies to various extents.
	Girand \cite{Gi} and  Mei \cite{Mei} considered the frequency $\alpha$ with eventually periodic continued fraction expansion. In both papers they showed that
	$\NN_{\alpha,\lambda}$ is exact-dimensional for small $\lambda$ and $\displaystyle\lim_{\lambda\to 0}\dim_{H} \NN_{\alpha,\lambda}=1.$
	This generalizes \eqref{asym-golden-dos}.
	For $\lambda>20$ and $\alpha$ with constant continued fraction expansion,  Qu \cite{Q1}   showed that $\NN_{\alpha,\lambda}$ is exact-dimensional and obtained similar asymptotic behaviors   as \eqref{asym-golden-spectra} and \eqref{asym-golden-dos} for $\lambda\to\infty.$  We remark that, for all works mentioned  above, the dynamical method is applicable due to the special types of the frequencies. Very recently, Luna \cite{Luna} generalized the results \cite{Gi,Mei} to including all the irrational frequencies of bounded-type.
	
	For general frequencies, the dimensional properties of the spectra are more complex.
Fix an irrational $\alpha\in (0,1)$  with continued fraction expansion $[a_1,a_2,\cdots]$. Based on the subordinacy theory,
	Damanik, Killip and Lenz \cite{DKL} showed that, if   $\limsup\limits_{k\rightarrow\infty}\frac{1}{k}\sum_{i=1}^k
	a_i<\infty$, then  $\dim_H \Sigma_{\alpha,\lambda}>0$. Later works  \cite{FLW,LPW07,LQW,LW,LW05}  were mainly built on Raymond's construction, let us explain them in more detail. The covering structure in \cite{R} makes it possible to define the so-called pre-dimensions $s_\ast(\alpha,\lambda)$ and $s^\ast(\alpha,\lambda)$ (see \eqref{pre-dim} for the exact definitions).
	Write
	\begin{equation*}\label{K-ast}
		K_\ast(\alpha):=\liminf_{k\rightarrow\infty}
		\left(\prod_{i=1}^k a_i\right)^{1/k}\ \ \ \text{ and }\ \ \  K^\ast(\alpha):=
		\limsup_{k\rightarrow\infty}\left(\prod_{i=1}^k a_i\right)^{1/k}.
	\end{equation*}

	The current  picture for the dimensions of the spectra is  the following:
	
	\noindent {\bf Theorem B}(\cite{ FLW,LPW07, LQW,LW})\  {\it
		Assume $\lambda\ge 24$ and $\alpha\in [0,1]\setminus \Q$. Then
		
		(1) The following dichotomies  hold:
		$$
		\begin{cases}
			\dim_H \Sigma_{\alpha,\lambda} \in(0,1) & \text{ if   }\ \  K_\ast(\alpha)< \infty\\
			\dim_H \Sigma_{\alpha,\lambda} =1 & \text{ if   }\ \  K_\ast(\alpha)= \infty
		\end{cases}
		\ \text{ and }\
		\begin{cases}
			\overline{\dim}_B \Sigma_{\alpha,\lambda} \in(0,1) & \text{ if   }\ \  K^\ast(\alpha)< \infty\\
			\overline{\dim}_B \Sigma_{\alpha,\lambda} =1 & \text{ if   }\ \  K^\ast(\alpha)= \infty
		\end{cases}.
		$$

		(2) $s_*(\alpha,\cdot)$ and $s^*(\alpha,\cdot)$ are
		Lipschitz continuous on any bounded interval of $[24,\infty)$ and
		$$
		\dim_H \Sigma_{\alpha,\lambda}=s_*(\alpha,\lambda)
		\ \ \ \text{ and }\ \ \  \overline{\dim}_B \Sigma_{\alpha,\lambda}=s^*(\alpha,\lambda).
		$$
		
		(3) There exist two constants $1<\rho_\ast(\alpha)\le \rho^\ast(\alpha)\le\infty $ such that
		 $$
		\lim_{\lambda\to \infty} s_*(\alpha,\lambda)\log \lambda
		=\log \rho_\ast(\alpha)\ \ \ \ \text{ and }\ \ \ \
		\lim_{\lambda\to \infty} s^*(\alpha,\lambda)\log \lambda
		= \log\rho^\ast(\alpha).
		$$

	}
	
	The dimensional properties of $\NN_{\alpha,\lambda}$ are less studied for general frequencies.
	In \cite{Q2}, Qu showed that for $\alpha$ with bounded continued fraction expansion and $\lambda>20$, the DOS $\NN_{\alpha,\lambda}$ is both exact upper- and lower-dimensional. He also constructed certain $\alpha$ such that the related  $\NN_{\alpha,\lambda}$ is not exact-dimensional. In \cite{JZ}, Jitomirskaya and Zhang constructed certain Liouvillian frequency $\alpha$ such that for any $\lambda>0$, the related DOS satisfies
	$\dim_H \NN_{\alpha,\lambda}<1$ but $\dim_P \NN_{\alpha,\lambda}=1$. Consequently, $\NN_{\alpha,\lambda}$ is not exact-dimensional.
	
	Until now, all the results are stated for deterministic frequencies. A natural question is that: how about the dimensional properties of $\Sigma_{\alpha,\lambda}$ and  $\NN_{\alpha,\lambda}$ for Lebesgue typical frequencies? For this question, Bellissard had the following conjecture in 1980s (see also \cite{DG15}):
	
	\noindent{\bf Conjecture:} {\it For every $\lambda>0$, the Hausdorff dimension of $\Sigma_{\alpha,\lambda}$ is Lebesgue a.e. constant in $\alpha$.}
	
	This conjecture was solved by Damanik and Gorodetski under a large coupling assumption:\\
	\noindent {\bf Theorem C}(\cite{DG15})\ {\it
		For every $\lambda\ge 24$, there exists two numbers $0<\underline{D}(\lambda)\le \overline{D}(\lambda)$ such that for Lebesgue a.e. $\alpha\in [0,1]\setminus \Q$,
		\begin{equation*}\label{DG-a.e.}
			\dim_H \Sigma_{\alpha,\lambda}=\underline{D}(\lambda)\ \ \text{ and }\ \ \overline{\dim}_B \Sigma_{\alpha,\lambda}=\overline{D}(\lambda).
		\end{equation*}
	}

Besides this work, Damanik et. al. \cite{DGLQ} studied the transport exponents of Sturmian Hamiltonians and they obtained a uniform upper bound for all time-averaged transport exponents in the large-coupling regime for Lebesgue a.e. $\alpha\in [0,1]\setminus \Q$. Munger \cite{Munger} showed that for $\lambda$ large and for Lebesgue a.e. $\alpha\in [0,1]\setminus \Q$, the DOS $\NN_{\alpha,\lambda}$ is not H\"older continuous.
	
	Based on Theorem C, one can raise several natural questions such as: for fixed $\lambda\ge 24,$ whether  $\underline{D}(\lambda)=\overline{D}(\lambda)$ holds? Do the full measure sets of frequencies depend on $\lambda$? How regular are the functions $\underline{D}(\lambda)$ and $\overline{D}(\lambda)$? What can one say about the DOS? etc.

Notice that, it is quite easy to construct a frequency $\alpha$ such that $K_\ast(\alpha)<\infty$ and $K^\ast(\alpha)=\infty$. Consequently by Theorem B, for such $\alpha$ and $\lambda\ge24$, we have $\dim_H\Sigma_{\alpha,\lambda}\in (0,1)$ but $\overline{\dim}_B\Sigma_{\alpha,\lambda}=1$. So a priori, it is possible that $\underline{D}(\lambda)<\overline{D}(\lambda)$.
	
In  this paper, we try to  answer  these questions and achieve a deeper understanding on the almost sure dimensional properties of Sturmian Hamiltonians.

\subsection{Main results}\
	
	
	Our main result for the dimensions of the spectrum is as follows:
	
	\begin{theorem}\label{main-dim-spectra}
		There exist  a   subset $\hat\II\subset [0,1]\setminus \Q$ of full Lebesgue measure and a function $D:[24,\infty)\to (0,1)$  such that the following hold:
		
		(1) For any $(\alpha,\lambda)\in\hat\II\times [24,\infty)$, the spectrum $\Sigma_{\alpha,\lambda}$ satisfies
		\begin{equation*}
			\dim_H \Sigma_{\alpha,\lambda}=\dim_B \Sigma_{\alpha,\lambda}=D(\lambda).
		\end{equation*}
		
		(2) $D(\lambda) $ satisfies a Bowen's type formula:  $D(\lambda)$ is the unique zero of the relativized pressure function $\mathbf{P}_{\lambda}(t)$ (see \eqref{relative-pre} for the exact definition).
		
		(3) $D(\lambda)$ is Lipschitz continuous on any bounded interval of $[24,\infty)$ and
		there exists a constant $\rho\in (1,\infty)$ such that
		\begin{equation}\label{dim-spectra-asym}
			\lim_{\lambda\to \infty} D(\lambda)\log \lambda= \log\rho.
		\end{equation}
	\end{theorem}
	
	\begin{remark}\label{rem-main-1}
		
			{\rm (i) Our result improve Theorem C in two aspects: firstly, the full measure set in Theorem C may depend on $\lambda$; while in our theorem, the full measure set $\hat\II$ is independent of $\lambda$. Secondly, our result shows that $\underline{D}(\lambda)=\overline{D}(\lambda)$ (Compare Theorem A(1)). That is, the two dimensions coincide. From fractal geometry point of view,   the spectrum $\Sigma_{\alpha,\lambda}$ is quite regular for typical frequency $\alpha$.
				
				(ii)  Through (2), we determine the  dimension $D(\lambda)$ explicitly and hence   answer \cite[Question 6.2]{DG15}.  Compare with the classical Bowen's formula for the expanding attractor (see for example \cite[Chapter 5]{Fal97}) and Theorem A(2), Theorem \ref{main-dim-spectra}(2) can be viewed as a random version of Bowen's formula.
				
				(iii)  The equation \eqref{dim-spectra-asym} gives the exact asymptotic behavior of $D(\lambda)$ as $\lambda \to \infty.$  It is a random version of \eqref{asym-golden-spectra}. In Section \ref{sec-proof-main-spec}, we determine $\rho^{-1}$ as the ``zero" of certain function $\varphi$ defined on $[0,1]$ (see \eqref{def-rho}), where each $\varphi(x)$ is the Lyapunov exponent of certain matrix-valued function $M_x$ (see \eqref{def-var-phi}).
  Thus we answer \cite[Question 6.3]{DG15}.

			}
	\end{remark}

	Our main result for the dimensions of the DOS is as follows:
	
	\begin{theorem}\label{main-dim-dos}
		 Let $\hat\II$ be the set in Theorem \ref{main-dim-spectra}. Then there exists	a function $d:[24,\infty) \to (0,1)$  such that the following hold:
		
		(1) For any $(\alpha,\lambda)\in\hat\II\times [24,\infty)$, the DOS $\NN_{\alpha,\lambda}$ is exact-dimensional and
		\begin{equation*}
			\dim_H \NN_{\alpha,\lambda}=\dim_P \NN_{\alpha,\lambda}=d(\lambda).
		\end{equation*}
		
		(2) $d(\lambda)$ satisfies a Young's type formula:
		\begin{equation}\label{young-form-dos}
			d(\lambda)=\frac{\gamma}{-(\Psi_\lambda)_*(\n)},
		\end{equation}
		where $\gamma$ is the Levy's constant(see \eqref{L-K}), $\n$ is a Gibbs measure (see Proposition \ref{Gauss-measure-Omega}) supported on the global symbolic space $\Omega$ (see \eqref{def-Omega}), and $\Psi_\lambda$ is the geometric potential  on $\Omega$ (see \eqref{def-Psi-lambda}).
		
		(3) $d(\lambda)$ is Lipschitz continuous  on any bounded interval of $[24,\infty)$  and there exists a constant $\varrho\in(1,\infty)$ such that
		\begin{equation}\label{dim-dos-asym}
			\lim_{\lambda\to\infty}d(\lambda)\log\lambda=\log\varrho.
		\end{equation}
	\end{theorem}
	
	\begin{remark}	\label{rem-main-2}	
{\rm	
(i)  Under a large coupling assumption ($\lambda\ge24$), we answer \cite[Question 6.5]{DG15}. Indeed, we say much more.
			
(ii)  Compare with the classical Young's dimension formula for ergodic measure supported on hyperbolic attractor (see \cite{Young}) and \eqref{dim-dos-golden}, the formula \eqref{young-form-dos} can be viewed as a random version of Young's formula, where $\gamma$ plays the role of entropy and $-(\Psi_\lambda)_*(\n)$ plays the role of  the Lyapunov exponent of the geometric potential $-\Psi_\lambda$ with respect to the ergodic measure $\n$.
			
(iii) See \eqref{def-var-rho} and \eqref{def-theta} for the exact value of $\varrho.$  \eqref{dim-dos-asym} is a random version of \eqref{asym-golden-dos}.
			
(iv) This theorem tells us that for typical frequency $\alpha$, the DOS $\NN_{\alpha,\lambda}$ behaves exactly as the special one $\NN_{\alpha_1,\lambda}$. On the other hand, the DOS appeared in \cite{Q2} and \cite{JZ} may behave badly. This is not a contradiction, since the set of frequencies in \cite{Q2} and \cite{JZ} are disjoint with $\hat \II,$ and  of zero Lebesgue measure, hence exceptional.			
}
\end{remark}

\subsection{Sketch of the main ideas}\label{sec-idea}\

At first we explain the proof of the  almost sure dimensional properties for the spectrum. In short, we use the ergodicity of Gauss measure together with the established deterministic result in \cite{LQW} to obtain the desired result.
Let us explain it in more detail and point out several key ingredients.

The basic tool is the symbolic coding introduced in \cite{R} and developed later  in \cite{FLW,LQW,LW,Q1,Q2}, see  Section \ref{sec-coding} for  detail.
	
	For each $\alpha\in [0,1]\setminus \Q$ and $\lambda>4$, Raymond  \cite{R} constructed a family $\{\B_n^\alpha(\lambda):n\ge0\}$ of decreasing coverings of $\Sigma_{\alpha,\lambda}$. 
By using this family,  Liu, Qu and Wen \cite{LQW} defined
 the so-called pre-dimensions $s_\ast(\alpha,\lambda)$ and $s^\ast(\alpha,\lambda)$ and showed that
\begin{equation*}
  \dim_H \Sigma_{\alpha,\lambda}=s_\ast(\alpha,\lambda);\ \ \overline{\dim}_B \Sigma_{\alpha,\lambda}=s^\ast(\alpha,\lambda).
\end{equation*}
Our primary goal is to show that $s_\ast(\alpha,\lambda)=s^\ast(\alpha,\lambda)$ for a.e. $\alpha.$

The first key observation is that: for typical $\alpha,$  the pre-dimensions are  ``zeros" of certain pressure-like  functions. This is inspired by  the classical Bowen's dimension formula.

 Let us define pre-dimensions.  For $t\ge 0$, define the $n$-th partition function as
 \begin{equation}\label{partition}
 \mathscr P_{\alpha,\lambda}(n,t):=\sum_{B\in\B_n^\alpha(\lambda) }|B|^{t},
\end{equation}
where $|B|$ is the length of the band $B.$ Let $s_n(\alpha,\lambda)$ be such that
$\mathscr P_{\alpha,\lambda}(n,s_n(\alpha,\lambda))=1.$
 Then the pre-dimensions are defined by
\begin{equation*}
  s_\ast(\alpha,\lambda):=\liminf_{n\to\infty} s_n(\alpha,\lambda)\ \ \text{ and }\ \  s^\ast(\alpha,\lambda):=\limsup_{n\to\infty} s_n(\alpha,\lambda).
\end{equation*}

We note that quantity similar with  \eqref{partition} appears naturally in  dynamical setting. We take a cookie-cutter system with two branches as example  (see \cite[Chapters 4 and 5]{Fal97}).  For the repeller $E$ of the system, one can naturally construct a covering $\{X_I: I\in \{0,1\}^n\}$ consisting of $n$-th basic sets $X_I$, and define the related partition function $\sum_{I\in \{0,1\}^n}|X_I|^t$. Moreover, the following limit  exists
$$
P(t):=\lim_{n\to \infty} \frac{1}{n}\log \sum_{I\in \{0,1\}^n}|X_I|^t
$$
 and the limit function $P$ is called the {\it pressure function} of the system. The classical Bowen's formula reads that: the Hausdorff and box dimensions of the repeller $E$ are the same and are given by the zero of  $P.$

Back to our case. Notice that,  $s_n(\alpha,\lambda)$ is the zero of $\log \mathscr P_{\alpha,\lambda}(n,t)/n$. If the
 limit
$$
\lim_{n\to \infty} \frac{1}{n}\log \mathscr P_{\alpha,\lambda}(n,t)
$$
exists for all $t$, one expect that $s_n(\alpha,\lambda)$ also tends to the zero of the limit function. However, it is not the case  for general $\alpha$. Instead, one consider the lower and upper limits:
$$
\underline{\rm P}_{\alpha,\lambda}(t):=\liminf_{n\to \infty} \frac{1}{n}\log \mathscr P_{\alpha,\lambda}(n,t);\ \ \overline{\rm P}_{\alpha,\lambda}(t):=\limsup_{n\to \infty} \frac{1}{n}\log \mathscr P_{\alpha,\lambda}(n,t).
$$
Then one expect that the pre-dimensions are the zeros of these two pressure functions respectively. We show that  these are true (after suitably defining the ``zero") for typical $\alpha$. We call these results generalized Bowen's formulas.

   Now, to show the coincidence of two dimensions, it is enough to show that the two pressure functions coincide. At this point, we will use the ergodicity of the Gauss measure. Then come the second key observation: the family of functions $\{\log \mathscr P_{\cdot,\lambda}(n,t):n\in \N \}$ are almost sub-additive over the system $([0,1],T,G)$, where $T$ is the Gauss map and $G$ is the Gauss measure (note that Gauss measure is equivalent to the Lebesgue measure on $[0,1]$). So one can define the so-called relativized pressure function $\mathbf P_\lambda(t)$ by
    $$
    \mathbf P_\lambda(t)=\lim_{n\to\infty}\frac{1}{n}\int_{[0,1]}\log \mathscr P_{\alpha,\lambda}(n,t)d G(\alpha).
    $$
 Since $G$ is ergodic, by Kingman's theorem,  we do have $\underline{\rm P}_{\alpha,\lambda}(t)=\overline{\rm P}_{\alpha,\lambda}(t)=\mathbf P_\lambda(t)$ for a.e. $\alpha$. Notice that the full measure set of frequencies depends on $t$. We can get rid of this dependence  by using an extra regularity of the upper pressure function $\overline{\rm P}_{\alpha,\lambda}(t)$. So we achieve the first goal. As a byproduct, we also show that the dimension as a function of $\alpha$ is almost constant.

Up to now, for fixed $\lambda$, we succeed to show that  for typical $\alpha$, the two fractal dimensions of the spectrum coincide and obtain a Bowen's type formula for the dimension. The last step is to find the $\lambda$-independent full measure set. But this is more or less standard.

\medskip

Now we turn to the dimensional properties of the DOS. As we mentioned earlier, except for nice frequencies, the dimensional properties of DOS are much less studied. For nice frequencies, all the existing results suggest that the related DOS is also nice since it is exact-dimensional and its dimension satisfies  Young's  formula.
See for example Theorem A(5) and \cite[Theorem 11]{Q1}.

Let us look more closely on the approach of \cite{Q1}, where the structure of the  coding plays essential role.
In \cite{LQW,Q1,Q2},  a coding $\Omega^\alpha$ for $\Sigma_{\alpha,\lambda}$  was constructed  based on Raymond's construction.  The coding is a generalization of topological Markov shift. The usual topological Markov shift is a subshift with one alphabet and one incidence matrix. In our case, $\Omega^\alpha$ is defined by a family of alphabets $\{\T_{0}, \A_{a_1},\A_{a_2},\cdots\}$ and a family of incidence matrices $\{A_{a_j a_{j+1}}:j\ge 1\}$, where $\{a_n:n\ge1\}$ come from the continued fraction expansion  $\alpha=[a_1,a_2,\cdots]$.
More precisely, let $\Omega_n^\alpha$ be the set of admissible words of order $n$. One can code the bands in $\B_n^\alpha(\lambda)$   such that
	$$
	\B_n^\alpha(\lambda)=\{B_w^\alpha(\lambda): w\in \Omega_n^\alpha\}.
	$$
	Now for any $E\in \Sigma_{\alpha,\lambda}$, there exists a unique infinite sequence $x\in \Omega^\alpha$ such that
	$$
	\{E\}=\bigcap_{n\ge1} B_{x|_n}^\alpha(\lambda),
	$$
where $x|_n$ is the $n$-th prefix of $x$.
	In this way, one defines a coding map $\pi^\alpha_\lambda: \Omega^\alpha\to \Sigma_{\alpha,\lambda}$ such that $\pi^\alpha_\lambda(x)=E$.
One can also endow a metric on $\Omega^\alpha$ as
$$
	\rho^\alpha_\lambda(x,y):=|B^\alpha_{x\wedge y}(\lambda)|,
	$$
where $x\wedge y$ denotes the common prefix of $x$ and $y$.  With this metric,  the coding map $\pi^\alpha_\lambda$ is always Lipschitz. It is bi-Lipschitz if the continued fraction expansion of  $\alpha$ are bounded (see \cite{Q2}).

Now assume $\alpha=[k,k,\cdots]$. Then two metric spaces $(\Omega^\alpha,\rho^\alpha_\lambda)$ and $(\Sigma_{\alpha,\lambda},|\cdot|)$ are bi-Lipschitz equivalent. If we denote the pull-back of $\NN_{\alpha,\lambda}$ through $\pi^\alpha_\lambda$  by $\mu^\alpha$, then the dimensional properties of $\mu^\alpha$ and $\NN_{\alpha,\lambda}$ are the same. Notice that in this case, $\Omega^\alpha$ is essentially a Markov shift with alphabet $\A_k$ and incidence matrix $A_{kk}$. It is not hard to see that
 \begin{equation}\label{idea-mu-asymp}
 \mu^\alpha([x|_n])\sim \frac{1}{\# \Omega^\alpha_n}\sim \frac{1}{q_n(\alpha)},
 \end{equation}
  where $q_n(\alpha)$ is the denominator of the $n$-th convergent of $\alpha.$
Thus $\mu^\alpha$ is essentially the measure of maximal entropy of this Markov shift.
It is well-known that $\mu^\alpha$ is exact-dimensional and the dimension is given by Young's formula:
$$
\dim_H \mu^\alpha=\frac{h(\mu^\alpha)}{-\left(\Psi_\lambda^\alpha\right)_\ast(\mu^\alpha)},
$$
where $h(\mu^\alpha)$ is the metric entropy of $\mu^\alpha$,   $\Psi_\lambda^\alpha:=\{\psi_{\lambda,n}^\alpha:n\ge0\}$ is the  geometric potential on $\Omega^\alpha$ defined by
	$$
	\psi_{\lambda,n}^\alpha(x):=\log |B^\alpha_{x|_n}(\lambda)|,
	$$
and $-\left(\Psi_\lambda^\alpha\right)_\ast(\mu^\alpha)$ is the Lyapunov exponent of $\Psi_\lambda^\alpha$ w.r.t. $\mu^\alpha.$

Next we explain why $\mu^\alpha$ is exact-dimensional.
Given $x\in \Omega^\alpha$, the local dimension of $\mu^\alpha$ at $x$  can be computed by
\begin{equation}\label{loc-dim-alpha}
  d_{\mu^\alpha}(x)=\lim_{n\to\infty}\frac{\log \mu^\alpha([x|_n])}{\log {\rm diam}([x|_n])}=\lim_{n\to\infty}\frac{-\log q_n(\alpha)}{\psi_{\lambda,n}^\alpha(x)}.
\end{equation}
It is known in this case that $\log q_n(\alpha)/n\to h(\mu^\alpha)=-\log \alpha.$
On the other hand,   one can show that $\Psi^\alpha_\lambda$ is almost-additive. By Kingman's ergodic theorem, for $\mu^\alpha$-a.e. $x$,
$$
\frac{\psi_{\lambda,n}^\alpha(x)}{n}\to\left(\Psi_\lambda^\alpha\right)_\ast(\mu^\alpha)
=\lim_{n\to\infty}\frac{1}{n}\int_{\Omega^\alpha} \psi^\alpha_{\lambda,n}d\mu^\alpha.
$$
Thus $\mu^\alpha$ is exact-dimensional.

Now we discuss the general frequency case. Fix $\alpha\in [0,1]\setminus\Q$, one can repeat all the constructions  above. In general, the situation become worse: at first, the inverse of $\pi^\alpha_\lambda$ is  not   Lipschitz any more; secondly, $\Omega^\alpha$ is not necessarily a Markov shift and $\mu^\alpha$ is not necessarily ergodic. However, since we only care about almost sure property, we can use the ergodicity of Gauss measure again to repair these. At this point, we will prove a geometric lemma--gap lemma, and use it to show that  for typical $\alpha$, the map $\pi^\alpha_\lambda$
 is ``almost bi-Lipschitz". Consequently, $\mu^\alpha$ and $\NN_{\alpha,\lambda}$ still have the same dimensional properties. One can further show that \eqref{idea-mu-asymp} still holds and one can still compute the local dimension of $\mu^\alpha$ by \eqref{loc-dim-alpha}. It is well-known that, for typical $\alpha$,
 $$
 \frac{\log q_n(\alpha)}{n}\to \gamma,
 $$
 where $\gamma$ is the L\'evy's constant.
 Thus to show that $\mu^\alpha$ is exact-dimensional for typical $\alpha$ and the dimension of $\mu^\alpha$  does not dependent on $\alpha,$ the remaining  thing is to check that there exists a constant $\mathscr L>0$ such that for $G$-a.e. $\alpha$ and $\mu^\alpha$- a.e. $x\in \Omega^\alpha,$
 $$
 \frac{\psi^\alpha_{\lambda,n}(x)}{n}\to -\mathscr L.
 $$

 This consideration motivates us to define a global symbolic space $\Omega$ and a measure $\n$ on $\Omega,$ which is related to entropy,  as follows:
 \begin{equation}\label{naive}
   \Omega:=\bigsqcup_{\alpha\in [0,1]\setminus \Q}\Omega^\alpha;\ \ \ \  \n:=\int_{\alpha\in [0,1]\setminus \Q}\mu^\alpha d G(\alpha).
 \end{equation}
Hopefully, $\Omega$ is a nice topological Markov shift and $\n$ is ergodic. Then by putting $\Psi_{\lambda}^\alpha$ together to form a global geometric potential $\Psi_\lambda$, and using the ergodicity of $\n$, we obtain the desired result for typical frequency.

 There are several technical difficulties. The first difficulty is that, by checking the definition of $\Omega^\alpha$ (see \eqref{Omega^alpha}) it is seen that the naive definition of $\Omega$ as in \eqref{naive} does not work since every sequence in $\Omega^\alpha$ starts with some special letter from $\T_0$, then after one shift, it is not in any $\Omega^\beta$. Indeed, the only reasonable definition of $\Omega$ is that: it is  a topological Markov shift with a full alphabet $\A$ which is the union of all $\A_n$ (see \eqref{def-A}) and an incidence matrix $A$ which is made of all the $A_{mn}$ in a right way (see \eqref{A-and-A-nm}). If we define $\Omega$ like that, it becomes a nice topological dynamical system, which has $([0,1]\setminus\Q,T,G)$ as its factor system through a natural projection map $\Pi:\Omega\to [0,1]\setminus\Q$.

  Now, one get a dynamical object: the fiber $\Omega_\alpha:=\Pi^{-1}(\alpha)$, which is different from $\Omega^\alpha$. A priori, it is not related to any spectrum. But it come as a surprise that there is a bijection between $\Omega_\alpha$  and $\Omega^{\check \alpha}$, where $\alpha=[a_1,a_2,\cdots]$ and $\check \alpha=[1,a_1,a_2,\cdots]$. Through this bijection, one can pull back the DOS $\NN_{\check \alpha,\lambda}$ to $\Omega_\alpha$ to obtain a kind of measure of maximal entropy $\n_\alpha$. Then one integrate the $\n_\alpha$ w.r.t. the Gauss measure $G$ to get a nice measure $\n$ on $\Omega.$
Once $\n$ is constructed, one can follow the line we have explained to derive the desired result.

  Up to now, we have solved the problem for   typical $\check\alpha$. To complete the proof, we need another geometric lemma--tail lemma, which says that if the expansions of $\alpha$ and $\beta$ have the same tail, then $\NN_{\alpha,\lambda}$ and $\NN_{\beta,\lambda}$
 have the same dimensional property.  By this lemma, we can relate the dimension property of $\Sigma_{\check \alpha,\lambda} $ to that of $\Sigma_{\alpha,\lambda} $ and conclude the proof.

\subsection{Further remarks}\
	
Given the recent result \cite{BBL}, one may expect that the method of the present paper can be applied to the small coupling regime. However, there are two technical difficulties to overcome. At first,  to apply the thermodynamical formalism, we heavily rely on  the bounded variation and covariation properties of Sturmian Hamiltonians established in \cite{LQW}, which are only available for large coupling for the moment. A more serious difficulty is that for $\lambda<4$, the covering structure become worse: the sub-bands in a father band may overlap each other. This will create essential difficulty for estimating the dimensions. Maybe a more promising approach  is by further developing the method used in \cite{Luna}.

 In Theorem \ref{main-dim-spectra}(3) and Theorem \ref{main-dim-dos}(3), we obtain the local Lipschitz regularity of  $D(\lambda)$ and $d(\lambda)$. Compare with Theorem A(3) and (6), it is reasonable to guess that both $D(\lambda)$ and $d(\lambda)$ are indeed real analytic. However, it seems that our method can hardly be adapted to deal with this problem. New idea is needed.

	At last, we say a few words about  notations.

In this paper, we use $\lhd$ and $\rhd$ to indicate the beginning and the end of a claim.

Throughout this paper, we write
\begin{equation*}\label{}
		\N:=\{1,2,3,\cdots\},\ \   \Z^+:=\{0\}\cup \N, \ \ \R^+:=[0,\infty),\ \ \Q^+:=\R^+\cap \Q.
	\end{equation*}

For two positive sequences $\{a_n:n\in \N\}$ and $\{b_n:n\in \N\}$,
   \begin{equation*}
     a_n\lesssim b_n\Leftrightarrow  a_n\le C b_n,\ \forall n,
   \end{equation*}
   where $C>1$ is a constant.
   \begin{equation*}
     a_n\sim b_n\Leftrightarrow a_n\lesssim b_n \text{ and } b_n\lesssim a_n.
   \end{equation*}

	Assume $X$ is a metric space and $\mu,\nu$ are two finite Borel measures on $X$. We say $\mu$ and $\nu$ are equivalent, denote by $\mu\asymp \nu$, if there exists a constant $C>1$ such that for any Borel set $B\subset X$, we have
	\begin{equation*}
		C^{-1}\nu(B)\le\mu(B)\le C\nu(B).
	\end{equation*}

As a final remark,  since later we mainly work on the symbolic space, we will substitute $\alpha$ to $a$, the continued fraction expansion of $\alpha$, to simplify the exposition. See especially the convention \eqref{convention}.
	
	The rest of the paper is organized as follows. In Section \ref{sec-preliminary}, we present the necessary materials which are needed for the proof of Theorem \ref{main-dim-spectra}. In Section \ref{sec-proof-main-spec}, we prove Theorem \ref{main-dim-spectra}. In Section \ref{sec-dim-fiber-symbolic}, we prove the symbolic version of Theorem \ref{main-dim-dos}.
In Section \ref{sec-geometric-lemma}, after stating two geometric lemmas, we finish the proof of Theorem \ref{main-dim-dos}. In Section \ref{sec-proof-tech}, we prove several technical results. In Section \ref{sec-proof-geo-lem}, we prove  two geometric lemmas. In Section \ref{sec-app}, we  provide a table of notations for the convenience of the readers.

\section{Preliminaries}\label{sec-preliminary}

	In this section, we prepare for the proof of Theorem \ref{main-dim-spectra}. At first, we review the continued fraction theory, with emphasis on the symbolic side. Then we recall  the covering structure and the coding of the spectrum. At last, we list several known results for  Sturmian Hamiltonians.

\subsection{Recall on the theory of continued fraction}\label{sec-continued-fra}\

We review briefly some aspects of the theory of continued fraction. For more details, see \cite{Bi,EW2011,IK}.

	From now on, we write  the set of irrationals in $[0,1]$ as
	\begin{equation}\label{def-I}
		\II:=[0,1]\setminus \Q.
	\end{equation}

	The  {\it Gauss map} $T: \mathbb I\to \II$ is defined by $T(x):=\{1/x\}$, where $\{x\}$ is the fractional part of $x$.
	
	Assume $\alpha\in \II$ has the continued fraction expansion
	\begin{equation}\label{cfe}
		\alpha=\dfrac{1}{a_1(\alpha)+ \dfrac{1}{a_2(\alpha)+ \dfrac{1}{\ddots}}}=[a_1(\alpha),a_2(\alpha),\cdots].
	\end{equation}
	It is classical  that $T(\alpha)=[a_2(\alpha),a_3(\alpha),\cdots].$
	
	The {\it Gauss measure} $G$ on $\II$ is defined by
	\begin{equation*}\label{gauss-meas}
		G(A):=\frac{1}{\log 2}\int_{A}\frac{dx}{1+x}.
	\end{equation*}
	It is seen that $G$ and the Lebesgue measure on $\II$ are equivalent, thus they have  the same null sets.
	It is well-known that $G$ is $T$-invariant and ergodic (see for example \cite[Theorem 3.7]{EW2011}).

Later in this paper, we will extensively work on symbolic space, so it is quite useful to present the symbolic representation of the system $(\II,T,G)$.
	
	At first, we recall  the definition  of the {\it  full shift}  over $\N$.  It is the topological dynamical system $(\N^\N, d, S)$, where $d$ is the metric:
\begin{equation*}
		d(a,b):=2^{-|a\wedge b|},
	\end{equation*}
where $a\wedge b$ denotes the common prefix of $a$ and $b$ and $|w|$ is the length of the word $w$;  $S:\N^\N\to\N^\N$ is the {\it shift map} defined by
$$S(a)=S((a_n)_{n\in\N}):=(a_{n+1})_{n\in\N}.$$
Later we often simplify the notations $S(a), S^n(a)$ to $Sa, S^na$.

For any $n\in\N$ and $\vec{a}=a_1a_2\cdots a_n\in\N^n$, define the {\it cylinder} $[\vec a]$ as
	\begin{equation*}
		[\vec{a}]:=\left\{a\in\N^\N :a|_n=\vec{a}\right\},
	\end{equation*}
where $a|_n$ is the prefix of $a$ of length $n$.
	
Now define a map $\Theta:\II\to \N^\N$ as
	\begin{equation}\label{def-Theta}
		\Theta(\alpha)=a(\alpha), \ \ \ \text{where }\ \  a(\alpha)=(a_n(\alpha))_{n\in \N},
	\end{equation}
	and $a_n(\alpha)$ is defined by \eqref{cfe}.
	The following property is standard:
	
	\begin{proposition}
		$\Theta:\II\to \N^\N$ is a topological conjugation between $(\II,T)$ and $(\N^\N,S)$. More precisely, $\Theta$ is a homeomorphism such that for any $\alpha\in \II$, we have
		$$
		\Theta\circ T(\alpha)=S\circ \Theta(\alpha).
		$$
	\end{proposition}
	
	Now we write the image of $G$ under $\Theta$ as
	\begin{equation*}
		\Gg:=\Theta_\ast(G).
	\end{equation*}
	Then $\Gg$ is invariant under $S$ and is also ergodic. We call $\Gg$ the {\it Gauss measure} on $\N^\N$. We summarize the above  consideration in the following commutative diagram:
	$$
	\xymatrix{
		(\II,G)\ar[d]_{\Theta} \ar[r]^{T}  &(\II,G)  \ar[d]^{\Theta}  \\
		(\mathbb{N}^{\mathbb{N}},\Gg) \ar[r]^{S} & (\mathbb{N}^{\mathbb{N}},\Gg)  }$$

	Next we recall the algorithm  of rational approximations of $\alpha\in\II$. Assume $\alpha=[a_1,a_2,\cdots]$.   Define  integer sequences $\{p_n(\alpha)\}_{n\geq-1}$ and $ \{q_n(\alpha)\}_{n\geq-1}$ recursively as
	\begin{eqnarray}\label{def-p-q-n}
		\begin{cases}
			p_{-1}(\alpha)=1,\ p_0(\alpha)=0, & p_{n+1}(\alpha)=a_{n+1}p_n(\alpha)+p_{n-1}(\alpha), \ (n\ge 0);\\
			q_{-1}(\alpha)=0,\ q_0(\alpha)=1, & q_{n+1}(\alpha)=a_{n+1}q_n(\alpha)+q_{n-1}(\alpha), \ (n\ge 0).
		\end{cases}
	\end{eqnarray}
	
	The fraction $p_n(\alpha)/q_n(\alpha)$ is called  the $n$-th {\it convergent}  of $\alpha$.

Assume $a=\Theta(\alpha)$. Later we always use the following convention:
	\begin{equation}\label{def-q-n-a}
		p_n(a):= p_n(\alpha);\ \ \ q_n(a):= q_n(\alpha).
	\end{equation}

\begin{remark}\label{q-n-vec-a}
  {\rm
  By \eqref{def-p-q-n}, it is seen that $p_n(a)$ and $q_n(a)$ only depend on the prefix $a|_n$ of $a$. Thus for any $\vec a\in \N^n$, we define
  \begin{equation*}
    p_k(\vec a):=p_k(a),\ \ q_k(\vec a):=q_k(a), \ \ \  k=-1,\cdots,n,
  \end{equation*}
  where $a\in \N^\N$ is such that $a|_n=\vec a.$
  }
\end{remark}
	
	The following two properties are well-known:
	\begin{lemma}\label{q-n}
		(1) For any $a\in \N^\N$ and $n,m\in \N$, we have
		\begin{equation*}\label{q-n-m}
			q_n(a)q_m(S^na)\le q_{n+m}(a)\le 2q_n(a)q_m(S^na).
		\end{equation*}
		
		(2) For any $\vec a\in\N^n$, we have $\Gg([\vec a])\sim q_n(\vec a)^{-2}.$
	\end{lemma}

See for example \cite[Lemma 4.1]{Q2} for (1) and \cite[Eq. (3.23)]{EW2011} for (2).

	At last, we mention two famous constants related to  continued fraction theory, see for example \cite[Corollary 3.8]{EW2011}. The result is formulated in the system $(\N^\N,S,\Gg)$.
	
	\begin{theorem}[\cite{YK,LS}]\label{Levy-Khinchin}
		For $\Gg$-a.e. $a\in \N^\N$, we have
		\begin{equation}\label{L-K}
			\frac{\log q_n(a)}{n}\to\gamma:=\frac{\pi^2}{12\log2};\ \ \ \left(\prod_{i=1}^{n}a_i\right)^{1/n}\to \kappa:=2.68\cdots.
		\end{equation}
		\end{theorem}

We have mentioned in the introduction that $\gamma$ is the   Levy's constant. Another constant  $\kappa$ is called  {\it Khinchin's constant}.

Now we pick up a $\Gg$-full measure set $\F_1$ as our initial set of frequencies. All the other $\Gg$-full measure sets appeared later are subsets of $\F_1$. We will see that, the spectral properties of the operator with frequency in $\F_1$ is already good enough (see Proposition \ref{Bowen-formula}).

\begin{proposition}\label{F-proposition}
There exists a $\Gg$-full measure set $\F_1\subset \N^\N$ such that for any $a\in\mathbb{F}_1$, the following hold:

(1) The following limits exist:
		\begin{equation*}
				\lim\limits_{n\to\infty}\frac{\log q_n(a)}{n}=\gamma;\ \ \ \lim_{n\to\infty}\left(\prod_{i=1}^{n}a_i\right)^{1/n}=\kappa;\ \ \lim\limits_{n\to\infty}\frac{\log a_n}{n}=0.
			\end{equation*}

(2) There  exists a sequence $n_k\uparrow \infty$, such that $a_{n_k+i}=1,\ i=1,\cdots,7$ for any $k\in \N$.
\end{proposition}

\begin{proof}
By Theorem \ref{Levy-Khinchin}, there exists a $\Gg$-full measure set $\mathbb H_1\subset \N^\N$ such that for any $a\in \mathbb H_1$, (1) holds.
Since $\Gg$ is ergodic, there exists a $\Gg$-full measure set $\mathbb H_2\subset \N^\N$ such that for any $a\in \mathbb H_2,$
  \begin{equation*}
\lim\limits_{n\to\infty}\frac{1}{n}\sum_{k=0}^{n-1}\chi_{[1^{7}]}(S^ka)=
\int_{\N^\N}\chi_{[1^7]}d\Gg=\Gg([1^{7}])>0,
\end{equation*}
 consequently, (2) holds. Now take $\F_1:=\mathbb H_1\cap \mathbb H_2.$	
\end{proof}

\subsection{The covering structure of the spectrum.}\

Since later we almost only work with the continued fraction expansion of the frequency, form now on we  take the following convention:
	
	\noindent {\bf Convention A}: {\it assume $\alpha\in \II$ and $a=\Theta(\alpha)\in\N^\N$. Then we write
		\begin{equation}\label{convention}
			\Sigma_{a,\lambda}:=\Sigma_{\alpha,\lambda};\ \ \  \NN_{a,\lambda}:=\NN_{\alpha,\lambda};\ \ \ H_{a,\lambda,\theta}:=H_{\alpha,\lambda,\theta}.
	\end{equation}}
	We will use this convention  throughout this paper without further mention.
	
	\vspace{1ex}
	Follow \cite{LW,R}, we describe  the covering structure of the spectrum $\Sigma_{a,\lambda}$.
	
	For $a\in \N^\N$, write $\alpha:=\Theta^{-1}(a)$.  The {\it standard} Strumian sequence $\{\mathscr S_k(a):k\in \Z\}$ is defined by
	\begin{equation*}
		\mathscr S_k(a):=\chi_{[1-\alpha,1)}(k\alpha\pmod1),\ \ \ (k\in\Z).
	\end{equation*}
	
	The standard Sturmian sequence has the following nice combinatorial property:
	\begin{lemma}\label{basic-sturm}
		If $a,b\in \N^\N$ satisfy $a|_n=b|_n$, then
		$$
		\mathscr S_k(a)=\mathscr S_k(b),\ \ \forall \ 1\le k\le q_n(a).
		$$
	\end{lemma}

\begin{proof}
  It follows directly from \cite{BIST} equation (6),  \cite{BIST} Lemma 1 a)  and Remark \ref{q-n-vec-a}.
\end{proof}
	
	Recall that the Sturmian potential is given by
	$v_k=\lambda\chi_{[1-\alpha,1)}(k\alpha+\theta\pmod 1).$
	Since $\Sigma_{a,\lambda}$ is independent of the phase $\theta$, in the rest of the paper we will take $\theta=0.$ So in this case we have
	$v_k=\lambda \mathscr S_k(a).$

	For any $n\in \N$ and $ E,\lambda\in\mathbb{R}$, the {\it transfer matrix} $M_n(E;\lambda)$ over $q_n(a)$ sites is defined by
	$$
	{\mathbf M}_n(E;\lambda):=
	\left[\begin{array}{cc}E-v_{q_n(a)}&-1\\ 1&0\end{array}\right]
	\left[\begin{array}{cc}E-v_{q_n(a)-1}&-1\\ 1&0\end{array}\right]
	\cdots
	\left[\begin{array}{cc}E-v_1&-1\\ 1&0\end{array}\right].
	$$
	By convention we  define
	$$
	\begin{array}{l}
		{\mathbf M}_{-1}(E;\lambda):= \left[\begin{array}{cc}1&-\lambda\\
			0&1\end{array}\right]
	\end{array}
	\ \ \ \text{ and }\ \ \
	\begin{array}{l}
		{\mathbf M}_{0}(E;\lambda):= \left[\begin{array}{cc}E&-1\\
			1&0\end{array}\right].
	\end{array}
	$$
	For $n\ge0$ and $p\ge-1$, define  trace polynomials and related  periodic approximations of the spectrum as
	\begin{equation*}\label{trace-polynomial}
		h_{(n,p)}(E;\lambda):={\rm tr} ({\mathbf M}_{n-1}(E;\lambda) {\mathbf M}_n^p(E;\lambda)) \ \ \text{and }\ \  \sigma_{(n,p)}(\lambda):=\{E\in\mathbb{R}:|h_{(n,p)}(E;\lambda)|\leq2\}.
	\end{equation*}
	The set  $\sigma_{(n,p)}(\lambda)$ is made of finitely many  disjoint intervals.
	Moreover, for any $n\ge 0,$
	\begin{equation*}\label{covering}
		\sigma_{(n+2,0)}(\lambda)\cup\sigma_{(n+1,0)}(\lambda)\subset
		\sigma_{(n+1,0)}(\lambda)\cup\sigma_{(n,0)}(\lambda);\ \ \ \Sigma_{a,\lambda}=\bigcap_{n\ge0}\left(\sigma_{(n+1,0)}(\lambda)
		\cup\sigma_{(n,0)}(\lambda)\right).
	\end{equation*}
	Each interval of  $\sigma_{(n,p)}(\lambda)$ is called a  {\em band}. Assume
	$B$ is a band of $ \sigma_{(n,p)}(\lambda)$, then   $h_{(n,p)}(\cdot;\lambda)$ is monotone on $B$ and $h_{(n,p)}(B;\lambda)=[-2,2].$ We call $h_{(n,p)}(\cdot;\lambda)$ the {\em generating polynomial} of $B$ and denote it by $h_B(\cdot;\lambda):=h_{(n,p)}(\cdot;\lambda)$.
	
	Note that $\{\sigma_{(n+1,0)}(\lambda)\cup\sigma_{(n,0)}(\lambda):n\ge 0\}$ forms a decreasing family of coverings  of $\Sigma_{a,\lambda}$. However there are some repetitions between $\sigma_{(n+1,0)}(\lambda)\cup\sigma_{(n,0)}(\lambda)$
	and $\sigma_{(n+2,0)}(\lambda)\cup\sigma_{(n+1,0)}(\lambda)$. When $\lambda>4,$ Raymond observed that
	it is possible to choose a covering of $\Sigma_{a,\lambda}$ elaborately such that one can get rid of these repetitions. Now we  describe  this special covering.
	
	\begin{definition}[\cite{LW,R}]
		For $\lambda>4$ and  $n\ge0$, define three types of bands as:
		
		$(n,\one)$-type band: a band of $\sigma_{(n,1)}(\lambda)$ contained in a
		band of $\sigma_{(n,0)}(\lambda)$;
		
		$(n,\two)$-type band: a band of $\sigma_{(n+1,0)}(\lambda)$ contained
		in a band of $\sigma_{(n,-1)}(\lambda)$;
		
		$(n,\three)$-type band: a band of $\sigma_{(n+1,0)}(\lambda)$ contained
		in a band of $\sigma_{(n,0)}(\lambda)$.
	\end{definition}
	
	All three types of bands actually occur and they are disjoint. These bands are called  {\em spectral generating bands} of order $n$.
For any $n\ge0$, define
	\begin{equation}\label{def-B-n}
			\B_n^a(\lambda):=\{B: B \text{ is a spectral generating band of order } n\}.
	\end{equation}

	The basic covering structure of $\Sigma_{a,\lambda}$ is described in this proposition:
	\begin{proposition}[\cite{LW,R}]\label{basic-struc}
		Fix $a\in \N^\N $ and  $\lambda>4$. Define $\B_n^a(\lambda)$   as above.
		Then
		
		(1) The bands in $\B_n^a(\lambda)$ are disjoint and
		\begin{equation*}\label{diameter-to-0}
			\max\{|B|: B\in \B_n^a(\lambda)\}\to 0, \ \ (n\to\infty).
		\end{equation*}
			
		(2) For any $n\ge0$, we have
		\begin{equation*}
			\sigma_{(n+2,0)}(\lambda)\cup\sigma_{(n+1,0)}(\lambda)\subset
			\bigcup_{B\in\B_n^a(\lambda)}B
			\subset \sigma_{(n+1,0)}(\lambda)\cup\sigma_{(n,0)}(\lambda).
		\end{equation*}
		Consequently, $\{\B_n^a(\lambda):n\ge0\}$ are nested and $\Sigma_{a,\lambda}=\bigcap_{n\ge0}\bigcup_{B\in\B_n^a(\lambda)}B.$
			
		(3) Any $(n,\one)$-type band in $\B_n^a(\lambda)$ contains only one band in $\B_{n+1}^a(\lambda)$, which is of  $(n+1,\two)$-type.

		(4) Any $(n,\two)$-type band in $\B_n^a(\lambda)$ contains  $2a_{n+1}+1$ bands in $\B_{n+1}^a(\lambda)$, $a_{n+1}+1$ of which are of  $(n+1,\one)$-type and $a_{n+1}$ of which are of  $(n+1,\three)$-type. Moreover,   the $(n+1,\one)$-type bands interlace the $(n+1,\three)$-type bands.

		(5) Any $(n,\three)$-type band in $\B_n^a(\lambda)$ contains  $2a_{n+1}-1$ bands in $\B_{n+1}^a(\lambda)$,
		$a_{n+1}$ of which are of  $(n+1,\one)$-type and $a_{n+1}-1$ of which are of  $(n+1,\three)$-type. Moreover,   the $(n+1,\one)$-type bands interlace  the $(n+1,\three)$-type bands.
	\end{proposition}

\begin{figure}
		\includegraphics[scale=0.5]{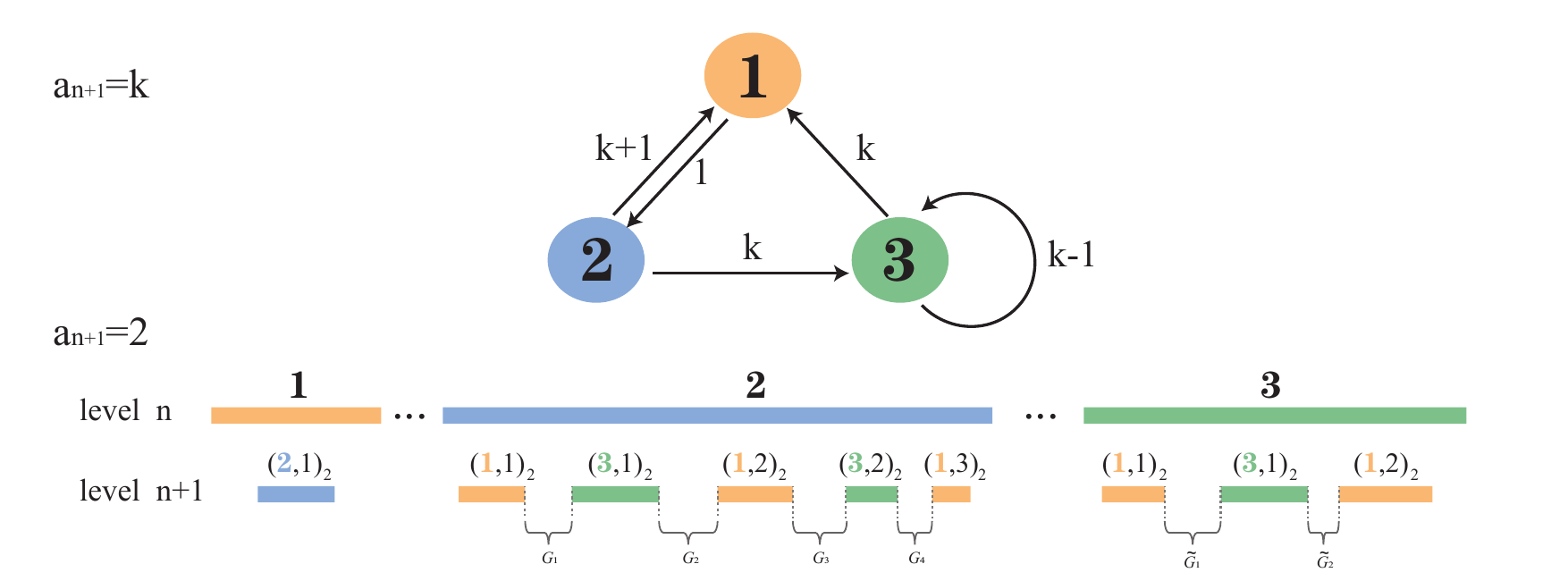}
		
		\caption{The covering structure}\label{type-evo}
 \end{figure}

Thus  $\{\B_n^a(\lambda): n\ge0\}$ form  natural coverings of the spectrum $\Sigma_{a,\lambda}$ (\cite{LPW07,LW05}).
 See Figure \ref{type-evo} for an illustration of Proposition \ref{basic-struc}. The figure above gives the rule of type evolution  and the number statistics of sub-bands. The figure below gives an illustration  how a band of order $n$  with certain type are divided into sub-bands of order $n+1$.

\begin{remark}\label{B-n}
		{\rm
		(i) Note that  there are only two spectral generating bands of order $0$: one is $\sigma_{(0,1)}(\lambda)=[\lambda-2,\lambda+2]$ with  generating polynomial $h_{(0,1)}(E;\lambda)=E-\lambda$ and type $(0,\one)$; the other is  $\sigma_{(1,0)}(\lambda)=[-2,2]$ with generating polynomial  $h_{(1,0)}(E;\lambda)=E$ and type $(0,\three)$. Thus
		\begin{equation*}\label{B-0}
			\B_0^a(\lambda)=\{[\lambda-2,\lambda+2], [-2,2]\}.
		\end{equation*}

		(ii) By Lemma \ref{basic-sturm} and the fact that $v_k=\lambda \mathscr S_k(a),$  we conclude that $\B_n^a(\lambda)$  only depends on $a|_n$.
		As a consequence, we have
		$
		\B_n^{a}(\lambda)=\B_n^b(\lambda)$ if $a|_n=b|_n.
		$
		}
\end{remark}

For later use,  we introduce
\begin{equation*}
  \B_{-1}(\lambda):=\{[-2,\lambda+2]\}.
\end{equation*}

Now we define the gaps of the spectrum. Let ${\rm Co}(\Sigma_{a,\lambda})$ be the convex hull of $\Sigma_{a,\lambda}$. Write
$${\rm Co}(\Sigma_{a,\lambda})\setminus\Sigma_{a,\lambda}=:\bigcup_iG_i(\lambda),$$
where each $G_i(\lambda)$ is an open interval.  $G_i(\lambda)$ is called a {\it gap} of $\Sigma_{a,\lambda}$. A gap $G(\lambda)$ is called {\it  of order  $ n$}, if $G(\lambda)$ is covered by a   band in $\B_n^a(\lambda)$ but not covered by any  band in $\B_{n+1}^a(\lambda)$.

For example, there exists a unique gap of order $-1$, which is contained in $[-2,\lambda+2]$, and contains at least the open interval $(2,\lambda-2)$.

	For any $n\ge-1$, define
	\begin{equation}\label{def-G-n}
\mathcal{G}_n^a(\lambda):=\{G: G \text{ is a gap  of } \Sigma_{a,\lambda}  \text{ of order }  n\}.
	\end{equation}

See Figure \ref{type-evo} for an illustration of some gaps of order $n$ (we remark that, the actual gaps may be bigger than what are indicated in the figure since at later steps the bands may shrink).


\subsection{The symbolic space and the coding of $\Sigma_{a,\lambda}$}\label{sec-coding}\

In the following we describe the coding of  the spectrum $\Sigma_{a,\lambda}$ based on \cite{LQW,Q1,Q2,R}. Here we essentially follow  \cite{Q2}, with some change of notations.

\subsubsection{The symbolic space $\Omega^a$}\label{sec-Omega^a}\

For each $a\in \N^\N$,
Proposition \ref{basic-struc} enables us to construct a symbolic space $\Omega^a$ encoding the spectrum $\Sigma_{a,\lambda}$. This space  is kind of  Markov shift with changing alphabets and incidence matrices. See Figure \ref{type-evo} again to gain some intuition for the following definitions.

Define two  alphabets of types as
\begin{equation*}
	\T:=\{\one,\two,\three\} \ \ \ \text{ and }\ \ \ \T_0:=\{\one,\three\}.
\end{equation*}

For each $n\in\N,$ define an alphabet $\A_n$ as
\begin{equation}\label{alphabet-n}
	\A_n:=\{(\one,k)_n:k=1,\cdots, n+1\}\cup \{(\two,1)_n\}\cup \{(\three,k)_n:k=1,\cdots, n\}.
\end{equation}
Then $\#\A_n=2n+2.$ Assume $e=(\tT,k)_n\in \A_n$, we call $\tT, k,n$ {\it the type, the index, the level} of $e$, respectively. We use the notations:
\begin{equation}\label{type-index-level}
	\tT_e:=\tT;\ \ \ \iI_e:=k; \ \ \  \ell_e:=n.
\end{equation}

Given $\tT\in \T$ and $\hat e\in \A_{n}$, we call $\tT\hat e$  {\it admissible}, denote by $\tT\to \hat e,$ if
\begin{align}\label{admissible-T-A}
	(\tT,\hat e)\in& \{(\one,(\two,1)_n)\}\cup\\
	\nonumber &\{(\two,(\one,k)_n): 1\le k\le n+1\}\cup \{(\two,(\three,k)_n): 1\le k\le n\}\cup\\
	\nonumber&\{(\three,(\one,k)_n):1\le k\le n\}\cup
	\{(\three,(\three,k)_n):1\le k\le n-1\}.
\end{align}
Given $e\in\A_n$ and $\hat e\in \A_m$, we call  $e\hat e$  {\it admissible}, denote by
\begin{equation}\label{admissible-A-A}
	e\to \hat e,  \ \ \text{ if }\ \ \tT_e\to \hat e.
\end{equation}

For pair $(\A_n, \A_m),$ define the {\it incidence matrix} $A_{nm}$   as
\begin{equation}\label{A_ij}
A_{nm}= (a_{e \hat e}), \ \ \text{where }\ \ a_{e\hat e}=\begin{cases}
	1,& \text{ if } e\to \hat e,\\
	0,& \text{ otherwise}.
\end{cases}
\end{equation}

\vspace{1ex}
The family $\{A_{nm}:n,m\ge1\}$ is {\it strongly primitive} in the following sense:
\begin{proposition}[{\cite[Proposition 4.4]{Q2}}]\label{sturm-primitive}
For any $k\ge 6$ and any  $a_1\cdots a_k\in \N^k$, the matrix
$A_{a_1a_2}A_{a_2a_3}\cdots A_{a_{k-1}a_k}$ is positive, i.e. all the entries of the matrix are positive.
\end{proposition}

\begin{remark}
  {\rm
   As a consequence, for   any $e\in\mathscr{A}_{a_1}$ and $\hat{e}\in\mathscr{A}_{a_k}$, there exists a word $w\in \prod_{i=1}^{k}\A_{a_i}$  such that $w_1=e$ and $w_k=\hat{e}$ and $w_i\to w_{i+1}$. In other words, there exists an admissible path  connecting $e$ and $\hat e$. This property will be used several times throughout  this paper. }
\end{remark}

Assume $a=a_1a_2a_3\cdots\in \N^\N$, define the {\it symbolic space}  $\Omega^{a}$ as
\begin{eqnarray}\label{Omega^alpha}
\Omega^{a}:=\left\{x=x_0x_1x_2\cdots \in \T_0\times\prod_{n=1}^\infty \A_{a_n}: x_n\to x_{n+1}, n\ge0 \right\}.
\end{eqnarray}
Intuitively, $\Omega^{a}$ consists of all the admissible infinite paths.

For any $n\ge0$, define the set of {\it admissible words of order $n$} as
\begin{equation}\label{def-Omega^alpha_n}
\Omega^{a}_{n}:=\left\{x|_{n}=x_0x_1\cdots x_n: x\in \Omega^{a}\right\}\ \ \text{ and }\ \  \Omega^{a}_{\ast}:=\bigcup_{n\ge 0} \Omega^{a}_{n}.
\end{equation}

\begin{remark}
{\rm
  (i) Here we warn that the notation $x|_n$ is different from $a|_n$ appeared in Section \ref{sec-continued-fra} in that $x|_n=x_0\cdots x_n$ starts with $0.$

  (ii) From the definitions \eqref{Omega^alpha} and \eqref{def-Omega^alpha_n}, it is seen that if $a,b\in \N^\N$ is such that $a|_n=b|_n$, then $\Omega^a_n=\Omega^b_n$. Thus for $\vec a\in \N^n$, we can choose any $a\in \N^\N$ with $a|_n=\vec a$ and define
  $
    \Omega_n^{\vec a}:=\Omega_n^a.
  $

  (iii) In symbolic dynamics and combinatorics on words, the set $\Omega^a_\ast$ is also known as language or dictionary.

  }
\end{remark}

For any $w=w_0w_1\cdots w_n\in\Omega^a_n$, define
\begin{equation}\label{type}
\tT_{w}:=\tT_{w_n},
\end{equation}
(More generally, if $w\in\prod_{j=1}^{n}\mathscr{A}_{m_j}$, define $\tT_w:=\tT_{w_n}$)
and define the cylinder $[w]^a$ as
\begin{equation*}
[w]^a:=\left\{x\in \Omega^{a}: x|_{n}=w\right\}.
\end{equation*}

\vspace{1ex}
We have the following estimate on the cardinality of $\Omega^a_n$:

\begin{lemma}[{\cite[Lemma 4.10]{Q2}}]\label{number-Omega-n-alpha}
For any $a\in \N^\N $, we have
\begin{equation*}
	q_n(a)\le \#\Omega_n^a\le 5 q_n(a).
\end{equation*}
\end{lemma}

\subsubsection{The coding map $\pi^a_\lambda$ and basic sets}\label{sec-coding-map}\

Assume $a\in \N^\N$ and $\lambda>4$. Now we explain that $\Omega^{a}$ is a coding of the spectrum $\Sigma_{a,\lambda}$.

 From the definition, we see that  $\Omega_0^a=\{\one,\three\}$. Define
\begin{equation}\label{0}
B_{\one}^a(\lambda):=[\lambda-2,\lambda+2]\ \ \text{ and }\ \  B_{\three}^a(\lambda):=[-2,2].
\end{equation}
Then by Remark \ref{B-n}(i), we see that
$\B_0^a(\lambda)=\{B_w^a(\lambda): w\in \Omega_0^{a}\}.
$

We make the following convention:
\begin{equation*}
B_{\emptyset}^a(\lambda):=[-2,\lambda+2].
\end{equation*}
Thus $\B_{-1}(\lambda)=\{B_{\emptyset}^a(\lambda)\}$ and
\begin{equation*}
  B_{\one}^a(\lambda), \ \ B_{\three}^a(\lambda)\subset B_{\emptyset}^a(\lambda).
\end{equation*}

Assume $B_w^a(\lambda)$ is defined for any $w\in  \Omega_{n}^{a}$ and
$$
\B_{n}^a(\lambda)=\{B_w^a(\lambda): w\in \Omega_{n}^{a}\}.
$$
Given $w=w_0w_1\cdots w_{n+1}\in\Omega_{n+1}^{a}$. Write
$w':=w|_{n}$ and $w_{n+1}=(\tT,k)_{a_{n+1}}.$
Define $B_w^a(\lambda)$ to be the unique $k$-th band of $(n+1,\tT)$-type in $\B_{n+1}^a(\lambda)$ which is contained in $B_{w^\prime}^a(\lambda)$. See Figure \ref{type-evo} for an illustration.  By Proposition \ref{basic-struc}(3)-(5) and \eqref{admissible-T-A}, $B_w^a(\lambda)$ is well-defined for every $w\in \Omega_{n+1}^a$ and
\begin{equation*}\label{coding-B-n-a}
\B_{n+1}^a(\lambda)=\{B_w^a(\lambda): w\in \Omega_{n+1}^{a}\}.
\end{equation*}
By induction, we  code all the bands in $\B_n^a(\lambda)$ by $\Omega_n^{a}.$ Now  we can define a natural map $\pi^a_\lambda:\Omega^{a}\to \Sigma_{a,\lambda}$ as
\begin{equation}\label{pi^alpha}
\pi^a_\lambda(x):=\bigcap_{n\ge 0} B_{x|_{n}}^a(\lambda).
\end{equation}
By Proposition \ref{basic-struc}, it is seen that $\pi^a_\lambda$ is a bijection, thus  $\Omega^{a}$  is a {\it coding} of $\Sigma_{a,\lambda}$.

For any $w\in \Omega^a_n,$ define the {\it  basic set} $X_w^a(\lambda)$ of $\Sigma_{a,\lambda}$ of order $n$ as
\begin{equation}\label{def-basic-set}
X_w^a(\lambda):=\pi^a_\lambda([w]^a)=B_w^a(\lambda)\cap\Sigma_{a,\lambda}.
\end{equation}
Then, the spectrum $\Sigma_{a,\lambda}$ is the union of disjoint basic sets of order $n$:
\begin{equation}\label{spectrum-basic-set}
\Sigma_{a,\lambda}=\bigsqcup_{w\in\Omega^a_n}X_{w}^a(\lambda).
\end{equation}

\subsection{Useful facts for Sturmian Hamiltonians}\

Here are  some known results regarding  Sturmian Hamiltonians, which we will frequently refer to later.

The following lemma (\cite[Lemma 3.7]{LQW}) provides  estimates for the length of the spectral generating band:

\begin{lemma}[\cite{LQW}]{\label{esti-band-length}}
Assume $a\in \N^\N$ and $\lambda\geq20$. Write $\tau_1:=(\lambda-8)/3$ and $\tau_2:=2(\lambda+5)$. Then for any $w=w_0\cdot\cdot\cdot w_n\in\Omega_n^{a}$,
 we have
\begin{equation}\label{band-length-esti}
	\tau_2^{-n}\prod_{i=1}^na_i^{-3}\left(\prod_{1\le i\le n;\tT_{w_i}= \two}\tau_2^{2-a_i}\right)\leq|B_w^a(\lambda)|\leq4\tau_1^{-n}\left(\prod_{1\le i\le n;\tT_{w_i}= \two}\tau_1^{2-a_i}\right).
\end{equation}
In particular, we have $|B_w^a(\lambda)|\leq2^{2-n}.$
\end{lemma}


Fix any $a\in \N^\N $ and $\lambda\ge24$. Define the pre-dimensions of $\Sigma_{a,\lambda}$ as
\begin{equation}\label{pre-dim}
s_\ast(a,\lambda):=\liminf_{n\to\infty} s_n(a,\lambda);\ \ \ s^\ast(a,\lambda):=\limsup_{n\to\infty} s_n(a,\lambda),
\end{equation}
where $s_n(a,\lambda)$ is the unique number such that
\begin{equation}\label{def-s-n}
\sum_{w\in \Omega^a_n}|B_w^a(\lambda)|^{s_n(a,\lambda)}=1.
\end{equation}

The following proposition (\cite[Theorem 1.1 and Proposition 8.2]{LQW}) provides the dimension formulas of the spectrum and the regularities of the dimension functions.

\begin{proposition}[\cite{LQW}]\label{dim-formula-14}
Assume $a\in \N^\N $, then

(1)  For any $\lambda\ge24$, we have
$\dim_H \Sigma_{a,\lambda}=s_\ast(a,\lambda)$ and \ $\overline{\dim}_B \Sigma_{a,\lambda}=s^\ast(a,\lambda).$

(2) There exists an absolute constant $C>0$ such that for any $24\leq\lambda_1<\lambda_2$,
\begin{equation*}
	|s_*(a,\lambda_1)-s_*(a,\lambda_2)|,\ |s^*(a,\lambda_1)-s^*(a,\lambda_2)|\leq C\lambda_1|\lambda_1-\lambda_2|.
\end{equation*}
\end{proposition}

The following proposition is a restatement of   \cite[Proposition 8.1]{LQW}:
\begin{proposition}[\cite{LQW}]\label{bco-cor}
Assume $a,b\in \N^\N$ and $\lambda_1,\lambda_2\ge 24$. Then there exist absolute constants $C_1,C_2,C_3>1$ such that if $w, wu\in  \Omega_\ast^{a}$ and $\tilde w, \tilde wu\in  \Omega_\ast^{b}$,
then
\begin{equation*}\label{strong-bco}
	\eta^{-1} \frac{|{B}_{wu}^a(\lambda_1)|}{|{B}_{w}^a(\lambda_1)|}\le
	\frac{|B_{\tilde wu}^{b}(\lambda_2)|}{|B_{\tilde w}^{b}(\lambda_2)|}\le \eta
	\frac{|{B}_{wu}^a(\lambda_1)|}{|{B}_{w}^a(\lambda_1)|},
\end{equation*}
with $\eta$ defined by
\begin{equation}\label{def-eta}
	\eta:=C_1\exp(C_2(\lambda_1+\lambda_2)+C_3m(u)\lambda_1|\lambda_1-\lambda_2|),
\end{equation}
where $u=u_1\cdots u_n$ with $u_k=(\tT_k,i_k)_{n_k}$ and
\begin{equation}\label{def-m-u}
	m(u):=\sum_{k:\tT_k=\two} (n_k-1)+\sum_{k:\tT_k\neq\two}1=n+\sum_{k: \tT_k=\two} (n_k-2).
\end{equation}
\end{proposition}
\begin{remark}
{\rm
In \cite[Proposition 8.1]{LQW},  only  the case $a=b$ is considered. However by checking the proof, one  can indeed show  the  stronger result as stated in Proposition \ref{bco-cor}. The original proposition is stated for the modified ladders (see \cite[Section 6.1]{LQW} for the definition). Since we do not need this technical notation, we just extract the content we need, which is the present form.}
\end{remark}
Take $\lambda_1=\lambda_2$, we obtain the following consequence, which is   \cite[Theorem 3.3]{LQW}:
\begin{proposition}[Bounded covariation \cite{LQW}]\label{bco}
Assume $a,b\in \N^\N$ and $\lambda\ge 24$. Then there exists constant $\eta=\eta(\lambda)>1$ such that if $w, wu\in  \Omega_\ast^{a}$ and $\tilde w, \tilde wu\in  \Omega_\ast^{b}$,
then
$$\eta^{-1} \frac{|{B}_{wu}^a(\lambda)|}{|{B}_{w}^a(\lambda)|}\le
\frac{|B_{\tilde wu}^{b}(\lambda)|}{|B_{\tilde w}^{b}(\lambda)|}\le \eta
\frac{|{B}_{wu}^a(\lambda)|}{|{B}_{w}^a(\lambda)|}.$$
\end{proposition}

The following lemma is implicitly proven in \cite{R}:
\begin{lemma}[\cite{R}]{\label{sigma-n+1-0}}
For any $a\in \N^\N,\lambda>4$ and  $n\ge0$, we have
$$
\sigma_{n+1,0}(\lambda)=\bigcup_{w\in\Omega^a_n:\  \tT_w=\two,\three}B_w^a(\lambda).
$$
\end{lemma}

The following proposition is \cite[Theorem 3.1 and Corollary 3.2]{LQW}:

\begin{proposition}[Bounded variation and Bounded distortion \cite{LQW}]\label{bv}
Assume  $a\in \N^\N$ and $\lambda\geq20$.  Then there exists a constant $C(\lambda)>1$ such that for any $w\in \Omega^a_\ast$ and $ E,E_1,E_2\in B_w^a(\lambda)$,
\begin{equation*}
	\frac{1}{C(\lambda)}\le \left|\frac{h_w'(E_1;\lambda)}{h_w'(E_2;\lambda)}\right|\le C(\lambda)\ \text{and}\ \ \frac{1}{C(\lambda)}\le |h_w'(E;\lambda)||B_w^a(\lambda)|\le C(\lambda),
\end{equation*}
where $h_w(E;\lambda):=h_{B^a_w(\lambda)}(E;\lambda)$ is the generating polynomial of $B_w^a(\lambda)$.
\end{proposition}

\section{Almost sure dimensional properties of the spectrum }\label{sec-proof-main-spec}

In this section, we prove Theorem \ref{main-dim-spectra}. We follow the plan given in Section \ref{sec-idea}. At first, we represent the pre-dimensions as the ``zeros" of two pressure functions. Next, we use ergodic theorem to show that the two pressure functions coincide and derive the Bowen's type formula for the dimension. Then we establish the regularity of the dimension function and prepare for the asymptotic behavior of it. At  last, we prove Theorem \ref{main-dim-spectra}.

\subsection{Pressure functions and generalized Bowen's formulas}\

Follow the explanation in Section \ref{sec-idea}, we introduce the following function $\QQ$, which is the logarithm of the partition function $\mathscr P_{\alpha,\lambda}$ defined in \eqref{partition}.

For any $a\in \N^\N, \lambda\in [24,\infty),t\in \R, n\in \N$, define
\begin{equation}\label{def-Q}
  \QQ(a,\lambda,t,n):=\log\sum_{w\in \Omega_{n}^a}|B^{a}_{w}(\lambda)|^t.
\end{equation}

Fix $a\in \N^\N$ and $\lambda\in [24,\infty)$. Define the {\it upper and lower pressure functions} as
\begin{equation}\label{def-lu-pre}
  \overline{\rm P}_{a,\lambda}(t):=\limsup_{n\to\infty} \frac{1}{n}\QQ(a,\lambda,t,n);\ \ \
  \underline {\rm P}_{a,\lambda}(t):=\liminf_{n\to\infty} \frac{1}{n}\QQ(a,\lambda,t,n).
\end{equation}

If  $a\in \F_1$, then these two pressure functions behave well:

\begin{lemma}\label{pressure-lip}
Assume $(a,\lambda)\in \F_1\times [24,\infty)$. Then

(1) For $t=0$, we have
\begin{equation*}
\overline {\rm P}_{a,\lambda}(0)=\underline {\rm P}_{a,\lambda}(0)=\gamma.
\end{equation*}

(2) For any $t\ge0$, we have
\begin{equation*}
\overline {\rm P}_{a,\lambda}(t), \ \ \underline {\rm P}_{a,\lambda}(t)\ge-t\log8\kappa\lambda.
\end{equation*}

(3) For any $0\le t_1<t_2$, we have
\begin{equation*}
\overline {\rm P}_{a,\lambda}(t_2)-\overline {\rm P}_{a,\lambda}(t_1),  \ \  \underline {\rm P}_{a,\lambda}(t_2)-\underline {\rm P}_{a,\lambda}(t_1)\le (t_1-t_2)\log 2.
\end{equation*}
In particular, both $\overline {\rm P}_{a,\lambda}(t)$ and $ \underline {\rm P}_{a,\lambda}(t)$ are strictly decreasing on $\R^+$.

(4) $\overline {\rm P}_{a,\lambda}(t)$ is convex on $\R^+$. Consequently, it is continuous on $(0,\infty)$.
\end{lemma}

\begin{proof}
(1) Since $a\in \F_1$,  the equalities follow from Lemma \ref{number-Omega-n-alpha} and Proposition \ref{F-proposition}(1).

(2) By \cite[Proposition 3.4]{DGLQ},  for $t\ge0$,  we have
\begin{equation*}
\sum_{w\in \Omega_{n}^a}|B^a_w(\lambda)|^t\geq|B^{{a}}_{n,\max}(\lambda)|^t
\geq\left(8^{n}\lambda^{n}\left(a_1a_2\cdots a_n\right)\right)^{-t},
\end{equation*}
where $B^{{a}}_{n,\max}(\lambda)$ is the band in $\B_n^a(\lambda)$ of maximal length.
By this, the inequalities   follow from the definition \eqref{def-lu-pre} and Proposition \ref{F-proposition}(1).

(3) By Lemma \ref{esti-band-length}, for any $0\le t_1<t_2$ and $n\in\N$, we have
\begin{align*}
\sum_{w\in \Omega_{n}^a}|B^a_w(\lambda)|^{t_2}
\le2^{(2-n)(t_2-t_1)}\sum_{w\in \Omega_{n}^a}|B^a_w(\lambda)|^{t_1}.
\end{align*}
This implies the two inequalities. By the two inequalities, we conclude that  $\overline {\rm P}_{a,\lambda}(t)$ and $\underline {\rm P}_{a,\lambda}(t)$ are strictly decreasing on $\R^+$.

(4) For any $t,s\ge0$ and $\beta\in (0,1)$, by H\"older inequality, we have
\begin{equation*}\label{convex}
\sum_{w\in \Omega_{n}^a}|B^a_w(\lambda)|^{\beta t+(1-\beta)s}\le \left(\sum_{w\in \Omega_{n}^a}|B^a_w(\lambda)|^t\right)^{\beta}\left(\sum_{w\in \Omega_{n}^a}|B^a_w(\lambda)|^s\right)^{1-\beta}.
\end{equation*}
Together with the definition \eqref{def-lu-pre}, we conclude that
$$
\overline {\rm P}_{a,\lambda}(\beta t+(1-\beta)s)\le \beta \overline {\rm P}_{a,\lambda}(t)+(1-\beta)\overline {\rm P}_{a,\lambda}(s).
$$
So $\overline {\rm P}_{a,\lambda}(t)$ is convex on $\R^+$. Consequently, it is continuous on $(0,\infty)$.
\end{proof}

\begin{remark}
 {\rm
One can show that  for $\Gg$-a.e. $a\in \F_1$ and  all $t<0$,
\begin{equation*}
\overline {\rm P}_{a,\lambda}(t)=\underline {\rm P}_{a,\lambda}(t)=\infty.
\end{equation*}
So for $t<0$, the pressure functions do not contain much information. Since we do not need this fact, we will not prove it.
 }
\end{remark}

Fix $(a,\lambda)\in \F_1\times [24,\infty)$. Define
\begin{equation}\label{def-d-D}
\underline D(a,\lambda):=\sup\{t\ge0: \underline {\rm P}_{a,\lambda}(t)\ge0\};\ \ \ \overline D(a,\lambda):=\sup\{t\ge0: \overline {\rm P}_{a,\lambda}(t)\ge0\}.
\end{equation}
If $\underline {\rm P}_{a,\lambda}(t)$ and $\overline {\rm P}_{a,\lambda}(t)$ are continuous, then  $\underline D(a,\lambda)$ and $\overline D(a,\lambda)$ are the zeros of them.

By Lemma \ref{pressure-lip}(1) and (3), we have
\begin{equation*}
\underline D(a,\lambda), \ \ \overline D(a,\lambda)\le \frac{\gamma}{\log2}.
\end{equation*}

Now we have the following generalized Bowen's formulas:
\begin{proposition}\label{Bowen-formula}
Assume $(a,\lambda)\in \F_1\times [24,\infty)$, then
\begin{equation*}
\dim_H \Sigma_{a,\lambda}=\underline{D}(a,\lambda)\in (0,1) \ \ \ \text{and } \ \ \ \overline\dim_B\Sigma_{a,\lambda}=\overline{D}(a,\lambda)\in (0,1).
\end{equation*}
Moreover, $\overline{D}(a,\lambda)$ is the unique zero of $\overline P_{a,\lambda}(t)$ on $\R^+.$
\end{proposition}

\begin{proof}
By Proposition \ref{dim-formula-14}(1), we only need to show:
\begin{equation*}
\underline{D}(a,\lambda)=s_\ast(a,\lambda) \ \ \ \text{and }\ \ \ \overline{D}(a,\lambda)=s^\ast(a,\lambda).
\end{equation*}

At first assume $\tau>\underline{D}(a,\lambda)$. By Lemma \ref{pressure-lip}(3), we have $\underline{\rm P}_{a,\lambda}(\tau)<0$. Fix $\epsilon_0\in (0,-\underline{\rm P}_{a,\lambda}(\tau))$, then there exist $n_k\uparrow\infty$ such that
\begin{equation*}
\sum_{w\in \Omega_{n_k}^a}|B^a_w(\lambda)|^\tau\le e^{-n_k\epsilon_0}.
\end{equation*}
By \eqref{def-s-n}, for $k$ large enough, $s_{n_k}(a,\lambda)\le \tau$. Hence by  \eqref{pre-dim}, we have
$$
s_\ast(a,\lambda)\le\liminf_{k\to\infty} s_{n_k}(a,\lambda)\le\tau.
$$
Since $\tau>\underline{D}(a,\lambda)$ is arbitrary, we conclude that $s_\ast(a,\lambda)\le\underline{D}(a,\lambda)$.		

Next assume $t>s_\ast(a,\lambda)$. By  \eqref{pre-dim}, there exists $n_k\uparrow\infty$ such that $s_{n_k}(a,\lambda)<t$. Hence
\begin{align*}
\underline{\rm P}_{a,\lambda}(t) =&  \liminf_{n\to\infty} \frac{1}{n}\log\sum_{u\in \Omega^{{a}}_{n}}|B^{{a}}_u(\lambda)|^t
\le\liminf_{k\to\infty} \frac{1}{n_k}\log\sum_{w\in\Omega^{{a}}_{n_k}}|B^{{a}}_w(\lambda)|^{t}\\
\leq&\liminf_{k\to\infty} \frac{1}{n_k}\log\sum_{w\in\Omega^{{a}}_{n_k}}|B^{{a}}_w(\lambda)|^{s_{n_k}(a,\lambda)}=0.
\end{align*}
Since $\underline{\rm P}_{a,\lambda}(t)$ is strictly decreasing, we conclude that $\underline{D}(a,\lambda)\le t$. Since  $t>s_\ast(a,\lambda)$ is arbitrary, we conclude that
$\underline{D}(a,\lambda)\le s_\ast(a,\lambda).$

As a result, $s_\ast(a,\lambda)=\underline{D}(a,\lambda)$. By essentially the same proof, we have $s^\ast(a,\lambda)=\overline{D}(a,\lambda)$.

Since $a\in \F_1$, we have $K_\ast(a)=K^\ast(a)=\kappa$. By Theorem B, $\underline{D}(a,\lambda),\overline{D}(a,\lambda)\in (0,1)$. Since $\overline{\rm P}_{a,\lambda}(t)$ is continuous and strictly decreasing on $(0,\infty)$, by the definition of $\overline{D}(a,\lambda)$, we conclude that it is the unique zero of $\overline{\rm P}_{a,\lambda}(t)$.
\end{proof}

With this proposition in hand, to show the coincidence of Hausdorff  and upper-box dimensions of the spectrum, it is sufficient to show that two pressure functions are equal.

\subsection{Relativized  pressure function and dimension formula}\

In the following, we will establish the coincidence of the upper and lower pressure functions for typical $a\in \N^\N$ via Kingman's ergodic theorem.
The key observation is the following:

\begin{lemma}\label{additive-on-A}
Assume $t\ge0$ and $\lambda\ge 24$. Then $\QQ^+(\cdot,\lambda,t,1)\in L^1(\Gg)$, and the family of functions $\{\QQ(\cdot,\lambda,t,n): n\in \N\}$ are almost sub-additive on $(\mathbb{N}^{\mathbb{N}},S,\Gg)$. That is, there exists a constant $c>0$ such that
$$
\QQ(a,\lambda,t,n+m)\leq\QQ(a,\lambda,t,n)+\QQ(S^na,\lambda,t,m)+c,\ \ \forall a\in \N^\N; n,m\in \N.
$$
\end{lemma}

\begin{remark}
{\rm By a bit more effort, one can show that $\QQ(\cdot,\lambda,t,1)\in L^1(\Gg)$.}
\end{remark}

The proof is a bit technical. We will give it in Section \ref{sec-proof-sub-additive}.

By Lemma \ref{additive-on-A}, for any $t\ge0$ and  $\lambda\ge24$, the following limit exists:
\begin{equation}\label{relative-pre}
\mathbf{P}_{\lambda}(t):=\lim_{n\to\infty}\frac{1}{n}\int_{\N^\N} \QQ(a,\lambda,t,n) d\Gg(a).
\end{equation}
We call $\mathbf{P}_{\lambda}(t)$ the {\it relativized pressure function}.

Now we are ready for the proof of the coincidence of two pressure functions.

\begin{proposition}\label{as-pressure}
Fix $\lambda\ge24$. Then

(1) There exists a set $\F(\lambda)\subset \F_1$ with $\Gg(\F(\lambda))=1$ such that for any $a\in \F(\lambda)$,
\begin{equation*}
\overline{\rm P}_{a,\lambda}(t)=\underline{\rm P}_{a,\lambda}(t)=\mathbf{P}_{\lambda}(t),\ \ \forall\ t\in\R^+.
\end{equation*}

(2) $\mathbf{P}_{\lambda}(t)$ is convex and strictly decreasing  on $\R^+$, continuous on $(0,\infty)$ and satisfies
\begin{equation*}
-t\log8\kappa\lambda\leq \mathbf{P}_{\lambda}(t)\leq \gamma-t\log2.
\end{equation*}

(3) $\mathbf{P}_{\lambda}(t)$ has a unique zero  on $(0,\infty)$, denoted by $D(\lambda)$. Moreover, $D(\lambda)\in(0,1)$.

\end{proposition}

\begin{proof}
(1)
Fix any $t\in \R^+$. Since $\Gg$ is ergodic, by Lemma \ref{additive-on-A} and  Kingman's ergodic theorem, there exists a set $\F_{t,\lambda}\subset \F_1$ with $\Gg(\F_{t,\lambda})=1$ such that for any $a\in \F_{t,\lambda}$,
$$
\lim_{n\to\infty}\frac{1}{n} \QQ(a,\lambda,t,n)=\mathbf{P}_{\lambda}(t).
$$
Combine with \eqref{def-lu-pre}, we conclude that for any $a\in  \F_{t,\lambda}$,
\begin{equation}\label{coincide-single}
\overline{\rm P}_{a,\lambda}(t)=\underline{\rm P}_{a,\lambda}(t)=\mathbf{P}_{\lambda}(t).
\end{equation}

Now, define the $\Gg$-full measure set $\F(\lambda)\subset\F_1$ as
\begin{equation} \label{def-F-lambda}
\F(\lambda):=\bigcap_{t\in \Q^+}\F_{t,\lambda}.
\end{equation}

Fix $a\in \F(\lambda)$. Then $a\subset\F_{t,\lambda}$ for all $t\in\Q^+$. By \eqref{coincide-single},
\begin{equation}\label{equal-Q}
\overline{\rm P}_{a,\lambda}(t)=\underline{\rm P}_{a,\lambda}(t)=\mathbf{P}_{\lambda}(t),\ \ \ \forall\ t\in \Q^+.
\end{equation}

By \eqref{def-Q}, $\QQ(a,\lambda,t,n)$ is decreasing for $t$ on $\R^+$, thus $\mathbf{P}_{\lambda}(t)$ is decreasing on $\R^+$.
By Lemma \ref{pressure-lip}, both $  \overline{\rm P}_{a,\lambda}(t)$ and $ \underline{\rm P}_{a,\lambda}(t)$ are strictly  decreasing on $\R^+$, and $  \overline{\rm P}_{a,\lambda}(t)$ is continuous on $(0,\infty)$.  Combine with \eqref{equal-Q}, we conclude that the equation holds.

(2) The statements follow from (1) and  Lemma \ref{pressure-lip}.

(3) It follows  from (1) and  Proposition \ref{Bowen-formula}.
\end{proof}

\begin{remark}
{\rm Indeed, one can show that $\mathbf{P}_{\lambda}(t)$ is $C^1$ and  strictly convex on $\R^+$. However, the proof is far from trivial. We will present  it  in \cite{CQ}.}
\end{remark}

Now we have  the following preliminary result for the dimensions of the spectrum.	

\begin{proposition}\label{dim-fix-lambda-prop}
Fix $\lambda\in [24,\infty)$. Let $\F(\lambda)$ be the $\Gg$-full measure set in Proposition \ref{as-pressure}. Then
for any $a\in \F(\lambda),$ we have
\begin{equation*}
\dim_H\Sigma_{a,\lambda} =\dim_B\Sigma_{a,\lambda}=D(\lambda).
\end{equation*}
\end{proposition}

\begin{proof}
Fix $\lambda\geq24$. By Propositions \ref{Bowen-formula}, \ref{as-pressure} and \eqref{def-d-D}, for any $a\in \F(\lambda)$, we have
$$
\dim_H \Sigma_{a,\lambda}=\overline\dim_B\Sigma_{a,\lambda}=D(\lambda).
$$
Since
$
\dim_H \Sigma_{a,\lambda}\le \underline\dim_B \Sigma_{a,\lambda}\le\overline\dim_B\Sigma_{a,\lambda}
$ always holds,
we conclude that $\dim_B \Sigma_{a,\lambda}=D(\lambda)$ and hence the result follows.
\end{proof}

\subsection{Lipschitz property and asymptotic property of $D(\lambda)$}\label{sec-lip-asym-D}\

Building on the previous  regularity properties of the dimension functions for deterministic frequencies (\cite{LQW}), we can show  the following regularity property of $D(\lambda)$:
\begin{proposition}\label{prop-lip}
There exists an absolute constant $C>0$ such that for any $24\le\lambda_1<\lambda_2<\infty$, we have
\begin{equation*}
|D(\lambda_1)-D(\lambda_2)|\le C\lambda_1|\lambda_1-\lambda_2|.
\end{equation*}
\end{proposition}

\begin{proof}
Fix any $\lambda_1, \lambda_2 \in [24, \infty)$. Let $\F(\lambda_i), i=1,2$ be the $\Gg$-full measure sets in Proposition \ref{as-pressure}. Then $\Gg(\F(\lambda_1) \cap \F(\lambda_2)) = 1$. By Proposition \ref{dim-fix-lambda-prop} and Proposition \ref{dim-formula-14}(1), for any $a \in \F(\lambda_1) \cap \F(\lambda_2)$, we have
$$
s_\ast(a, \lambda_1) = D(\lambda_1);\ \ \ \ \  s_\ast(a, \lambda_2) = D(\lambda_2).
$$
The result then follows directly from Proposition \ref{dim-formula-14}(2).
\end{proof}

As a consequence, we can strengthen  Proposition \ref{dim-fix-lambda-prop} by choosing a $\lambda$-independent $\Gg$-full measure set such that the result  holds.

\begin{corollary}\label{coincide-whole}
There exists a subset $\F_2\subset \F_1$ with $\Gg(\F_2)=1$ such that for any $(a,\lambda)\in \F_2\times[24,\infty)$, we have
\begin{equation*}
\dim_H\Sigma_{a,\lambda} =\dim_B\Sigma_{a,\lambda}=D(\lambda).
\end{equation*}
\end{corollary}
\begin{proof} For any $\lambda\ge 24,$ let $\F(\lambda)$ be the $\Gg$-full measure set in Proposition \ref{as-pressure}.
Define
\begin{equation}\label{def-tilde-F}
\F_2:=\bigcap_{\lambda\in [24,\infty)\cap \Q}\F(\lambda).
\end{equation}
Then $\Gg(\F_2)=1$ and $\F_2\subset \F_1$. Fix any $a\in \F_2,$ by  Proposition \ref{dim-fix-lambda-prop},  the equalities hold for any $\lambda\in [24,\infty)\cap \Q$. By Proposition \ref{prop-lip}, $D(\cdot)$ is continuous on $[24,\infty)$. By Proposition \ref{dim-formula-14}(2), both $s_\ast(a,\cdot)$ and $s^\ast(a,\cdot)$ are continuous on $[24,\infty)$. Thus we conclude that the equalities hold for any $\lambda\in [24,\infty).$
\end{proof}

Next we prepare for the proof of  the asymptotic property of $D(\lambda)$. It follows from that of  $s_\ast(\alpha,\lambda)$ established in \cite{LQW} and Kingman's ergodic theorem. As we will see, the Lyapunov exponents of certain family of matrix-valued functions will play an essential role in the proof.

Inspired by \cite{LQW}, for any $k\in\mathbb{N}$ and $x\in [0,1]$, we define a matrix
$$R_k(x):=\begin{bmatrix}
0 & x^{k-1} & 0 \\
(k+1)x & 0 & kx \\
kx & 0 & (k-1)x
\end{bmatrix}.
$$
It is seen that
$$
\|R_k(x)\|_\infty\le k+1, \ \ \text{where }\  \  \|A\|_\infty:=\max\{|a_{ij}|:1\le i,j\le3\}.
$$

For any $x\in [0,1]$, define a map $M_x:\N^\N\to M_3(\R)$ and related cocycle as
$$
M_x(a):=R_{a_1}(x); \ \  M_x(n,a):=M_x(a_1)\cdots M_x(a_n).
$$
It is seen that the following cocycle relation holds:
\begin{equation}\label{cocycle}
M_x(n+m,a)=M_x(n,a)M_x(m,S^na).
\end{equation}
Let $\|\cdot\|$ be the operator norm on $M_3(\R)$.
Since $\|\cdot\|$ and $ \|\cdot\|_\infty$ are equivalent, we have
\begin{equation*}
\int_{\N^\N}(\log \|M_x(a)\|)^+d \Gg(a)\lesssim \sum_{k\in \N}\frac{\log (k+1)}{k^2}<\infty.
\end{equation*}
By \eqref{cocycle}, the sequence of functions
$\{\log  \|M_x(n,\cdot)\|:n\in \N\}$ are sub-additive. So the following function $\varphi:[0,1]\to [-\infty,\infty)$ can be defined:
\begin{equation}\label{def-var-phi}
\varphi(x):=\lim_{n\to\infty} \frac{1}{n}\int_{\N^\N}\log \|M_x(n,a)\|d \Gg(a).
\end{equation}
Indeed, $\varphi(x)$ is the Lyapunov exponent of $M_x$. We will see soon that $\varphi(x)$ is strictly increasing and its ``zero" characterizes the asymptotic behavior of $D(\lambda)$.

Follow \cite{LQW}, for any  $a\in \N^\N$ and $x\in [0,1]$,  we define
$$
\psi(a,x):=\liminf_{n\to\infty}\|M_x(n,a)\|^{1/n};\ \ \ \phi(a,x):=\limsup_{n\to\infty}\|M_x(n,a)\|^{1/n}.
$$

We have the following:

\begin{lemma}\label{lyapunov}
There exists $\mathbb B\subset \N^\N$ with $\Gg(\mathbb B)=1$ such that for any $(a,x)\in \mathbb B\times [0,1]$,
\begin{equation}\label{coincide-phi-psi}
\psi(a,x)=\phi(a,x)=\exp(\varphi(x)).
\end{equation}
Consequently, $\varphi:[0,1]\to [-\infty,\infty)$ is strictly increasing and satisfies
\begin{equation}\label{bd-varphi}
\frac{\kappa x}{2}\le \exp(\varphi(x))\le \kappa\sqrt{2x}.
\end{equation}
\end{lemma}

\begin{proof}
Since $\{\log \|M_x(n,\cdot)\|:n\in \N\}$ are sub-additive and $(\log \|M_x(\cdot)\|)^+$ is integrable,  by Kingman's ergodic theorem, for any  $x\in [0,1]$, there exists $\mathbb B_x\subset \N^\N$ with $\Gg(\mathbb B_x)=1$ such that  for any $a\in \mathbb B_x$,
\begin{equation}\label{def-varphi-x}
\frac{\log \|M_x(n,a)\|}{n}\to \varphi(x).
\end{equation}

For any $n\in \N$, define $\hat \xi_n: \N^\N\times [0,1]\to \R\cup \{-\infty\}$ as
$$
\hat \xi_n(a,x):=\log \|M_x(n,a)\|_\infty.
$$
Since $\|\cdot\|$ and $\|\cdot\|_\infty$ are equivalent, we conclude that
$$
\varphi(x)=\lim_{n\to\infty} \frac{1}{n}\int_{\N^\N}\hat \xi_n(a,x)d \Gg(a).
$$
It is seen that $\hat \xi_n(a,x)$ is increasing as a function of $x$.
Consequently,  $\varphi(x)$ is increasing. Let $D_{\varphi}$ be the set of discontinuous points of $\varphi$, then $D_{\varphi}$ is countable.
Define $\mathbb B\subset\N^\N$ as
$$
\mathbb B:=\left(\bigcap_{x\in [0,1]\cap \Q}\mathbb B_x\right) \cap \left(\bigcap_{x\in D_{\varphi}} \mathbb B_x\right)\cap \F_1,
$$
Then $\Gg(\mathbb B)=1$. Take any $a\in\mathbb B$, then $a\in \F_1$ and hence
$$
K_\ast(a)=K^\ast(a)=\kappa.
$$
By \eqref{def-varphi-x}, we have
\begin{equation}\label{coincide-phi-psi-1}
\psi(a,x)=\phi(a,x)=\exp(\varphi(x)), \ \ \forall x\in ([0,1]\cap \Q)\cup D_{\varphi}.
\end{equation}
By \cite[Lemma 5.1]{LQW}, $\psi(a,x)$ and $\phi(a,x)$ are strictly increasing and satisfy
\begin{equation}\label{bd-phi}
\left(\frac{\kappa }{2}\vee 2^{1/3}\right)x\le \psi(a,x),\ \ \phi(a,x)\le \kappa\sqrt{2x}.
\end{equation}
Notice that by \eqref{L-K},  $2^{1/3}=1.25\cdots<1.34\cdots=\kappa/2.$ By using the continuity of $\varphi$ and monotonicity of three functions, for any $x\in [0,1]\setminus D_{\varphi}$, we have
$$
\psi(a,x)=\phi(a,x)=\exp(\varphi(x)).
$$
Together with \eqref{coincide-phi-psi-1}, we get \eqref{coincide-phi-psi}. Thus $\varphi$ is also strictly increasing, and by \eqref{bd-phi}, we get \eqref{bd-varphi}.
\end{proof}

\subsection{Proof of Theorem \ref{main-dim-spectra}}\

\begin{proof}[Proof of Theorem \ref{main-dim-spectra}]
Define $D$ by Proposition \ref{as-pressure}(3).
Denote
\begin{equation}\label{def-tilde-I}
\tilde \II:=\Theta^{-1}(\F_2),
\end{equation}
where $\F_2$ is defined by \eqref{def-tilde-F}. Since $\F_2$ is of full $\Gg$-measure, we conclude that $\tilde \II$ is of full $G$-measure, hence full Lebesgue measure. We will show that Theorem \ref{main-dim-spectra} holds for any $\alpha\in\tilde \II$. Later in Section \ref{sec-geometric-lemma}, we will choose the final full Lebesgue measure subset $\hat \II\subset \tilde \II$ (see \eqref{def-hat-I}). Then obviously Theorem \ref{main-dim-spectra} holds for any $\alpha\in\hat \II$.

(1) By  Corollary \ref{coincide-whole} and $\alpha=\Theta^{-1}(a)$, the statement holds.

(2) It is the content of Proposition \ref{as-pressure}(3).

(3) By Proposition \ref{prop-lip}, $D$ is Lipschitz continuous on any bounded interval of $[24,\infty).$
Define  the constant $\rho$ as
\begin{equation}\label{def-rho}
\rho^{-1}:=\sup\{x\in [0,1]: \varphi(x)\le0\}.
\end{equation}
By \eqref{bd-varphi} and \eqref{L-K}, we have
\begin{equation*}
0<\frac{1}{2\kappa^2}\le \frac{1}{\rho}\le\frac{2}{\kappa}<1.
\end{equation*}

Take any $a\in \mathbb B\cap\F_2$. By Proposition \ref{F-proposition}(1), $K_\ast(a)=\kappa.$  Fix any $\lambda\in [24,\infty)$. By Corollary \ref{coincide-whole} and Proposition \ref{dim-formula-14}(1),  $
s_\ast(a,\lambda)=\dim_H\Sigma_{a,\lambda}=D(\lambda).
$ By Lemma \ref{lyapunov}, the equality \eqref{coincide-phi-psi} holds.
Now by \cite[Proposition 5.3]{LQW}, we have
\begin{equation*}
\frac{\log \rho}{6\log 4\kappa^2+\log 2(\lambda+5)}\le D(\lambda)\le \frac{\log \rho}{\log (\lambda-8)/3}.
\end{equation*}
This obviously implies \eqref{dim-spectra-asym}.
\end{proof}

\section{Almost sure  dimensional properties of the fiber measures}\label{sec-dim-fiber-symbolic}

In this section, we prove the symbolic version of Theorem \ref{main-dim-dos}. At first, we construct the global symbolic space $\Omega$ and study its basic properties. Next we introduce the entropic potential $\Phi$, and study its Gibbs measure $\n$ and the fiber $\n_a$ of $\n$. Note that $\n_a$ is the symbolic counterpart of the DOS $\NN_{a,\lambda}$. Then we introduce the geometric potential $\Psi_\lambda$ and show its relation with Lyapunov exponent. At last, we  derive the almost sure  dimensional properties of $\n_a$.

In this part we will extensively use the thermodynamical formalism  of topological Markov shift with countable alphabet. For details, see \cite{IY,Sa}.

\subsection{ The global symbolic space and its fibers}\

Follow the explanation in Section \ref{sec-idea}, we introduce the global symbolic space  as follows.

\subsubsection{ The global symbolic space $\Omega$ and its basic property}\

For each $n\in \N$, we have defined the alphabet $\A_n$ by \eqref{alphabet-n}. Now we define the full  alphabet $\mathscr A$ as
\begin{equation}\label{def-A}
  \A:=\bigsqcup_{n\in \N} \A_n.
\end{equation}

Given  $e,\hat e\in \A$.  Assume $e\in \A_n, {\hat e}\in \A_m$, define $e\to \hat e$ by \eqref{admissible-A-A}. Define the related {\it incidence matrix} $A=(a_{e\hat e})$  as
\begin{equation*}\label{}
  a_{e\hat e}:=
  \begin{cases}
    1, & \mbox{if } e\to\hat e , \\
    0, & \mbox{otherwise}.
  \end{cases}
\end{equation*}
The relation between $A$ and $A_{nm}$ (see \eqref{A_ij}) is
\begin{equation}\label{A-and-A-nm}
  A=
  \begin{bmatrix}
    A_{11}& A_{12}&A_{13}&\cdots \\
    A_{21}& A_{22}&A_{23}&\cdots \\
    A_{31}& A_{32}&A_{33}&\cdots \\
    \vdots& \vdots&\vdots&\ddots \\
  \end{bmatrix}.
\end{equation}

Define the {\it global symbolic space} $\Omega$ to be the {\it topological Markov shift (TMS) } with alphabet $\A$ and incidence matrix $A$ (see \cite[Section 1.3]{Sa} for related definitions). That is, $\Omega$ is a set defined by
\begin{equation}\label{def-Omega}
  \Omega:=\{x=x_1x_2\cdots\in \A^\N: a_{x_jx_{j+1}}=1, j\in \N\},
\end{equation}
equipped with a metric $\rho(x,y):=2^{-|x\wedge y|}$,
and endowed with the shift map $\sigma:\Omega\to \Omega$
\begin{equation}\label{def-shift}
  \sigma(x)=\sigma(x_1x_2\cdots):=x_2x_3\cdots.
\end{equation}

 Thus $(\Omega,\sigma)$ is a topological  dynamical system. It is good in the sense that:

 \begin{proposition}\label{mixing-BIP}
   $(\Omega,\sigma)$ is topologically mixing and satisfies BIP property.
 \end{proposition}

 See \cite[Sections 1.3 and  4.4]{Sa}  for the definitions of topologically mixing and BIP, respectively.

 \begin{proof}
  By Proposition \ref{sturm-primitive}, for any $k\ge 5$,  $A^k$ is a positive (infinite) matrix. So for any $e,\hat e\in \A$ and $k\ge5$, there exists an admissible path
  $e\to e_2\to \cdots \to e_{k}\to \hat e.$
  This implies that $(\Omega,\sigma)$ is topologically mixing (see for example \cite[Proposition 1.1]{Sa}).

  Define
  $
  \mathcal S:=\{(\one,1)_1, (\two,1)_1\}.
  $
 For any $n\in \N$, by \eqref{admissible-T-A} and \eqref{admissible-A-A}, we have
 \begin{eqnarray*}
   &&(\two,1)_1\to (\one,j)_n\to (\two,1)_1,\ \ \ j=1,\cdots,n+1\\
   &&(\one,1)_1\to (\two,1)_n\to (\one,1)_1, \\
   &&(\two,1)_1\to (\three,j)_n\to (\one,1)_1,\ \ \ j=1,\cdots,n
 \end{eqnarray*}
 So $(\Omega,\sigma)$ satisfies the BIP property.
 \end{proof}

Denote  the set of admissible words of length $n$ by $\Omega_n$, then
\begin{equation*}\label{Omega-n}
  \Omega_n=\{w\in \A^n: a_{w_iw_{i+1}}=1, 1\le i<n\}.
\end{equation*}
We also write $\Omega_\ast:=\bigcup_{n\in \N} \Omega_n$.
For any $w\in \Omega_n$,  define the cylinder $[w]$ as
\begin{equation*}\label{def-cylinder}
  [w]:=\{x\in \Omega: x|_n=w\}.
\end{equation*}

\subsubsection{The fibers of $\Omega$}\label{sec-fiber-Omega}\

We define a natural projection $\Pi:\Omega\to \N^\N$ as
\begin{equation}\label{def-Pi}
  \Pi(x):=(\ell_{x_n})_{n\in \N},
\end{equation}
where $\ell_e$ is the level of $e$, see \eqref{type-index-level}.
It is ready  to check that $\Pi$ is surjective and
\begin{equation}\label{factor}
  \Pi\circ\sigma=S\circ\Pi.
\end{equation}
Thus $(\N^\N,S)$ is a factor of $(\Omega,\sigma).$

For any $a\in \N^\N$, define the {\it  fiber of $\Omega$ indexed by $a$} as
\begin{equation*}
  \Omega_a:=\Pi^{-1}(\{a\})=\{x\in \Omega:\Pi(x)=a\}.
\end{equation*}
Assume $a=(a_n)_{n\in \N}$, then $\Omega_a$ can be expressed as the following equivalent form:
\begin{equation*}
  \Omega_a=\{x\in \prod_{n\in \N}\A_{a_n}:x_n\to x_{n+1}, n\in \N\}.
\end{equation*}

For any $\vec a=a_1\cdots a_n\in \N^n$, define
\begin{equation*}
  \Omega_{\vec a,n}:=\{w\in \prod_{j=1}^n\A_{a_j}:w_j\to w_{j+1},1\le j<n\}.
\end{equation*}

For any $a\in \N^\N$, define
\begin{equation*}
  \Omega_{a,n}:=\Omega_{a|_n, n} \ \ \text{ and }\ \  \Omega_{a,\ast}:=\bigcup_{n\in \N} \Omega_{a,n}.
\end{equation*}

For each $w\in \Omega_{a,n}$, define the cylinder   $[w]_a$ as
\begin{equation*}
  [w]_a:=\{x\in \Omega_a: x|_n=w\}.
\end{equation*}

\subsubsection{A bijection between $\Omega_a$ and $\Omega^{1a}$}\

In this part, we study the relation between $\Omega_a$ and $\Omega^a$.
A simple but crucial observation is that, there exists a natural bijection between $\Omega_a$ and $\Omega^{1a}$. This observation allows us to introduce a metric on $\Omega_a$, which captures the geometric information of $\Sigma_{1a,\lambda}$.
We can even patch  all the bijection together to obtain a bijection between $\Omega$ and $\widehat \Omega$, where
\begin{equation}\label{hat-Omega}
  \widehat \Omega:=\bigsqcup_{a\in \N^\N}\Omega^{\check a},\ \ \text{ where }\ \ \check a:=1a.
\end{equation}

Now we construct such a bijection. Define a map $\iota:\Omega\to \{\one,\three\}\times \A_1\times\A^\N$ as
\begin{equation}\label{Def-iota}
    \iota(x):=
    \begin{cases}
      \one(\two,1)_1 x, & \mbox{if } \mathbf t_{x_1}=\one  \text{ or } \three,\\
      \three(\one,1)_1 x, & \mbox{if } \mathbf t_{x_1}=\two.
    \end{cases}
  \end{equation}

 By \eqref{admissible-T-A} and \eqref{admissible-A-A},  $\iota$ is well-defined and obviously  injective. Moreover

\begin{lemma}\label{iota-a}
 $\iota:\Omega\to \widehat \Omega$ is a bijection such that $\iota(\Omega_a)=\Omega^{\check a}$ for any $a\in \N^\N$.
\end{lemma}

\begin{proof}
Fix any $a\in \N^\N$. By \eqref{Omega^alpha} and  \eqref{admissible-T-A},
we have $\iota(\Omega_a)\subset \Omega^{\check a}$. By \eqref{def-Omega^alpha_n} and \eqref{admissible-T-A},
  $$
  \Omega^{\check a}_1=\Omega^{1a}_1=\{\one(\two,1)_1, \three(\one,1)_1\}.
  $$
So for any $y=y_0y_1\cdots\in \Omega^{\check a}$, we have $y^*:=y_2y_3\cdots\in \Omega_a$ and $\iota(y^*)=y.$ Hence $\iota(\Omega_a)=\Omega^{\check a}$. Since $\iota$ is injective, we conclude that $\iota: \Omega_a\to \Omega^{\check a}$ is bijective.

Since $\Omega=\bigsqcup_{a\in \N^\N}\Omega_a$, we conclude that $\iota:\Omega\to\widehat \Omega$ is bijective.
  \end{proof}

  For later use, we extend the definition of $\iota$ to $\Omega_\ast$ as follows. For any $w\in \Omega_\ast$, define
\begin{equation}\label{Def-iota-w}
    \iota(w):=
    \begin{cases}
      \one(\two,1)_1 w, & \mbox{if } \mathbf t_{w_1}=\one  \text{ or } \three,\\
      \three(\one,1)_1 w, & \mbox{if } \mathbf t_{w_1}=\two.
    \end{cases}
  \end{equation}

  As an application of Lemma \ref{iota-a}, we prove the following analog of Lemma \ref{number-Omega-n-alpha}:
\begin{lemma}\label{number-Omega-a-n}
For any $a\in\N^\N$ and $n\in\N$, we have
\begin{equation*}
	q_{n}(a)\leq\#\Omega_{a,n}\leq10q_{n}(a).
\end{equation*}
\end{lemma}

\begin{proof}
  By  \eqref{Def-iota-w}, $\Omega^{\check a}_{n+1}=\iota(\Omega_{a,n})$. By Lemma \ref{iota-a}, Lemma \ref{number-Omega-n-alpha} and Lemma \ref{q-n}(1),
  $$
  q_n(a)\le q_{n+1}(\check a)\le\#\Omega_{a,n}=\#\Omega_{n+1}^{\check a}\le 5q_{n+1}(\check a)\le 10q_n(a).
  $$
  So the result follows.
\end{proof}

\subsubsection{An adapted metric on $\Omega_a$}\

  The map $\iota$ enables us to introduce a metric on $\Omega_a$ as follows.

  At first, define a map
$\pi_{a,\lambda}: \Omega_a\to \Sigma_{\check a,\lambda}$ as
\begin{equation}\label{pi_a-lambda}
  \pi_{a,\lambda}=\pi^{\check a}_\lambda\circ \iota.
\end{equation}
Since both $\iota$ and $\pi^{\check a}_\lambda$ are bijective, so is  $\pi_{a,\lambda}$.  Thus, $\pi_{a,\lambda}$ affords  another  coding of $\Sigma_{\check a,\lambda}$.

  Next define $\rho_{a,\lambda}: \Omega_a\times \Omega_a\to \R^+$ as
\begin{equation}\label{rho-a-lambda}
  \rho_{a,\lambda}(x,y):=|B_{\iota(x)\wedge \iota(y)}^{\check a}(\lambda)|.
\end{equation}
Geometrically, $\rho_{a,\lambda}(x,y)$ is the length of the smallest spectral band which contains both $\pi_{a,\lambda}(x)$ and $\pi_{a,\lambda}(y)$. We have

\begin{proposition}\label{metric-O-a}
  Assume $(a,\lambda)\in \N^\N\times[24,\infty).$ Then $\rho_{a,\lambda}$ is a ultra-metric on $\Omega_a$ and
  $\pi_{a,\lambda}: (\Omega_a,\rho_{a,\lambda})\to (\Sigma_{\check a,\lambda}, |\cdot|)$ is 1-Lipschitz.
   \end{proposition}

\begin{proof}
By the definition, $\rho_{a,\lambda}$ is nonnegative and symmetric.
If $x\ne y$, then $\iota(x)\ne \iota(y)$. Since $\pi^{\check a}_\lambda$ is bijective, we have
$\pi_{a,\lambda}(x)\ne \pi_{a,\lambda}(y)$. Notice that both points are in $B_{\iota(x)\wedge \iota(y)}^{\check a}(\lambda)$, so we have
\begin{equation}\label{1-lip}
  \rho_{a,\lambda}(x,y)=|B_{\iota(x)\wedge \iota(y)}^{\check a}(\lambda)|\ge |\pi_{a,\lambda}(x)- \pi_{a,\lambda}(y)|>0.
\end{equation}

Now assume $z\in \Omega_a.$ By \eqref{Def-iota}, either $\iota(x)\wedge \iota(z) $, or $\iota(y)\wedge \iota(z)$  is a prefix of $ \iota(x)\wedge \iota(y)$. Then by \eqref{rho-a-lambda}, we have
\begin{equation*}
  \rho_{a,\lambda}(x,y)\le \max\{\rho_{a,\lambda}(x,z),\rho_{a,\lambda}(y,z)\}.
\end{equation*}
So $\rho_{a,\lambda}$ is a ultra-metric. By \eqref{1-lip}, $\pi_{a,\lambda}$ is $1$-Lipschitz.
\end{proof}

\begin{remark}
{\rm
  Later we will show that for typical $a,$ $\pi_{a,\lambda}^{-1}$  is weakly Lipschitz (see Corollary \ref{weak-bi-lip-d_alpha}). As a consequence, the dimensional properties of $\Omega_a $ and $\Sigma_{\check a,\lambda}$ is the same.
  }
\end{remark}

From now on, we always endow $\Omega_a$ with the metric $\rho_{a,\lambda}$ without further mention.

\subsubsection{Various subsets of admissible words and their cardinalities }\

Later we will encounter various subsets of admissible words. The cardinalities of these sets are very important for us. Now we introduce them.

Given $\tT,\tT'\in \T$ and $\vec{a}=a_1\cdots a_n\in\N^n$, define
\begin{equation}\label{def-Xi-N}
\begin{cases}
 \ \ \ \ \Xi(\tT,\vec a):=\{w_1\cdots w_n\in \prod_{j=1}^n\A_{a_j}: \tT\to w_1; w_j\to w_{j+1},1\le j<n\}, \\
 \Xi(\tT,\vec a,\tT'):=\{w\in \Xi(\tT,\vec a): \tT_{w}=\tT'\}.
\end{cases}
  \end{equation}

  For any $w\in \Omega_{a,n}, \tT\in \T$, define two sets of descendants of $w$ of $m$-th generation as
\begin{equation}\label{des-w}
\begin{cases}
	\Xi_{a,m}(w):=\{u\in \Omega_{a,n+m}: u|_m=w\}=\{wv: v\in \Xi(\tT_w, S^na|_m)\},\ \ \\
	\Xi_{a,m,\tT}(w):=\{u\in \Xi_{a,m}(w): \tT_u=\tT\}=\{wv: v\in \Xi(\tT_w, S^na|_m,\tT)\}.
\end{cases}
\end{equation}

  We have the following  estimates on the cardinalities of $\Xi(\tT,\vec a)$ and $\Xi(\tT,\vec a,\tT')$.
 These estimates are essential for the study of DOS, see Lemma \ref{imp} and Proposition \ref{imp-dos}.

\begin{lemma}\label{number-est-basic}
For any $n\ge 3$ and $\vec a=a_1\cdots a_n\in \N^n$, the following estimates hold:

(1) \ \ \ \  $\#\Xi(\one, \vec a)\sim q_n(\vec a)/a_{1}.$

(2)  \ \ \ \ $\#\Xi(\two, \vec a)\sim  q_n(\vec a).$

(3)  \ \ \ \ $\#\Xi(\three, \vec a)\sim  q_n(\vec a)$ if $  a_{1}\ge2.$

(4)  \ \ \ \ $\#\Xi(\three, \vec a)\sim q_n(\vec a)/a_{2}$ if $ a_{1}=1.$

(5)  \ \ \ \ $\#\Xi(\tT,\vec a,\two)+\#\Xi(\tT,\vec a,\three)\sim \#\Xi(\tT,\vec a)$ for any $ \tT\in\T.$

 (6)  \ \ \ \ $\#\Xi(\two, \vec a,\one)\sim q_n(\vec a).	$

\end{lemma}

We will prove Lemma \ref{number-est-basic} in Section \ref{sec-cardinality}.

\subsection{The entropic potential $\Phi$ and related Gibbs measure $\n$}\

In this part, we  introduce a special potential $\Phi$ on $\Omega$,  which we call the {\it entropic potential}. We will show that $\Phi$ admits a Gibbs measure $\n$, and obtain exact estimates for its fiber measure $\n_a$.  Later  we will see that the image of  $\n_a$ under $\pi_{a,\lambda}$ is equivalent to $\NN_{\check a,\lambda}$ (see Corollary \ref{equiv-fiber-dos-coro}).

\subsubsection{Recall on almost-additive  thermodynamical formalism for TMS}\

For this part, we essentially follow \cite{IY}, with some change of notations.

Assume $\Upsilon=\{\upsilon_n:n\in\N\}\subset C(\Omega,\R)$. We  call $\Upsilon$ a {\it potential} on $\Omega.$ We say that $\Upsilon$ is {\it almost-additive}, if there exists a constant $C=C_{aa}(\Upsilon)>0$ such that
\begin{equation*}\label{C-aa-Phi}
|\upsilon_{n+m}(x)-\upsilon_n(x)-\upsilon_m(\sigma^n(x))|\leq C,\ \ \forall x\in\Omega,\ \forall n,m\in \N.
\end{equation*}
We say that $\Upsilon$ has {\it bounded variation}, if there exists a constant $C=C_{bv}(\Upsilon)>0$ such that
\begin{equation*}\label{def-bd-variation}
\sup\{|\upsilon_n(x)-\upsilon_n(y)|: x|_n=y|_n\}\le C,\ \ \forall n\in \N.
\end{equation*}

Denote by $\mathcal{F}(\Omega)$ the set of  almost-additive potentials with  bounded variation. Let $\mathcal M(\Omega)$ be the set of $\sigma$-invariant Borel probability measures on $\Omega$.

\begin{definition}\label{Gibbs}
Assume   $\Upsilon\in\mathcal{F}(\Omega)$. A measure $\mu\in\mathcal M(\Omega)$ is called a Gibbs measure for $\Upsilon$ if there exist two  constants $C\ge1, P\in \R$ such that for any $n\in \N$ and $x\in \Omega$,
\begin{equation*}\label{def-Gibbs}
\frac{1}{C}\le \frac{\mu([x|_n])}{\exp(-nP+\upsilon_n(x))}\le C.
\end{equation*}
\end{definition}

The criterion for  existence and uniqueness of Gibbs measure is  established in \cite{IY}:

\begin{theorem}[{\cite[Theorem 4.1]{IY}}]\label{IY-exist-Gibbs}
Assume $\Upsilon\in \mathcal F(\Omega)$  satisfies
\begin{equation}\label{condi-gibbs}
\sum_{e\in \mathscr A} \sup\left\{\exp(\upsilon_1(x)):x\in [e]\right\}<\infty,
\end{equation}
then there is a unique Gibbs measure $\mu$ for $\Upsilon$ and it is ergodic with $P=P(\Upsilon)$, where
\begin{equation}\label{def-pressure}
P(\Upsilon):=\lim\limits_{n\to\infty}\frac{1}{n}\log \sum_{x:\sigma^n (x)=x}\exp(\upsilon_n(x))\chi_{[e]}(x),\ \ \ e\in\mathscr{A},
\end{equation}
is the {\it Gurevich pressure} of $\Upsilon$, which exists and does not depend on $e$.

\end{theorem}

We note that, in the original theorem of Iommi and Yayama, the TMS is asked to be topologically mixing and satisfy BIP. It is the case for $(\Omega,\sigma)$ by Proposition \ref{mixing-BIP}.

\subsubsection{The entropic potential $\Phi$ and its Gibbs measure $\n$}\

Let us give some hints for the definitions of $\Phi$ and $\n$. We have explained how to construct $\n$ formally in Section \ref{sec-idea}. For any $\vec a\in \N^n$, we know that $\Gg([\vec a])\sim q_n^{-2}(\vec a)$.  For any $a\in [\vec a]$, the $\n_a$-measure of any $n$-th cylinder in $\Omega_a$ is  of the order $q_n^{-1}(a)$. Thus for any such cylinder, the $\n$-measure of it is roughly $q_n^{-3}(a)$. Since we expect $\n$ is the Gibbs measure of $\Phi$, it is natural to introduce the following definition.

Define $\phi_n:\Omega\to\mathbb{R}$ as
	\begin{equation}\label{phi-poten}
		\phi_n(x):=-3\log q_n(\Pi(x)).
	\end{equation}
It is seen that $\phi_n$ is continuous.
Write $\Phi:=\{\phi_n:n\in \N\}$, then $\Phi$ is a potential on $\Omega$.

	\begin{proposition}\label{Gauss-measure-Omega}
		$\Phi\in\mathcal F(\Omega)$. Moreover, there exists a unique Gibbs measure $\n$ for $\Phi$ with   $P(\Phi)=0$ and
$\Pi_*(\n)=\Gg.$
	\end{proposition}

	\begin{proof}
By \eqref{factor} and Lemma \ref{q-n}(1), it is seen that $\Phi$ is almost-additive. By \eqref{def-Pi} and Remark \ref{q-n-vec-a}, $\Phi$ has bounded variation with $C_{bv}(\Phi)=0$. So, $\Phi\in \mathcal F(\Omega)$.

By the definition of $\A$ (see \eqref{def-A} and \eqref{alphabet-n}), we have
\begin{eqnarray*}
\sum_{e\in\A}\sup_{x\in[e]} e^{\phi_1(x)}  =\sum_{e\in\A}\frac{1}{q_1(\Pi(x))^3}=\sum_{n=1}^{\infty}\sum_{e\in\A_n}\frac{1}{n^3}= \sum_{n=1}^{\infty}\frac{2n+2}{n^3}<\infty.
\end{eqnarray*}
So \eqref{condi-gibbs} holds. By Theorem \ref{IY-exist-Gibbs}, there is a unique Gibbs measure for $\Phi$, which is ergodic. We denote this measure by $\n.$

Next we compute $P(\Phi)$.
Take $e=(\two,1)_1$. Assume $x\in \Omega$ is such that $x_1=e$ and $\sigma^{n+1}(x)=x$. Then by \eqref{admissible-T-A}, we must have $\sigma(x)|_n\in \Xi(\two,\vec a(x),\one)$, where $\vec a(x)=\Pi(\sigma (x))|_n$. Conversely, assume $w\in \Xi(\two,\vec a,\one)$ for some $\vec a\in \N^n$, if we define $x^w:=(ew)^\infty$, then $x^w\in \Omega$, $x^w_1=e$ and $\sigma^{n+1}(x^w)=x^w$.
Hence, by Lemma \ref{number-est-basic}(6) and Lemma \ref{q-n}, we have
\begin{align*}
\sum_{x: \sigma^{n+1}x=x}e^{\phi_{n+1}(x)}\chi_{[e]}(x)=&\sum_{\vec{a}\in \N^n}\sum_{w\in \Xi(\two,\vec{a},\one)}e^{\phi_{n+1}(x^w)}
=\sum_{\vec{a}\in \N^n}\sum_{w\in \Xi(\two,\vec{a},\one)}\frac{1}{q_{n+1}(1\vec{a})^3}\\
\sim&\sum_{\vec{a}\in \N^n}\frac{1}{q_{n}(\vec{a})^2}\sim\sum_{\vec{a}\in \N^n}\Gg([\vec{a}])=1.
\end{align*}
By \eqref{def-pressure}, we have $P(\Phi)=0$.

 Since $\n$ is ergodic and $\Pi$ is a factor map,  $\mu:=\Pi_\ast(\n)$ is also ergodic. Fix  $a\in\N^\N$ and $n\in\N$. For any $w\in \Omega_{a,n}$ and $x_w\in[w]$, by the Gibbs property we have
\begin{equation}\label{meas-Gibbs}
\n([w])\sim \frac{e^{\phi_n(x_w)}}{e^{nP(\Phi)}}=\frac{1}{q_n(a)^3}.
\end{equation}
Combine with Lemma \ref{number-Omega-a-n} and Lemma \ref{q-n}(2), we get
\begin{eqnarray*}
\mu([a|_n]) = \n(\Pi^{-1}([a|_n]))=\n(\bigcup_{w\in \Omega_{a,n}}[w])=\sum_{w\in \Omega_{a,n}}\n([w])\sim \frac{1}{q_n(a)^2}\sim \Gg([a|_n]).
\end{eqnarray*}
This means that $\mu\asymp \Gg$. Since both measures are ergodic, we conclude that $\mu=\Gg.$
\end{proof}

We call $\Phi$ the {\it entropic potential} on $\Omega$, since we will see soon that it is related to the entropy of the fiber measures of $\n$.

\subsubsection{The fiber measures of $\n$}\

In this part, we will show that there exists a disintegration of $\n$ w.r.t. $\Gg$ and obtain exact estimates for the fiber measures.

Notice that for any $a\in \N^\N$, the fiber $\Omega_a=\Pi^{-1}(\{a\})$ is a closed set of $\Omega$. Since $\Pi$ is a factor map, the family $
\{\Omega_a: a\in \N^\N\}$
forms a measurable  partition of $\Omega$.

Given $\mu\in\mathcal{M}(\Omega)$ and let $\nu=\Pi_*(\mu)$. A family of measures $\{\mu_{a}:a\in \N^\N\}$ on $\Omega$ is called a {\it disintegration of $\mu$ w.r.t. $\nu$} if for each Borel measurable set $A\subset\Omega$, the map $a\to\mu_{a}(A)$ is measurable; for $\nu$-a.e. $ a\in \N^\N$, the measure $\mu_{a}$ is a Borel probability measure supported on $\Omega_{a}$. If such a family exists, then  $\mu_{a}$ is called the {\it fiber} of $\mu$ indexed by $a$. In this case, we write
\begin{equation*}\label{form-disinteg}
\mu=\int_{\N^\N}\mu_{a}d\nu(a).
\end{equation*}

\vspace{1ex}
By the classical Rokhlin's theory \cite{Ro}, we have

\begin{lemma}[\cite{Ro}]\label{basic-disinteg}
Fix any $\mu\in\mathcal{M}(\Omega)$ and let $\nu=\Pi_*(\mu)$. Then

(1) The disintegration of $\mu$  w.r.t. $\nu$ always exists. It is unique in the sense that if $\{\tilde{\mu}_{a}:a\in \N^\N\}$ is another disintegration, then $\mu_{a}=\tilde{\mu}_{a}$ for $\nu$-a.e. $a$.

(2) For $\nu$-a.e. $a$ and any $w\in\Omega_{a,n}$, $\mu_{a}([w]_{a})$ can be computed as
\begin{equation*}\label{compute-fiber-meas}
\mu_{a}([w]_{a})=\lim_{m\to\infty}
\frac{\sum_{u\in\Xi_{a,m}(w)}\mu([u])}{\sum_{u\in\Omega_{a,n+m}}\mu([u])}.
\end{equation*}
\end{lemma}

\begin{proof}
To prove this lemma, we only need to check that $\{\Omega_{a}:a\in \N^\N\}$ is a measurable partition in the sense of Rokhlin. For each $n\in \N$ and any $\vec a\in \N^n$, define
\begin{equation*}
\Omega_{\vec a}:=\bigcup_{w\in \Omega_{\vec{a},n}}[w].
\end{equation*}
Then  the measurable family
$
\{\Omega_{\vec a}: n\in \N, \vec a\in \N^n\}
$
is countable and
for each $a\in \N^\N$ we have
$
\Omega_a=\bigcap_{n\in \N}\Omega_{a|_n}.
$
So $\{\Omega_{a}:a\in \mathbb{N}^{\mathbb{N}}\}$ is a measurable partition in the sense of Rokhlin. Now (1) and (2) follow from the classical Rokhlin's theory, see
\cite{Pa,Ro}. For part (2), see also \cite[Lemma 4.1]{FS}.
\end{proof}

Now we apply Lemma \ref{basic-disinteg} to compute the fiber of $\n$. Recall that $\F_2$ is a $\Gg$-full measure set defined in Corollary \ref{coincide-whole}.

\begin{lemma}\label{imp}
Let $\n$ be the Gibbs measure defined by Proposition \ref{Gauss-measure-Omega}. Then the disintegration $\{\n_a: a\in \N^\N\}$ of $\n$ w.r.t. to $\Gg$ exists and is unique.  Moreover, there exists a subset  $\F_3\subset \F_2$  with $\Gg(\F_3)=1$ such that for any  $a\in \F_3$ and any $w\in \Omega_{a,n}$,
\begin{equation}\label{fiber-measure}
\n_a([w]_a)\sim
\frac{1}{\eta_{\tT_w,a,n}\ q_{n}(a)}, \ \ \text{ where }\ \
\eta_{\tT,a,n}=\begin{cases}
		a_{n+1},& \text{ if }\ \tT=\one,\\
		1,& \text{ if }\ \tT=\two,\\
		1,& \text{ if }\ \tT=\three;\ a_{n+1}\ge2,\\
		a_{n+2} ,& \text{ if }\ \tT=\three;\ a_{n+1}=1.
	\end{cases}
\end{equation}

\end{lemma}

\begin{proof}
By Proposition \ref{Gauss-measure-Omega}, $\Gg=\Pi_\ast(\n)$. By Lemma \ref{basic-disinteg}(1),  the disintegration exists and is unique. By Lemma \ref{basic-disinteg}(2), there exists a set $\F_3 \subset \F_2$ with $\Gg(\F_3)=1$, such that for any $a \in \F_3$ and any $w \in \Omega_{a,n}$,
\begin{align*}
\n_{a}([w]_{a})=&\lim_{m\to\infty}
\frac{\sum_{u\in\Xi_{a,m}(w)}\n([u])}{\sum_{u\in\Omega_{a,n+m}}\n([u])}.
\end{align*}

By \eqref{meas-Gibbs}, Lemma \ref{number-Omega-a-n}, \eqref{des-w} and Lemma \ref{q-n}(1),
we have
\begin{eqnarray*}
&&\frac{\sum_{u\in\Xi_{a,m}(w)}\n([u])}{\sum_{u\in\Omega_{a,n+m}}\n([u])}
\sim\frac{\#\Xi_{a,m}(w)}{\#\Omega_{a,n+m}}\sim \frac{\#\Xi(\tT_w,S^na|_m)}{q_{n+m}(a)}\sim\frac{\#\Xi(\tT_w,S^na|_m)}{q_{n}(a)q_{m}(S^na)}.
\end{eqnarray*}
Now combine with Lemma \ref{number-est-basic}(1)-(4), we obtain  \eqref{fiber-measure}.
\end{proof}

\begin{remark}\label{fiber-entropy}
  {\rm By Lemma \ref{number-Omega-a-n},  the number of $n$-th cylinders of $\Omega_a$ is of the order $q_n(a)$. By Lemma \ref{imp}, for  $a\in \F_3$, the $\n_a$ measure of each cylinder is roughly of the order $1/q_n(a)$. This means that $\n_a$ is evenly distributed on the $n$-th cylinders. In this sense, we can say that $\n_a$ is kind of ``measure of maximal entropy". Since $\F_3\subset \F_2\subset \F_1$, we also know that $\log q_n(a)/n\to \gamma$. Thus
  $$
  -\frac{\log(\n_a([w]_a))}{n}\to \gamma.
  $$
  By an analog with Shannon-McMillan-Breiman theorem,  $\gamma$ plays the role of the entropy of $\n_a.$ In this sense, the potential $\Phi$ is related to entropy.
  This explains our terminology of ``entropic potential".
  }
\end{remark}

\subsection{The geometric potential $\Psi_\lambda$}\

Motivated by the Young's dimension formula, we  introduce another potential which contains the  geometric information of the spectra and will be related to Lyapunov exponent in the end.

Fix $\lambda\in [24,\infty)$. Define the {\it geometric potential} $\Psi_\lambda$ on $\Omega$ as
\begin{equation}\label{def-Psi-lambda}
  \psi_{\lambda,n}(x):=\log|B^{\check a}_{\iota(x)|_{n+1}}(\lambda)|,\ \text{ where }\ a=\Pi(x);\ \  \Psi_\lambda:=\{\psi_{\lambda,n}:n\in \N\}.
\end{equation}

 Here, $\exp(\psi_{\lambda,n}(x))$  is  the length of the band in $\B_{n+1}^{\check a}(\lambda)$ which contains $\pi_{a,\lambda}(x)$. Thus $\psi_{\lambda,n}$ records the geometric information of $\B_{n+1}^{\check a}(\lambda)$. That is why we call $\Psi_\lambda$ the geometric potential.

We have the following analog of \cite[Lemma 7]{Q1}:

\begin{lemma}\label{Psi-aa}
Assume $\lambda\in[24,\infty)$. Then   $\Psi_{\lambda}\in \mathcal F(\Omega)$ and for any $n\in \N,$
\begin{equation}\label{negative-psi}
  \psi_{\lambda,n}(x)\le (1-n)\log2, \ \ \forall n\in \N, x\in \Omega.
\end{equation}
\end{lemma}

\begin{proof}
 By \eqref{def-Psi-lambda}, \eqref{Def-iota} and Remark \ref{B-n}(ii),  $\Psi_\lambda$ has bounded variation property with $C_{bv}(\Psi_\lambda)=0$.

Next  we show that $\Psi_\lambda$ is almost-additive.
Fix $x\in\Omega$, assume $\Pi(x)=a$. Write $b=S^na$ and $y=\sigma^n(x)$.  Then we
			have
\begin{align*}
  \psi_{\lambda,n}(x)=&\log|B_{\iota(x)|_{n+1}}^{\check{a}}(\lambda)|,\ \ \
\psi_{\lambda,n+k}(x)=\log|B_{\iota(x)|_{n+k+1}}^{\check{a}}(\lambda)|,\\
\psi_{\lambda,k}(\sigma^n(x))=&\psi_{\lambda,k}(y)=\log|B_{\iota(y)|_{k+1}}^{\check{b}}(\lambda)|.
\end{align*}

By Proposition \ref{bco},  we have
		$$ \frac{|B_{\iota(x)|_{n+k+1}}^{\check{a}}(\lambda)|}
{|B_{\iota(x)|_{n+1}}^{\check{a}}(\lambda)|}\sim
\frac{|B^{\check{b}}_{\iota(y)|_{k+1}}(\lambda)|}{|B^{\check{b}}_{\iota(y)|_1}(\lambda)|}.$$
By \eqref{Def-iota}, we see that $\iota(y)|_1=\one(\two,1)_1$ or $\three(\one,1)_1$. By lemma \ref{esti-band-length}, we have
\begin{equation*}\label{B}
\frac{1}{\tau_2}\le|B_{\iota(y)|_1}^{\check b}(\lambda)|\le 2.
\end{equation*}
Consequently, we have
 $|\psi_{\lambda,n+k}(x)-\psi_{\lambda,n}(x)-\psi_{\lambda,k}(\sigma^nx)|\lesssim 1.$
Thus $\Psi_\lambda$ is almost-additive. So $\Psi_\lambda\in \mathcal F(\Omega).$

By Lemma \ref{esti-band-length}, for any $n\in \N$ and $x\in \Omega$, we have $|B_{\iota(x)|_{n+1}}^{\check a}(\lambda)|\le 2^{1-n}.$
So \eqref{negative-psi} holds.
	\end{proof}

Now for typical $a$, we can express the local dimension of $\n_a$ via the geometric potential.

We introduce some notations. For any $x\in \Omega$, define
\begin{equation}\label{lyapunov-new}
  \overline\eL_\lambda(x):=\limsup_{n\to\infty}\frac{-\psi_{\lambda,n}(x)}{n};\ \ \underline\eL_\lambda(x):=\liminf_{n\to\infty}\frac{-\psi_{\lambda,n}(x)}{n}.
\end{equation}

 Notice that the metric $\rho_{a,\lambda}$ on $\Omega_a$ depends on $\lambda.$ So
for  $x\in \Omega_a$, we  write the lower and upper local dimensions of $\n_a$ at $x$ as
\begin{equation*}
  \underline{d}_{\n_a}(x,\lambda) \ \ \ \text{ and }\ \ \  \overline{d}_{\n_a}(x,\lambda)
\end{equation*}
to emphasize  the dependence on $\lambda$.

\begin{lemma}\label{loc-dim-cylinder}
For any $(a,\lambda)\in \F_3\times [24,\infty)$ and any $x\in \Omega_a$, we have
\begin{equation}\label{dim-cylinder}
\underline{d}_{\n_a}(x,\lambda)=\frac{\gamma}{\overline\eL_\lambda(x) }\ \ \ \text{ and }\ \ \
 \overline{d}_{\n_a}(x,\lambda)=\frac{\gamma}{\underline\eL_\lambda(x) }.
\end{equation}
\end{lemma}

\begin{proof}
We only prove the first equality of \eqref{dim-cylinder}, since the proof the second one is the same.
Fix any $(a,\lambda)\in \F_3\times [24,\infty)$ and  $x\in \Omega_a$.  By \eqref{rho-a-lambda}, \eqref{def-Psi-lambda} and Lemma \ref{Psi-aa},
\begin{equation}\label{r-n}
  r_n:={\rm diam}([x|_n]_a)=\exp(\psi_{\lambda,n}(x))\le 2^{1-n}.
\end{equation}
Since $a\in \F_3\subset \F_1$, by Lemma \ref{imp} and Proposition \ref{F-proposition}(1),
\begin{equation}\label{mass-n-a}
\frac{1}{a_{n+1}a_{n+2}q_n(a)}\lesssim \n_a([x|_n]_a)\lesssim \frac{1}{q_{n}(a)};\ \
\frac{\log q_n(a)}{n}\to \gamma;\ \ \  \frac{\log a_{n}}{n}\to 0.
\end{equation}
 Consequently,
\begin{equation*}
   \lim_{n\to\infty}\frac{-\log \n_a([x|_n]_a) }{n}=\gamma;\ \  \limsup_{n\to\infty}\frac{-\log {\rm diam}([x|_n]_a)}{n}=\limsup_{n\to\infty}\frac{-\psi_{\lambda,n}(x)}{n}=\overline\eL_\lambda(x).
\end{equation*}
So  we have
\begin{equation}\label{bridge}
  \liminf_{n\to \infty}\frac{\log \n_a([x|_n]_a)}{\log {\rm diam}([x|_n]_a)}=\frac{\gamma}{\overline\eL_\lambda(x) }.
\end{equation}

Since $\rho_{a,\lambda}$ is a ultra-metric on $\Omega_a$, we have
$
[x|_n]_a=\overline{B(x,r_n)}.
$
Since $\overline{B(x,r_n)}=\bigcap_{r>r_n}B(x,r),$
by the continuity of $\n_a$, we have
$$
\frac{\log \n_a([x|_n]_a)}{\log {\rm diam}([x|_n]_a)}=\lim_{r\downarrow r_n}\frac{\log \n_a(B(x,r))}{\log r}.
$$
Take into account the definition \eqref{def-loc-dim}, we conclude that
\begin{equation}\label{first-half-0}
\underline{d}_{\n_a}(x,\lambda)=\liminf_{r\to0}\frac{\log \n_a(B(x,r))}{\log r}\le\liminf_{n\to \infty}\frac{\log \n_a([x|_n]_a)}{\log {\rm diam}([x|_n]_a)}.
\end{equation}

By Proposition \ref{basic-struc}(1) and the fact that $\rho_{a,\lambda}$ is a ultra-metric, for each small $r>0$, there exists $n(r)\in \N$ with $n(r)\to\infty$ as $r\to 0$, such that
\begin{equation*}\label{r-compare}
r_{n(r)}<r\le r_{n(r)-1};\ \ \ 	[x|_{n(r)}]_a=  B(x,r)\subset [x|_{n(r)-1}]_a.
\end{equation*}
So we have
\begin{eqnarray*}\label{first-half}
\frac{\log \n_a([x|_{n(r)-1}]_a) }{\log {\rm diam}([x|_{n(r)}]_a)}\le\frac{\log \n_a(B(x,r))}{\log r} .
\end{eqnarray*}
Let $r\to0$ we have
\begin{eqnarray}\label{first-half}
\liminf_{n\to \infty}\frac{\log \n_a([x|_{n-1}]_a) }{\log {\rm diam}([x|_{n}]_a)}\le\liminf_{r\to0}\frac{\log \n_a(B(x,r))}{\log r}=\underline{d}_{\n_a}(x,\lambda).
\end{eqnarray}
Compare \eqref{bridge}, \eqref{first-half-0} and \eqref{first-half}, to show the first equality, it is enough to show that
\begin{equation*}
  \liminf_{n\to \infty}\frac{\log \n_a([x|_{n-1}]_a)-\log\n_a([x|_{n}]_a) }{\log {\rm diam}([x|_{n}]_a)}=0.
\end{equation*}
But this follows easily from \eqref{mass-n-a}, Lemma \ref{q-n}(1) and \eqref{r-n}.
\end{proof}

\begin{remark}
  {\rm
  We have explained in Remark \ref{fiber-entropy} that $\gamma$ can be viewed as the entropy of $\n_a$. On the other hand, by the equality of \eqref{r-n}, $\overline\eL_\lambda(x)$ and $\underline\eL_\lambda(x)$ are  the   exponents of decay  for the  diameters of the cylinders $[x|_n]_a$. So they can be viewed as local upper and lower Lyapunov exponents at $x$. Thus \eqref{dim-cylinder} are  kinds of local Young's formulas. Later, by using ergodic theorem, we can show that for typical $a$ and typical $x\in \Omega_a$, the two Lyapunov exponents coincide, and consequently establish the exact-dimensionality of $\n_a.$
  }
\end{remark}
\subsection{Dimensional properties of the fiber measures}\

In this part, we follow the similar strategy  we  have used in studying the spectra. Namely, at first, we prove two properties of $\n_a$ for ``deterministic" frequency $a$, which are the  exact lower- and upper- dimensional properties of $\n_a$ and  the continuity of the dimensions of $\n_a$. Then, by using ergodicity  of $\n$, we improve the exact lower- and upper- dimensional properties to exact-dimensionality and then obtain the complete result.

\subsubsection{Dimensional properties of fiber measures---``deterministic"  case}\label{subsubsec-deterministic}\

At first we have

\begin{proposition}\label{exact-H-P-dim}
    For any $(a,\lambda)\in\F_3\times[24,\infty)$,  $\n_{a}$ is both exact lower- and upper-dimensional.
\end{proposition}

This proposition is an analog of  \cite[Theorem 3.3]{Q2}. The proof is also an adaption of the argument given in \cite{Q2}. However, the proof here is considerably more involved since we are dealing with symbolic space  with an infinite alphabet. We will prove Proposition \ref{exact-H-P-dim} in Section \ref{sec-proof-exact-dim}.

Define two functions  $\underline{d},\  \overline{d}:\F_3\times [24,\infty)\to [0,1]$ as
 \begin{equation}\label{def-d-a-lambda}
   \underline{d}(a,\lambda):=\dim_H (\n_a,\lambda);\ \ \ \overline{d}(a,\lambda):=\dim_P (\n_a,\lambda).
 \end{equation}
 Here again, we write $\dim_{\ast}(\n_a,\lambda)$ for $\dim_\ast \n_a$ to emphasize the dependence on $\lambda.$

 \begin{proposition}\label{lip-fiber}
  There exist a  set $\F_4\subset \F_3$ with $\Gg(\F_4)=1$ and an  absolute constant $C>0$  such that for any $a\in \F_4$ and $24\le\lambda_1<\lambda_2$  we have
  \begin{equation}\label{Lip-of-d-a-lambda}
    |\underline{d}(a,\lambda_1)-\underline{d}(a,\lambda_2)|,\ \  |\overline{d}(a,\lambda_1)-\overline{d}(a,\lambda_2)|\le C\lambda_1\frac{|\lambda_1-\lambda_2|}{(\log \lambda_1)^2}.
  \end{equation}
  \end{proposition}

  The proof is quite delicate and will be given in Section \ref{sec-proof-conti-d}.

\subsubsection{Dimensional properties of fiber measures---almost sure case}\

By Lemma \ref{Psi-aa}, for any $\lambda\in [24,\infty)$,    the following limit exists and satisfies
\begin{equation}\label{Psi-ast-n}
(\Psi_\lambda)_\ast(\n):=\lim_{n\to \infty} \frac{1}{n}\int_{\Omega}\psi_{\lambda,n}(x)d\n(x)\le  -\log 2.
\end{equation}

We have the following asymptotic behavior for $(\Psi_\lambda)_\ast(\n)$.

\begin{lemma}\label{asym-psi-lambda-n}
  There exists a constant $\theta>0$ such that
  \begin{equation*}\label{}
    \lim_{\lambda\to\infty}\frac{(\Psi_\lambda)_\ast(\n)}{\log \lambda}=-\theta.
  \end{equation*}
\end{lemma}
Lemma \ref{asym-psi-lambda-n} follows easily from a technical lemma (Lemma \ref{iota-Psi-a}) and will be proven in Section \ref{sec-proof-conti-d}.

Define a function $d:[24,\infty)\to \R^+$ as
\begin{equation}\label{def-dim-dos}
  d(\lambda):=\frac{\gamma}{-(\Psi_\lambda)_\ast(\n)}.
\end{equation}

\begin{proposition}\label{exact-dim-fix-lambda}
  Fix $\lambda\in [24,\infty)$. Then there exists a $\Gg$-full measure set $
  \widehat \F(\lambda)\subset \F_4$ such that for any $a\in \widehat\F(\lambda)$, $\n_a$ is exact-dimensional and
  \begin{equation}\label{fix-lambda}
    \underline{d}(a,\lambda)=\overline{d}(a,\lambda)=d(\lambda).
  \end{equation}
\end{proposition}

\begin{proof}
Define
\begin{equation*}\label{Y-lambda}
  \Omega(\lambda):=\{x\in \Omega: \frac{\psi_{\lambda,n}(x)}{n}\to (\Psi_\lambda)_\ast(\n)\}.
\end{equation*}
Since $\n$ is ergodic, by Kingman's ergodic theorem, $\n(\Omega(\lambda))=1$.    Since
$$
1=\n(\Omega(\lambda))=\int_{ \N^\N}\n_a(\Omega(\lambda))d\Gg(a)=\int_{\N^\N} \n_a(\Omega(\lambda)\cap \Omega_a)d\Gg(a),
$$
 there exists a $\Gg$-full measure set $\widehat\F(\lambda)\subset \F_4$ such that
\begin{equation*}\label{full-meas-fiber}
  \n_a(\Omega(\lambda)\cap \Omega_a)=1,\ \ \forall a\in \widehat \F(\lambda).
\end{equation*}

Now fix any $a\in \widehat \F(\lambda).$
By the definition of $\Omega(\lambda)$, for any $x\in \Omega(\lambda)\cap\Omega_a$, we have
$$
\lim_{n\to\infty} \frac{\psi_{\lambda,n}(x)}{n}=(\Psi_\lambda)_\ast(\n).
$$
Since $\widehat\F(\lambda)\subset \F_4\subset  \F_3$, by Lemma \ref{loc-dim-cylinder},
for any $x\in \Omega(\lambda)\cap\Omega_a$,
\begin{equation*}\label{}
  {d}_{\n_a}(x,\lambda)=
  \frac{\gamma}{-(\Psi_\lambda)_\ast(\n)}=d(\lambda).
\end{equation*}
This means that $\n_a$ is exact-dimensional, and consequently,
 \eqref{fix-lambda} holds.
\end{proof}

As a consequence, we get the regularity and the asymptotic behavior  of $d$:

\begin{corollary}\label{lip-d}
There exists an absolute constant $C>0$  such that for any  $24\le\lambda_1<\lambda_2$,
  \begin{equation}\label{Lip-of-d-lambda}
   |d(\lambda_1)-{d}(\lambda_2)|\le C\lambda_1\frac{|\lambda_1-\lambda_2|}{(\log \lambda_1)^2}.
  \end{equation}
  Moreover, there exists a constant $\varrho\in (1,\infty)$ such that
\begin{equation}\label{asymp-d}
  \lim_{\lambda\to\infty}d(\lambda)\log\lambda=\log \varrho.
\end{equation}
\end{corollary}

\begin{proof}  Fix any $24\le \lambda_1<\lambda_2$, let $\widehat \F(\lambda_i),  i=1,2$ be the $\Gg$-full measure sets in Proposition \ref{exact-dim-fix-lambda}.
Fix $a\in \widehat\F(\lambda_1)\cap \widehat\F(\lambda_2)$, then $a\in \F_4.$ By Proposition \ref{exact-dim-fix-lambda},
\begin{eqnarray*}
  |d(\lambda_1)-d(\lambda_2)| = |\underline{d}(a,\lambda_1)-\underline{d}(a,\lambda_2)|.
\end{eqnarray*}
Now \eqref{Lip-of-d-lambda} follows from Proposition \ref{lip-fiber}.
Define
\begin{equation}\label{def-var-rho}
  \varrho:=\exp(\gamma/\theta).
\end{equation}
By \eqref{def-dim-dos} and Lemma \ref{asym-psi-lambda-n}, we obtain \eqref{asymp-d}.
\end{proof}

Now we can improve Proposition \ref{exact-dim-fix-lambda} as follows:

\begin{corollary}\label{h=p}
There exists a  $\Gg$-full measure set $\widehat{\F}\subset \F_4$ such that for any $(a,\lambda)\in\widehat\F\times[24,\infty)$, $\n_a$ is exact-dimensional and
  \begin{equation}\label{fix-lambda-strong}
    \underline{d}(a,\lambda)=\overline{d}(a,\lambda)=d(\lambda).
  \end{equation}
\end{corollary}

\begin{proof}		
Define the $\Gg$-full measure set $\widehat{\F}$ as
\begin{equation}\label{def-hat-F}
\widehat\F:=\bigcap_{\lambda\in [24,\infty)\cap \Q} \widehat\F(\lambda).
\end{equation}

Fix $a\in\widehat{\F}$. Since $a\in\widehat{\F}(\lambda)$ for $\lambda\in[24,\infty)\cap \Q$, by Proposition \ref{exact-dim-fix-lambda}, we have
\begin{equation*}\label{l-u-eq}
\underline{d}(a,\lambda)=\overline{d}(a,\lambda)=d(\lambda),\ \ \ \forall\ \lambda\in [24,\infty)\cap \Q.
\end{equation*}
Since $a \in \widehat{\F} \subset \F_4$, by Proposition \ref{lip-fiber}, both $\underline{d}(a,\cdot)$ and $\overline{d}(a,\cdot)$ are continuous on $[24, \infty)$. By Corollary \ref{lip-d}, $d$ is also continuous on $[24, \infty)$. Therefore,
\begin{equation*}\label{l-u-eq-all}
\underline{d}(a,\lambda)=\overline{d}(a,\lambda)=d(\lambda),\ \ \ \forall\ \lambda\in [24,\infty).
\end{equation*}
Since $a\in\widehat{\F} \subset \F_3$, by Proposition \ref{exact-H-P-dim} and \eqref{fix-lambda-strong}, $\n_a$ is exact-dimensional.
\end{proof}

\section{Two geometric lemmas and their consequences}\label{sec-geometric-lemma}

In this section, we state two geometric lemmas and use them to transfer the dimensional property of $\n_a$ to that of $\NN_{a,\lambda}$, and finish the proof of Theorem \ref{main-dim-dos}.

The first lemma gives a lower bound on the ratio of the lengths between a gap and its father band. By this lemma, we can show that for typical $a$, the coding map $\pi_{a,\lambda}: \Omega_a\to \Sigma_{\check a,\lambda}$ is close to bi-Lipschitz. Hence $\n_a$ and $\NN_{\check a,\lambda}$ have the same dimensional properties.

The second lemma says that, if $a,b\in \N^\N$ have the same tail, then $\Sigma_{a,\lambda}$ and $\Sigma_{b,\lambda}$ are locally bi-Lipschitz equivalent. By this lemma, we can show that  $\NN_{\check a,\lambda}$ and $\NN_{a,\lambda}$ have the same dimensional properties.

Combine the properties of $\n_a$ established in Section \ref{sec-dim-fiber-symbolic}, we conclude Theorem \ref{main-dim-dos}.
The proofs of two lemmas are quite technical,  and  will be given in Section \ref{sec-proof-geo-lem}.

\subsection{The statements of two geometric lemmas } \

 At first, we  introduce some terminologies.

\begin{definition}
Assume $X,Y$ are two metric spaces and  $f:X\to Y$ is a map.

(1) $f$ is called weak-Lipschitz if  $f$ is $s$-H\"older continuous for any $s\in (0,1)$.

(2) $f$ is called bi-Lipschitz   if $f$ is bijective and both $f$ and $f^{-1}$ are Lipschitz.

(3) If there exists a bi-Lipschitz map $f:X\to Y$, then we say that $X$ and $Y$ are bi-Lipschitz equivalent.
\end{definition}

  Recall that $\mathcal{B}_n^a(\lambda)$ and $\mathcal{G}_n^a(\lambda)$ are sets of  spectral generating bands and gaps of order $n$, respectively(see \eqref{def-B-n} and \eqref{def-G-n}). The first geometric lemma generalizes \cite[Lemma 6]{Q1}:

\begin{lemma}[gap lemma]\label{geo-lem-gap}
Fix  $\lambda\ge 24$. Then there exists a constant $C(\lambda)>0$ such that for any $a=a_1a_2\cdots\in \N^\N$ and any gap $G(\lambda)\in \mathcal{G}_n^a(\lambda)$, we have
\begin{equation*}\label{gap-ratio}
	\frac{|G(\lambda)|}{|B_G(\lambda)|}\geq \frac{C(\lambda)}{a_{n+1}^3},
\end{equation*}
where $B_G(\lambda)$ is the unique band in $\mathcal{B}_n^a(\lambda)$ which contains $G(\lambda)$.
\end{lemma}	

Recall that $X_w^a(\lambda)$ is the basic set of $\Sigma_{a,\lambda}$ coded by $w$ ( see \eqref{def-basic-set}). The second geometric lemma tells us that if $a,b$ have the same tail, then the related basic sets look similar.

\begin{lemma}[tail lemma]\label{geo-lem-tail}
Assume $\lambda\ge 24$ and $a,b\in \N^\N$ are such that $S^na=S^mb$ for some $n,m\in \Z^+$.   Assume $ k\in \Z^+$ and  $u\in \Omega^a_{n+k}, v\in \Omega^{b}_{m+k}$ are  such that $\tT_u=\tT_v$, then there exists a natural bi-Lipschitz map
$\tau_{uv}:X_u^a(\lambda)\to X_v^b(\lambda).$
\end{lemma}

\subsection{The consequence of gap lemma }\

The gap lemma enables us to improve the regularity of $\pi_{a,\lambda}$ define by \eqref{pi_a-lambda}.

\begin{corollary}\label{weak-bi-lip-d_alpha}
  Assume $(a,\lambda)\in \F_1\times [24,\infty)$. Then  $\pi_{a,\lambda}: (\Omega_a,\rho_{a,\lambda})\to (\Sigma_{\check a,\lambda},|\cdot|)$ is $1$-Lipschitz and  $\pi_{a,\lambda}^{-1}$ is weak-Lipschitz.
\end{corollary}

\begin{proof}
 By Proposition \ref{metric-O-a}, $\pi_{a,\lambda}$ is 1-Lipschitz for any $(a,\lambda)\in \N^\N\times [24,\infty)$.

    Now assume $(a,\lambda)\in \F_1\times [24,\infty)$, we show that $\pi_{a,\lambda}^{-1}$ is weak-Lipschitz. Fix any $s\in (0,1)$.
     By Proposition \ref{F-proposition}(1), we have $\log a_n/n\to 0$ as $n\to\infty.$ Hence there exists $N\in \N$ such that for any $n\ge N$, we have
     \begin{equation}\label{bound-a-n+1}
         a_{n+1}^3\le 2^{n(s^{-1}-1)}.
     \end{equation}

     Fix $x,y\in \Omega_a$ with $x\ne y$. Write $\tilde x=\iota(x) $ and $\tilde y=\iota(y)$. By \eqref{Def-iota}, either $x_1\ne y_1$ and $\{\tilde x_0,\tilde y_0\}=\{\one,\three\}$, or
    $\tilde x=wx, \tilde y=wy$, where $w\in \{\one(\two,1)_1, \three(\one,1)_1\}$.
    For the first case, two points $\pi_{a,\lambda}(x), \pi_{a,\lambda}(y)$ are in $[-2,2]$ and $[\lambda-2,\lambda+2]$, respectively. Hence
    $|\pi_{a,\lambda}(x)-\pi_{a,\lambda}(y)|\ge \lambda-4.$

    Now assume the second case holds. Assume $u:=x\wedge y\in \Omega_{a,n}$. By \eqref{rho-a-lambda} and Lemma \ref{esti-band-length}, we have
    \begin{equation}\label{upbd-rho-a-lambda}
      \rho_{a,\lambda}(x,y)=|B_{wu}^{\check{a}}(\lambda)|\le  2^{1-n}.
    \end{equation}

    Since $x_{n+1}\ne y_{n+1}$, by Proposition \ref{basic-struc}, there is a gap $G(\lambda)\subset B_{wu}^{\check{a}}(\lambda)$ of order $n+1$ separating $\pi_{a,\lambda}(x)$ and $\pi_{a,\lambda}(y)$, see Figure \ref{fig-weak-lip} for an illustration. Then by Lemma \ref{geo-lem-gap},
\begin{align*}
	|\pi_{a,\lambda}(x)-\pi_{a,\lambda}(y)|&\ge |G(\lambda)|\gtrsim \frac{|B_{wu}^{\check{a}}(\lambda)|}{ \check a_{n+2}^3}=\frac{\rho_{a,\lambda}(x,y) }{ a_{n+1}^3}.
\end{align*}
If $n\le N$, then
$$
|\pi_{a,\lambda}(x)-\pi_{a,\lambda}(y)|\gtrsim\rho_{a,\lambda}(x,y).
$$
If $n\ge N$, then by \eqref{bound-a-n+1} and \eqref{upbd-rho-a-lambda}, we have
 $$
|\pi_{a,\lambda}(x)-\pi_{a,\lambda}(y)|\gtrsim\frac{\rho_{a,\lambda}(x,y) }{ 2^{n(s^{-1}-1)}}\gtrsim \frac{\rho_{a,\lambda}(x,y) }{ \rho_{a,\lambda}(x,y)^{-(s^{-1}-1)}}=\rho_{a,\lambda}(x,y)^{1/s}.
$$
Thus, $\pi_{a,\lambda}^{-1}$ is $s$-H\"older. Since  $s\in (0,1)$ is arbitrary,  $\pi_{a,\lambda}^{-1}$ is weak-Lipschitz.
 \end{proof}

\begin{figure}
		\includegraphics[scale=0.55]{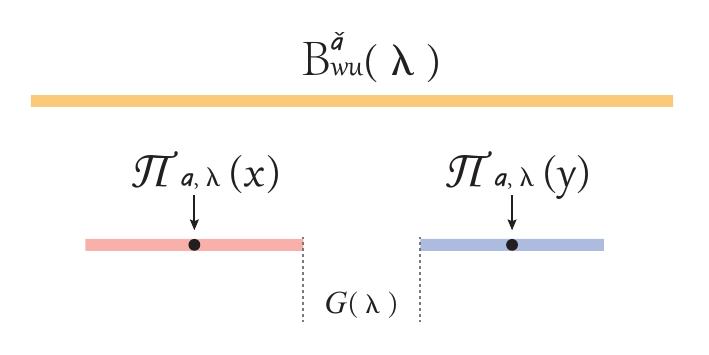}
		
		\caption{weak-Lipschitz of $\pi_{a,\lambda}^{-1}$}\label{fig-weak-lip}
 \end{figure}

The important consequence of Corollary \ref{weak-bi-lip-d_alpha} is that for typical $a,$ $\n_a$ and $\NN_{\check a,\lambda}$ have the same dimensional property.
To see this, we first show that  the image of $\n_a$ under $\pi_{a,\lambda}$ is equivalent to $\NN_{\check a,\lambda}$.
We need a characterization of the DOS similar to Lemma \ref{imp}.

\begin{proposition}\label{imp-dos}
For any $(a,\lambda)\in  \N^\N\times [24,\infty)$ and $w\in \Omega_{n}^a$, we have
\begin{equation*}\label{fiber-of-n}
	\NN_{a,\lambda}(X^a_w(\lambda))\sim
	\frac{1}{\eta_{\tT_w,a,n}\ q_{n}(a)},
\end{equation*}
where $\eta_{\tT,a,n}$ is defined by \eqref{fiber-measure}.
\end{proposition}

\begin{proof}
Let $H_n$ be the restriction of $H_{a,\lambda,0}$ to the box $[1,q_n(a)]$ with periodic boundary condition. Let $\chi_n=\{E_{n,1},\cdots,E_{n,q_n(a)}\}$ be the eigenvalues of $H_n$. It is well-known that each band in $\sigma_{n+1,0}(\lambda)$ contains exactly one point in $\chi_n$ and each point in $\chi_n$ is contained in some band of $\sigma_{n+1,0}(\lambda)$(see for example \cite{R,T}). Define
$$
\nu_n:=\frac{1}{q_n(a)}\sum_{i=1}^{q_n(a)}\delta_{E_{n,i}},
$$
where $\delta_E$ denotes the Dirac measure supported on $\{E\}$. Then $\nu_n\to\mathcal{N}_{a,\lambda}$ weakly(see for example \cite{CL}).

Now fix $w\in \Omega^a_n$.
By Lemma \ref{sigma-n+1-0}, \eqref{def-Xi-N}, Lemma \ref{number-est-basic}(5) and Lemma \ref{q-n}(1), we have
\begin{align*}
	\nu_{n+m}(B_w^a(\lambda))=& \frac{\# \chi_{n+m}\cap B_w^a(\lambda)}{q_{n+m}(a)}=\frac{\#\{u\in\Omega^a_{n+m}:u|_n=w,\ \tT_u=\two \text{ or } \three\}}{q_{n+m}(a)}\\
=&\frac{\#\{wv\in\Omega_{n+m}^a: \tT_v=\two \text{ or } \three\}}{q_{n+m}(a)}=\frac{\#\Xi(\tT_{{w}},S^na|_m,\two)+
\#\Xi(\tT_{{w}},S^na|_m,\three)}{q_{n+m}(a)}\\
\sim &\frac{\#\Xi(\tT_{{w}},S^na|_m)}{q_{n}(a)q_m(S^na)}.
\end{align*}
Let $m\to \infty$ and using Lemma \ref{number-est-basic}(1)-(4), the result follows.
\end{proof}

Now we can show what we propose:

\begin{corollary}\label{equiv-fiber-dos-coro}
Assume $(a,\lambda)\in \F_3\times [24,\infty)$.   Then
\begin{equation}\label{equiv-fiber-dos-new}
(\pi_{a,\lambda})_\ast(\n_a)\asymp \NN_{\check{a},\lambda}.
\end{equation}
Consequently,
  \begin{equation*}
    \dim_H(\n_a,\lambda)=\dim_H\NN_{\check a,\lambda};\ \ \dim_P(\n_a,\lambda)=\dim_P\NN_{\check a,\lambda},
  \end{equation*}
 and  $\n_a$ is exact lower(upper)-dimensional if and only if $\NN_{\check a,\lambda}$ is exact lower(upper)-dimensional.

\end{corollary}

\begin{proof}
Assume $w\in \Omega_{a,\ast}$.
By the definition of $\pi_{a,\lambda}$, we have $\pi_{a,\lambda}([w]_a)=X^{\check a}_{\iota(w)}(\lambda)$, where $\iota(w)$ is defined by \eqref{Def-iota-w}.
Since $a \in \F_3$, by  Proposition \ref{imp-dos}, \eqref{fiber-measure}, Lemma \ref{q-n}(1) and Lemma \ref{imp},
\begin{eqnarray*}
  \NN_{\check a,\lambda}(X^{\check a}_{\iota(w)}(\lambda))\sim\frac{1}{\eta_{\tT_{\iota(w)},\check a,n+1}\ q_{n+1}(\check a)}\sim \frac{1}{\eta_{\tT_{w}, a,n}\ q_{n}(a)}\sim \n_a([w]_a)=(\pi_{a,\lambda})_\ast(\n_a)(X^{\check a}_{\iota(w)}(\lambda)).
\end{eqnarray*}
Since
  $\{X^{\check a}_{\iota(w)}(\lambda): w\in \Omega_{a,\ast}\}
$
generates the Borel $\sigma$-algebra of  $\Sigma_{\check a,\lambda}$, \eqref{equiv-fiber-dos-new} holds.

As a consequence, $X\subset \Omega_a$ is of null-$\n_a$ measure if and only if $\pi_{a,\lambda}(X)$ is of null-$\NN_{\check a,\lambda}$ measure.
By \eqref{dim-meas} and the definitions of exact lower- and upper- dimensionality, to show the rest results, we only need to show the following claim.

\noindent {\bf Claim:} For any $x\in \Omega_a$, we have
\begin{equation}\label{equal-loc-dim-2}
\underline{d}_{\n_{a}}(x,\lambda)=\underline{d}_{\NN_{\check a,\lambda}}(\pi_{a,\lambda}(x));\ \ \
\overline{d}_{\n_{a}}(x,\lambda)=\overline{d}_{\NN_{\check a,\lambda}}(\pi_{a,\lambda}(x)).
\end{equation}

\noindent $\lhd $ Fix any $t>1$. Since $a \in \F_3$, by Corollary \ref{weak-bi-lip-d_alpha},   there exists a constant $c_1>0$ such that for any $x \in \Omega_a$ and sufficiently small $r > 0$, we have:
\begin{equation*}
B(\pi_{a,\lambda}(x),c_1r^t)\subset \pi_{a,\lambda}(B(x,r))\subset B(\pi_{a,\lambda}(x),r).
\end{equation*}
Then by \eqref{equiv-fiber-dos-new},
\begin{equation*}
\NN_{\check a, \lambda}(B(\pi_{a,\lambda}(x),c_1r^t)) \lesssim \n_a(B(x,r))\lesssim \NN_{\check a, \lambda}(B(\pi_{a,\lambda}(x),r)).
\end{equation*}
From this, we conclude:
\begin{equation*}
\underline{d}_{\NN_{\check a,\lambda}}(\pi_{a,\lambda}(x))\le \underline{d}_{\n_{a}}(x,\lambda)\le t\underline{d}_{\NN_{\check a,\lambda}}(\pi_{a,\lambda}(x)).
\end{equation*}
Since $t>1$ is arbitrary, the first equality of \eqref{equal-loc-dim-2} holds. The same reasoning shows that the second equality also holds.
\hfill $\rhd$

So the rest results hold.
\end{proof}

\subsection{The consequence of tail lemma}\

The tail lemma  has two consequences. The first one is an analog of Corollary \ref{equiv-fiber-dos-coro}.

\begin{proposition}\label{equiv-different-dos}
Assume $\lambda\in [24,\infty)$ and $a,b\in \N^\N$ with $S^na=S^mb$ for some $n,m\in \Z^+$. Assume $k\in \Z^+$ and  $u\in \Omega^a_{n+k}, v\in \Omega^{b}_{m+k}$ with $\tT_u=\tT_v$. Then
\begin{equation*}\label{equiv-dos}
	(\tau_{uv})_*(\NN_{a,\lambda}|_{X_u^a(\lambda)})\asymp \NN_{b,\lambda}|_{X_v^b(\lambda)}.
\end{equation*}
In particular, for any $x\in X_u^a(\lambda)$, we have
\begin{equation*}\label{equal-loc-dim}
	\underline{d}_{\NN_{a,\lambda}}(x)=\underline{d}_{\NN_{b,\lambda}}(\tau_{uv}(x));\ \ \
	\overline{d}_{\NN_{a,\lambda}}(x)=\overline{d}_{\NN_{b,\lambda}}(\tau_{uv}(x)).
\end{equation*}
\end{proposition}

\begin{proof}
  By Proposition \ref{imp-dos} and Lemma \ref{geo-lem-tail}, the proof is the same as that of Corollary \ref{equiv-fiber-dos-coro}.
\end{proof}

The second consequence is

\begin{proposition}\label{loc-bi-lip}
For any $(a,\lambda)\in \N^\N\times[24,\infty)$, one can write $\Sigma_{a,\lambda}$ and $\Sigma_{Sa,\lambda}$ as disjoint unions of basic sets:
\begin{equation*}
	\Sigma_{a,\lambda}=\bigsqcup_{i=1}^k X_i;\ \  \Sigma_{Sa,\lambda}=\bigsqcup_{j=1}^l Y_j,
\end{equation*}
and find two maps $\phi:\{1,\cdots,k\}\to \{1,\cdots,l\}$, $\psi:\{1,\cdots,l\}\to \{1,\cdots,k\}$ such that $X_i$ and $Y_{\phi(i)}$ are bi-Lipschitz equivalent for any $i=1,\cdots,k$ and
$X_{\psi(j)}$ and $Y_{j}$ are bi-Lipschitz equivalent for any $j=1,\cdots,l.$
\end{proposition}

\begin{proof}
Write $b:=Sa$, then $S^1a=S^0b$. By \eqref{spectrum-basic-set}, we have
$$
\Sigma_{a,\lambda}=\bigsqcup_{u\in \Omega^a_7}X^a_u(\lambda); \ \ \ \ \ \ \ \Sigma_{Sa,\lambda}=\bigsqcup_{v\in \Omega^b_6}X^b_v(\lambda).
$$
By Proposition \ref{sturm-primitive}, we have
$$
\{\tT_u: u\in \Omega^a_7\}=\{\tT_v: v\in \Omega^b_6\}=\{\one,\two,\three\}.
$$

Fix any $u\in \Omega^a_7$, find some $v\in \Omega^b_6$ such that $\tT_v=\tT_u$. Since $Sa=b$, by Lemma \ref{geo-lem-tail}, we conclude that $\tau_{uv}: X^a_u(\lambda)\to X^{b}_v(\lambda)$ is  bi-Lipschitz.		

Conversely, fix any $v\in \Omega^b_6$, we can find some $u\in \Omega^a_7$ such that $\tT_u=\tT_v$. Then still by Lemma \ref{geo-lem-tail}, $\tau_{vu}:X_{v}^b(\lambda)\to X_{u}^{a}(\lambda)$ is  bi-Lipschitz.
\end{proof}

\begin{remark}
{\rm
Damanik and Gorodetski stated the following conjecture in \cite{DG15}:

 \noindent {\bf  Conjecture:} For any $a\in \N^\N$ and $\lambda>0$, there exist two  neighborhoods $U\supset \Sigma_{a,\lambda}$ and $V\supset \Sigma_{Sa,\lambda}$ and a $C^1$-diffeomorphism $f:U\to V$ such that $f(\Sigma_{a,\lambda})= \Sigma_{Sa,\lambda}$.

Proposition \ref{loc-bi-lip} provides a weak answer for this Conjecture: for $\lambda\ge 24,$ $\Sigma_{a,\lambda}$ and $\Sigma_{Sa,\lambda}$ are locally bi-Lipschitz equivalent.
}
\end{remark}

Combining Propositions \ref{equiv-different-dos} and  \ref{loc-bi-lip}, we get

\begin{corollary}\label{tail-eqal-dim}
  Assume $(a,\lambda)\in \N^\N\times [24,\lambda)$. Then
  \begin{equation*}
    \dim_H \NN_{a,\lambda}=\dim_H\NN_{\check a,\lambda};\ \ \dim_P \NN_{a,\lambda}=\dim_P\NN_{\check a,\lambda},
  \end{equation*}
 and  $\NN_{a,\lambda}$ is exact lower(upper)-dimensional if and only if $\NN_{\check a,\lambda}$ is exact lower(upper)-dimensional.
\end{corollary}

\begin{proof}
  By Proposition \ref{loc-bi-lip},  we can write
\begin{equation*}
	X:=\Sigma_{\check a,\lambda}=\bigsqcup_{i=1}^k X_i;\ \  Y:=\Sigma_{a,\lambda}=\bigsqcup_{j=1}^l Y_j,
\end{equation*}
and find two maps $\phi:\{1,\cdots,k\}\to \{1,\cdots,l\}$, $\psi:\{1,\cdots,l\}\to \{1,\cdots,k\}$ such that $X_i$ and $Y_{\phi(i)}$ are bi-Lipschitz equivalent   and
$X_{\psi(j)}$ and $Y_{j}$ are bi-Lipschitz equivalent. Write $\mu:=\NN_{\check a,\lambda}$ and $\nu:=\NN_{a,\lambda}$. By \eqref{dim-meas} and  Proposition \ref{equiv-different-dos},
\begin{align*}
  \dim_H \mu =& {\rm essinf}_{x\in X} \ \underline{d}_\mu(x)=\min_{1\le i\le k}{\rm essinf}_{x\in X_i} \ \underline{d}_\mu(x)=\min_{1\le i\le k}{\rm essinf}_{y\in Y_{\phi(i)}} \ \underline{d}_\nu(y)\\
  \ge&\min_{1\le j\le l}{\rm essinf}_{y\in Y_{j}} \ \underline{d}_\nu(y)={\rm essinf}_{y\in Y} \ \underline{d}_\nu(y)=\dim_H\nu.
\end{align*}
By using $\psi$, we get $\dim_H\nu\ge \dim_H\mu.$ So $\dim_H\nu= \dim_H\mu.$

Assume $\mu$ is exact lower-dimensional. Then there exists $s>0$ such that
$$
\mu(Z)=0, \text{ where } \ Z:=\{x\in X: \underline{d}_\mu(x)\ne s\}.
$$
Write $\hat Z:=\{y\in Y: \underline{d}_\nu(y)\ne s\}$ and
$$
Z_i:=\{x\in X_i: \underline{d}_\mu(x)\ne s\};\ \ \  \hat Z_j=\{y\in Y_j: \underline{d}_\nu(y)\ne s\}.
$$
Since $Z=\bigcup_{1\le i\le k}Z_i$, we have $\mu(Z_i)=0$. For any $j=1,\cdots,l$, by Proposition \ref{equiv-different-dos}, we have $\nu(\hat Z_j)=\mu(Z_{\psi(j)})=0$. So
$$\nu(\hat Z)=\nu(\bigcup_{1\le j\le l}\hat Z_j)=0.$$
That is, $\nu$ is exact lower-dimensional. By using $\phi$ we can show that $\nu$ is exact lower-dimensional implies  $\mu$ is exact lower-dimensional.

The proofs of the other two statements are the same.
\end{proof}

\subsection{Proof of Theorem \ref{main-dim-dos}}\

Now we are ready for the proof of the second main result.

\begin{proof}[Proof of Theorem \ref{main-dim-dos}]
Define $d$ by \eqref{def-dim-dos}. Let $\widehat{\F}$ be the set in Corollary \ref{h=p}.
Define
\begin{equation}\label{def-hat-I}
\hat \II:=\Theta^{-1}(\widehat \F).
\end{equation}
Since $\widehat \F$ is of  full $\Gg$-measure,  $\hat \II$ is of full $G$-measure, hence full Lebesgue measure.

(1) Fix any $(\alpha,\lambda)\in \hat \II\times [24,\infty)$. Write $a=\Theta(\alpha)$,
by Corollary \ref{h=p}, $\n_a$ is exact-dimensional and \eqref{fix-lambda-strong} holds.
By Corollaries \ref{equiv-fiber-dos-coro} and \ref{tail-eqal-dim},
$\NN_{\alpha,\lambda}$ is exact-dimensional
and the equalities hold.

(2) It follows from the definition \eqref{def-dim-dos}.

(3)  This is the content of Corollary \ref{lip-d}.
\end{proof}

\section{Proofs of technical results}\label{sec-proof-tech}

In this section, we prove several technical results from Sections \ref{sec-proof-main-spec} and \ref{sec-dim-fiber-symbolic}.

\subsection{Proof of Lemma \ref{additive-on-A}}\label{sec-proof-sub-additive}\

The difficult part of Lemma \ref{additive-on-A} is the almost sub-additivity of $\{\QQ(\cdot,\lambda,t,n):n\in \N\}$. The rough idea of proof is as follows. At first, we write  the exponential of $\QQ(a,\lambda,t,n+m)$ in an equivalent form  as \eqref{sub-additive}. Then we use the bounded covariation property (see Proposition \ref{bco}) to prove the claim below and then finish the proof. To prove the claim, we need to compare the terms in both sides of the inequality. It is enough to find a correspondence $\zeta$ between $\Xi(\tT_u,b|_m)$ and $ \Omega^b_m$ and prove an estimate like \eqref{cov-pre}. We need to discuss several cases.   In each case, the construction of $\zeta$ is more or less natural, the proofs of \eqref{cov-pre} is  also not hard, except one case.
To prove that special case, we need the following lemma, which is a direct consequence of Proposition \ref{bv} and \cite[Lemma 3.6]{LQW}:

\begin{lemma}\label{father-son}
  Assume  $a\in \N^\N$ and $\lambda\ge 24$. Then there exist constants $\eta_2(\lambda)\ge\eta_1(\lambda)>0$ such that if $w\in \Omega_n^{a}$ and $e=(\mathbf t,l)_{a_{n+1}}$ with $w\to e$ and $\mathbf t\ne \mathbf 2$,
then
$$
\frac{\eta_1}{p} \sin^2\frac{l\pi}{p}\le\frac{|{B}_{we}^a(\lambda)|}{|{B}_{w}^a(\lambda)|}\le \frac{\eta_2}{p}\sin^2\frac{l\pi}{p},\ \ \ \text{ where }\ \ \
p=\begin{cases}
    a_{n+1}+2, & \mbox{if } \mathbf t_w\mathbf t=\mathbf 2\mathbf 1, \\
    a_{n+1}+1, & \mbox{if } \mathbf t_w\mathbf t\in \{\mathbf 2\mathbf 3,\mathbf 3\mathbf 1\},  \\
    a_{n+1}, & \mbox{if } \mathbf t_w\mathbf t=\mathbf 3\mathbf 3.
  \end{cases}$$
\end{lemma}

\begin{proof}[Proof of Lemma \ref{additive-on-A}]
At first  we show that $\QQ^+(\cdot,\lambda,t,1)\in L^1(\Gg)$.
Assume $a\in \N^\N$ with $a_1=k$, then $\#\Omega_{1}^a=2k$. For any $t\ge0$, by Lemma \ref{esti-band-length}, we have
\begin{equation*}
\QQ(a,\lambda,t,1)=\log\sum_{w\in \Omega_{1}^a}|B^a_w(\lambda)|^t\leq\log(\#\Omega_1^a\cdot 2^t)=\log2k+t\log2.
\end{equation*}
By Lemma \ref{q-n}(2), we have $\Gg([k]) \lesssim k^{-2}.$ Therefore,
\begin{eqnarray*}
\int_{\N^\N}\QQ^+(a,\lambda,t,1)d\Gg(a) \le  \sum_{k=1}^{\infty}\Gg([k])\left(\log2k+t\log2\right)<\infty.
\end{eqnarray*}		

Next we show that $\{\QQ(\cdot,\lambda,t,n):n\in \N\}$ are almost sub-additive.
Fix $a\in\N^\N$ and write $b=S^na$. We have
\begin{align}\label{sub-additive}
	\sum_{w\in\Omega_{n+m}^a}|B^a_w(\lambda)|^t=\sum_{u\in\Omega_{n}^a}\sum_{v\in \Xi(\tT_u,b|_m)}|B^a_{uv}(\lambda)|^t
	=\sum_{u\in\Omega_{n}^a}|B^a_u(\lambda)|^t\sum_{v\in \Xi(\tT_u,b|_m)}\frac{|B^{{a}}_{uv}(\lambda)|^t}{|B^a_u(\lambda)|^t} .
\end{align}

{\bf Claim}: There exists constant $ C>0$ such that  for all $u\in\Omega_n^a$ and $m\geq1$,
$$\sum_{v\in \Xi(\tT_u,b|_m)}\frac{|B^{{a}}_{uv}(\lambda)|^t}{|B^a_u(\lambda)|^t}\leq C\sum_{v\in\Omega^b_m}|B_v^b(\lambda)|^t.$$

\noindent $\lhd$ Since $t\ge0$,
it suffices to show that there exist a constant $C_1>0$ and  a map $\zeta:\Xi(\tT_u,b|_m)\to \Omega^b_m$  such that $\zeta$ is at most $2$-to-$1$ and for any $v\in \Xi(\tT_u,b|_m)$, we have
\begin{equation}\label{cov-pre}
  \frac{|B^{{a}}_{uv}(\lambda)|}{|B^a_u(\lambda)|}\le  C_1 |B_{\zeta(v)}^b(\lambda)|.
\end{equation}

We discuss three cases:

{\it Case 1:} $\mathbf{t}_u=\mathbf{1}$ or $\mathbf{3}$. Define $\zeta$ as $\zeta(v):=\tT_u v.$ By \eqref{def-Omega^alpha_n} and \eqref{admissible-T-A}, $\zeta$ is 1-to-1 and $\zeta(v)\in \Omega^b_m$.  By \eqref{0}, we see that $|B_{\mathbf t_u}^b(\lambda)|=4$. Combine with Proposition \ref{bco},
\begin{align*}
	 \frac{|B^{{a}}_{uv}(\lambda)|}{|B^a_u(\lambda)|}\leq \eta\frac{|B_{\tT_uv}^b(\lambda)|}{|B_{\tT_u}^b(\lambda)|}\leq \eta |B_{\zeta(v)}^b(\lambda)|.
\end{align*}

{\it Case 2:} $\mathbf{t}_u=\mathbf{2}$ and $b_1=1$. By \eqref{admissible-T-A}, for any $v=v_1\cdots v_m\in \Xi(\tT_u,b|_m)$, we have
$$
v_1\in \{(\mathbf 1,1)_{1},(\mathbf 1,2)_{1},(\mathbf 3,1)_{1}\}.
$$
Define $\zeta$ as follows:
\begin{equation*}
  \zeta(v):=
  \begin{cases}
     \mathbf{3}(\mathbf 1,1)_{1}v_2\cdots v_m,& \mbox{if } v_1= (\mathbf 1,1)_{1} \text{ or } (\mathbf 1,2)_{1},\\
    \mathbf{1}(\mathbf 2,1)_{1}v_2\cdots v_m,& \mbox{if } v_1= (\mathbf 3,1)_{1}.
  \end{cases}
\end{equation*}
It is clear that $\zeta$ is at most 2-to-1, and by \eqref{admissible-T-A}, $\zeta(v)\in \Omega^b_m$.

If $v_1=(\mathbf 1,1)_{1}$ or $(\mathbf 1,2)_{1}$, write $v=v_1w$, then by Proposition \ref{bco} and \eqref{band-length-esti},
\begin{eqnarray*}
  \frac{|B^{{a}}_{uv}(\lambda)|}{|B^a_u(\lambda)|} = \frac{|B^{{a}}_{uv_1w}(\lambda)|}{|B^a_u(\lambda)|}
  \le\frac{|B^{{a}}_{uv_1w}(\lambda)|}{|B^a_{uv_1}(\lambda)|}\le\eta
  \frac{|B^{{b}}_{\mathbf 3(\mathbf 1,1)_{1}w}(\lambda)|}{|B^b_{\mathbf 3(\mathbf 1,1)_{1}}(\lambda)|}\le \tau_2\eta|B^b_{\zeta(v)}(\lambda)|.
\end{eqnarray*}

If $v_1=(\mathbf 3,1)_{1}$,   then by Proposition \ref{bco} and \eqref{band-length-esti} and the same argument as above, we conclude that \eqref{cov-pre} holds.

{\it Case 3:} $\mathbf{t}_u=\mathbf{2}$ and $b_1\geq2$. By \eqref{admissible-T-A}, for any $v=v_1\cdots v_m\in \Xi(\tT_u,b|_m)$, we have
$$
v_1\in \{(\mathbf 1,k)_{b_1}:k=1,\cdots,b_1+1\}\cup\{(\mathbf 3,k)_{b_1}:k=1,\cdots,b_1\}.
$$

Define $\zeta$ as follows:
\begin{equation*}
  \zeta(v):=
  \begin{cases}
    \mathbf{3}v, & \mbox{if } v_1 \notin \{(\mathbf 1,b_1+1)_{b_1},(\mathbf 3,b_1)_{b_1} \},\\
     \mathbf{3}(\mathbf 1,b_1)_{b_1}v_2\cdots v_m,& \mbox{if } v_1= (\mathbf 1,b_1+1)_{b_1},\\
    \mathbf{3}(\mathbf 3,b_1-1)_{b_1}v_2\cdots v_m,& \mbox{if } v_1= (\mathbf 3,b_1)_{b_1}.
  \end{cases}
\end{equation*}
It is clear that $\zeta$ is at most 2-to-1, and by \eqref{admissible-T-A}, $\zeta(v)\in \Omega^b_m$.

If $v_1 \notin \{(\mathbf 1,b_1+1)_{b_1},(\mathbf 3,b_1)_{b_1} \}$, the same argument as Case 1 shows that \eqref{cov-pre} holds.

If $v_1=(\mathbf 1,b_1+1)_{b_1}$, write $v=v_1w$, then by Proposition \ref{bco} and Lemma \ref{father-son},
\begin{align*}
  \frac{|B^{{a}}_{uv}(\lambda)|}{|B^a_u(\lambda)|} =& \frac{|B^{{a}}_{uv_1w}(\lambda)|}{|B^a_u(\lambda)|}
  =\frac{|B^{{a}}_{uv_1w}(\lambda)|}{|B^a_{uv_1}(\lambda)|}
  \frac{|B^{{a}}_{uv_1}(\lambda)|}{|B^a_u(\lambda)|}\le\eta
  \frac{|B^{{b}}_{\mathbf 3(\mathbf 1,b_1)_{b_1}w}(\lambda)|}{|B^b_{\mathbf 3(\mathbf 1,b_1)_{b_1}}(\lambda)|}\frac{|B^{{a}}_{uv_1}(\lambda)|}{|B^a_u(\lambda)|}\\
  =&\eta
  \frac{|B^{{b}}_{\zeta(v)}(\lambda)|}{|B^b_{\mathbf 3 }(\lambda)|}\frac{|B^{{b}}_{\mathbf 3 }(\lambda)|}{|B^b_{\mathbf 3(\mathbf 1,b_1)_{b_1}}(\lambda)|}\frac{|B^{{a}}_{uv_1}(\lambda)|}{|B^a_u(\lambda)|}\le \eta |B^{{b}}_{\zeta(v)}(\lambda)|\frac{\frac{\eta_2}{b_1+2}\sin \frac{(b_1+1)\pi}{b_1+2}}{\frac{\eta_1}{b_1+1}\sin \frac{b_1\pi}{b_1+1}}\\
  \le& \frac{\eta\eta_2}{\eta_1} |B^{{b}}_{\zeta(v)}(\lambda)|.
\end{align*}

If $v_1=(\mathbf 3,b_1)_{b_1}$, the same argument shows that \eqref{cov-pre} holds.
\hfill $\rhd$

By \eqref{sub-additive} and the  claim above, we have
\begin{align*}
	\sum_{w\in\Omega_{n+m}^a}|B^a_w(\lambda)|^t
	 \leq C\sum_{u\in\Omega_{n}^a}|B^a_u(\lambda)|^t
	\sum_{v\in \Omega_m^b}|B^{b}_{v}(\lambda)|^t.
\end{align*}
This implies that $\{\QQ(\cdot,\lambda,t,n):n\in \N\}$ are almost sub-additive.		
\end{proof}

\subsection{Proof of Lemma \ref{number-est-basic}}\label{sec-cardinality}\

To prove the lemma, we need the following result which is implicitly proven in \cite{R}:

\begin{lemma}[\cite{R}]\label{Raymond}
For any $\tT,\tT'\in \T$ and  $\vec a=a_1\cdots a_n\in\N^n$, we have
\begin{equation*}\label{N-T-a}
	\#\Xi(\tT,\vec a)=v_{\tT}\hat A_{a_{1}}\cdots \hat A_{a_{n}}v_\ast^T;\ \ \ \#\Xi(\tT,\vec a,\tT')=v_{\tT}\hat A_{a_{1}}\cdots \hat A_{a_{n}}v_{\tT'}^T,
\end{equation*}
where the vectors are defined as
$$
v_\one:=(1,0,0);\ \  v_\two:=(0,1,0);\ \  v_\three:=(0,0,1);\ \ v_\ast:=(1,1,1),
$$
and the matrix $\hat A_k$ is defined as
\begin{equation}\label{hat-A-n}
	\hat A_k=
	\left[\begin{array}{ccc}
		0&1&0\\
		k+1&0&k\\
		k&0&k-1
	\end{array}\right].
\end{equation}
\end{lemma}	

We remark that $\hat A_k$ is the incidence matrix of the directed  graph in Figure \ref{type-evo}.

\begin{proof}[Proof of Lemma \ref{number-est-basic}]

Write $\vec b:=a_2\cdots a_n$ and $\vec c:=a_3\cdots a_n$. By Lemma \ref{q-n}(1), we have
\begin{equation}\label{imp-q-n}
	a_nq_{n-1}(\vec a)\sim q_n(\vec a)\sim a_1q_{n-1}(\vec b)\sim a_1a_2q_{n-2}(\vec c).
\end{equation}
By the first equation of Lemma \ref{Raymond},  we have the following recurrence relations:
\begin{align}
	\#\Xi(\one,\vec a)=&(0,1,0)\hat A_{a_{2}}\cdots \hat A_{a_n}v_\ast^T=\#\Xi(\two,\vec b),\label{1}\\
	\#\Xi(\two,\vec a)
	=&(a_{1}+1)\#\Xi(\one,\vec b)+a_{1}\#\Xi(\three,\vec b),\label{2}\\
	\#\Xi(\three,\vec a)
	=&a_{1}\#\Xi(\one,\vec b)+(a_{1}-1)\#\Xi(\three,\vec b).\label{3}
\end{align}
We start with the proof of (2).

(2) Notice that $
\Omega^{\vec a}_{n}=\{\one v:v\in\Xi(\one,\vec a)\}\cup\{\three v:v\in\Xi(\three,\vec a)\}.
$ By Lemma \ref{number-Omega-n-alpha}, we have
\begin{align*}
	\#\Xi(\one,\vec a)+\#\Xi(\three,\vec a)=\#\Omega^{\vec a}_n\sim	q_n(\vec a).
\end{align*}
Combine with  \eqref{2} and \eqref{imp-q-n}, we have
\begin{equation}\label{imp-1}
	\#\Xi(\two,\vec a) \sim a_1q_{n-1}(\vec b)\sim q_n(\vec a).
\end{equation}

(1)	By \eqref{1}, \eqref{imp-1} and \eqref{imp-q-n}, we get
\begin{equation*}
	\#\Xi(\one,\vec a)=\#\Xi(\two,\vec b)\sim q_{n-1}(\vec b)\sim\frac{q_n(\vec a)}{a_1}.
\end{equation*}

(3)	If $a_1\geq2$, by \eqref{2}, \eqref{3} and \eqref{imp-1}, then
\begin{equation*}
	\#\Xi(\three,\vec a)\sim \#\Xi(\two,\vec a)\sim q_n(\vec a).
\end{equation*}

(4)	If $a_1=1$, by \eqref{3}, \eqref{1}, \eqref{imp-1} and \eqref{imp-q-n}, then
\begin{equation*}
	\#\Xi(\three,\vec a)=\#\Xi(\one,\vec b)=\#\Xi(\two,\vec c)\sim q_{n-2}(\vec c)\sim\frac{q_n(\vec a)}{a_2}.
\end{equation*}

(5) One can check directly that
$$
\hat A_{a_{n-1}}\hat A_{a_n}(v_\two^T+v_\three^T)\ge \frac{1}{2}\hat A_{a_{n-1}}\hat A_{a_n}v_\one^T.
$$
Using the second equation of Lemma \ref{Raymond}, for any $\tT\in\T$, we have
$$
\#\Xi(\tT,\vec a,\two)+\#\Xi(\tT,\vec a,\three) \ge \frac{\#\Xi(\tT,\vec a,\one)}{2}.
$$
 Combine with  the fact:
$
\#\Xi(\tT,\vec a,\one)+\#\Xi(\tT,\vec a,\two)+\#\Xi(\tT,\vec a,\three)=\#\Xi(\tT,\vec a),$
 (5) holds.

(6) Write $\vec a_\ast:=a_1\cdots a_{n-1}.$	By    Lemma \ref{Raymond}, the statement (5), \eqref{imp-1} and \eqref{imp-q-n}, we have
\begin{align*}
	\#\Xi(\two,\vec a,\one)&=(a_{n}+1)\#\Xi(\two,\vec a_\ast,\two)+a_n\#\Xi(\two,\vec a_\ast,\three)\\
	&
	\sim a_n\cdot\#\Xi(\two,\vec a_\ast)\sim a_nq_{n-1}(\vec a)\sim q_{n}(\vec a).
\end{align*}
  Hence  (6) holds.
\end{proof}

\subsection{ Proof of Proposition \ref{exact-H-P-dim}}\label{sec-proof-exact-dim}\

Although the proof is a bit tricky, the idea behind is not that hard. Let us  give a sketch of the proof. We take the exact lower-dimensionality as example. Assume $\n_a$ is not exact lower-dimensional, then we can find two subsets $X,Y\subset \Omega_a$ of positive measures such that for any $x\in X, y\in Y$ we have  $\underline{d}_{\n_a}(y,\lambda)>\underline{d}_{\n_a}(x,\lambda).$   On the other hand, $\Omega_a$ has the Besicovitch's covering property. So we can find $x\in X$ and $y\in Y$ such that they satisfy the density properties with respect to $\n_a$. Then we use the special structure of $\Omega_a$ and the fact that $\rho_{a,\lambda}$ is a ultra-metric to conclude that there exist some $\tilde x\in X$ quite close to $x$ and some $\tilde y\in Y$ quite close to $y$ such that $\tilde x$ and $\tilde y$ have the same tail. Now by using Lemma \ref{loc-dim-cylinder} and the almost-additivity of $\Psi_\lambda$, we conclude that $\underline{d}_{\n_a}(\tilde x,\lambda)=\underline{d}_{\n_a}(\tilde y,\lambda)$, a contradiction.

\begin{proof}[Proof of Proposition \ref{exact-H-P-dim}]

Fix $(a,\lambda)\in \F_3\times[24,\infty)$. We show that $\n_a$ is exact lower-dimensional.

We prove it by contradiction.
Assume $\n_a$ is not exact lower-dimensional. Then there exist two disjoint Borel subsets $X, Y\subset\Omega_{a}$ with $\n_a(X), \n_a(Y )>0$ and two numbers $0\le d_1<d_2$ such that
\begin{equation*}
\underline{d}_{\n_a}(x,\lambda)\leq d_1\ \ (\forall x\in X)\ \ \ {\rm and}\ \ \ \underline{d}_{\n_a}(y,\lambda)\geq d_2\ \ (\forall y\in Y).
\end{equation*}
The space $\Omega_{a}$ has Besicovitch's covering property (see for example \cite[Lemma 2.2]{Q2}). By the density property for Radon measure (see for example \cite[Chapter 2]{Mat}), there exist $x\in X$ and $y\in Y$ such that
\begin{equation}\label{limit-1}
\lim_{n\to\infty}\frac{\n_a(X\cap [x|_n]_{a})}{\n_a([x|_n]_{a})}=1\ \ \  {\rm and}\ \ \ \lim_{n\to\infty}\frac{\n_a(Y\cap [y|_n]_{a})}{\n_a([y|_n]_{a})}=1.
\end{equation}

Since $a\in\F_3$, by Lemma \ref{imp}, there exists a constant $C>1$ such that for any $w\in \Omega_{a,n}$,
\begin{equation}\label{measure-of-cylinder}
C^{-1}\le \n_a([w]_a)\eta_{\tT_w,a,n}\ q_{n}(a)\le C.
\end{equation}

Since $a\in\F_3\subset\F_1$, by Proposition \ref{F-proposition}(2), there exists a sequence $n_k\uparrow \infty$, such that for any $k\in \N$,
\begin{equation}\label{se}
a_{n_k+i}=1,\ \ \ \ i=1,2,\cdots,7.
\end{equation}
Write $\delta_1:=(2^9 C^6)^{-1}$. By \eqref{limit-1}, assume $k$ is large enough such that for $m:=n_k$,
\begin{equation}\label{Dos-1}
\begin{cases}
	\n_a(X\cap[x|_m]_{a})\geq(1-\delta_1)\n_a([x|_m]_{a}),\\
	\n_a(Y\cap[y|_m]_{a})\geq(1-\delta_1)\n_a([y|_m]_{a}).
\end{cases}
\end{equation}

\vspace{1ex}
Recall  that $\Xi_{a,m}(w)$ is defined by \eqref{des-w}. Write $\delta_2:=(2^{3}C^{4})^{-1}<1/8$.
\vspace{1ex}

{\noindent \bf Claim 1}:  For any $v\in\Xi_{a,5}(x|_m)$ and $w\in\Xi_{a,5}(y|_m)$, we have
\begin{equation}\label{Dos-2}
\begin{cases}
	\n_a(X\cap[v]_{a})\geq(1-\delta_2)\n_a([v]_{a}),\\
	\n_a(Y\cap[w]_{a})\geq(1-\delta_2)\n_a([w]_{a}).
\end{cases}
\end{equation}

\noindent $\lhd$
We only show the first inequality of \eqref{Dos-2}, the proof of the second one is the same. Notice that  $[x|_m]_{a}=\bigcup_{v\in\Xi_{a,5}(x|_m)}[v]_{a}$. Define
$$
\xi:=\max\left\{\frac{\n_a([v]_{a}\setminus X)}{\n_a([v]_{a})}:\ v\in\Xi_{a,5}(x|_m)\right\}.
$$
If $\xi=0$, the result holds trivially. So we assume $\xi> 0$ and $\hat{v}\in\Xi_{a,5}(x|_m)$  attains the maximum. Consequently by \eqref{Dos-1},
$$\xi\n_a([\hat{v}]_{a})=\n_a([\hat{v}]_{a}\backslash X)\leq\n_a([x|_m]_{a}\backslash X)\leq\delta_1\n_a([x|_m]_{a}).$$	
Since $m=n_k$ and $\delta_1=(2^9 C^6)^{-1}$, by \eqref{measure-of-cylinder}, \eqref{se} and Lemma \ref{q-n}(1),
$$\xi\leq\delta_1\frac{\n_a([x|_m]_{a})}{\n_a([\hat{v}]_{a})}\leq C^2\delta_1\frac{q_{m+5}(a)}{q_{m}(a)}\leq2C^2\delta_1 q_5(S^ma)\le 2^6C^2\delta_1=\delta_2.$$
Then the result follows.
\hfill $\rhd$

By Proposition \ref{sturm-primitive}, we can choose $v\in\Xi_{a,5}(x|_m)$ and $w\in\Xi_{a,5}(y|_m)$ such that $v_{m+5}=w_{m+5}$. In particular,  $\tT_v=\tT_w$. We fix such a pair $(v, w)$.

{\noindent \bf Claim 2}: There exist $\tilde{x}\in[v]_{a}\cap X$ and $\tilde{y}\in[w]_{a}\cap Y$ such that $\tilde{x}$ and $\tilde{y}$ have the same tail, i.e., there exist $l\geq m+5$ and $z\in\prod_{j=l+1}^{\infty}\mathcal{A}_{a_j}$ such that $\tilde{x}=\tilde{x}|_l\cdot z$
and $\tilde{y}=\tilde{y}|_l\cdot z$.

\noindent $\lhd$
We show it by contradiction. Assume for any $\tilde{x}\in[v]_{a}\cap X$ and $\tilde{y}\in[w]_{a}\cap Y$, $\tilde{x}$ and $\tilde{y}$ have no the same
tail. Since $\n_a$ is a Borel probability measure, by the first inequality of Claim 1, there exists a compact set $\hat{X}\subset X\cap [v]_{a}$ such that
$$\n_a(\hat{X})>(1-2\delta_2)\n_a([v]_{a}).$$
Notice that $[v]_{a}\backslash\hat{X}$ is an open set, thus it is a
countable disjoint union of cylinders: $$[v]_{a}\setminus\hat{X}=\bigcup_{j\geq1}[vw_j]_{a},$$
where different $w_j$ are non compatible.  Thus we conclude that
\begin{equation}\label{Dos-3}
	\sum_{j\geq1}\frac{\n_a([vw_j]_{a})}{\n_a([v]_{a})}=\frac{\n_a([v]_{a}\backslash\hat{X})}{\n_a([v]_{a})}\leq2\delta_2=2^{-2}C^{-4}.
\end{equation}

By the assumption, any $\tilde{x}\in[v]_{a}\cap X$ and $\tilde{y}\in[w]_{a}\cap Y$ have no the same
tail, and $v_{m+5}=w_{m+5}$, we must have
\begin{equation}\label{Dos-4}
	[w]_{a}\cap Y\subset\bigcup_{j\geq1}[ww_j]_{a}.
\end{equation}
Since $\tT_v=\tT_w$ and $\tT_{vw_j}=\tT_{ww_j}$, by	\eqref{measure-of-cylinder}, we get
\begin{equation}\label{Dos-5}
	C^{-4}\leq\frac{\n_a[ww_j]_{a}}{\n_a([w]_{a})}/\frac{\n_a[vw_j]_{a}}{\n_a([v]_{a})}\leq C^4.
\end{equation}
Thus by \eqref{Dos-4}, \eqref{Dos-5} and \eqref{Dos-3}, we have
$$\frac{\n_a([w]_{a}\cap Y)}{\n_a([w]_{a})}\leq\sum_{j\geq1}\frac{\n_a([ww_j]_{a})}{\n_a([w]_{a})}\le C^4\sum_{j\geq1}\frac{\n_a([vw_j]_{a})}{\n_a([v]_{a})}
\leq\frac{1}{4}<\frac{1-\delta_2}{2},$$
which contradicts with the second inequality of \eqref{Dos-2}.
\hfill $\rhd$

Claim 2 leads to a contradiction. Indeed by Lemma \ref{loc-dim-cylinder}, we have
$$\underline{d}_{\n_a}(\tilde{x},\lambda)=
\frac{\gamma}{\limsup\limits_{n\to\infty}-\psi_{\lambda,n}(\tilde{x})/n }\ \ \ \ {\rm and}\ \ \ \ \underline{d}_{\n_a}(\tilde{y},\lambda)=
\frac{\gamma}{\limsup\limits_{n\to\infty}-\psi_{\lambda,n}(\tilde{y})/n }.$$
By the almost-additivity of $\Psi_\lambda$ (see Lemma \ref{Psi-aa}), we have
\begin{align*}
|\psi_{\lambda,l+n}(\tilde x)-\psi_{\lambda,l}(\tilde x)-\psi_{\lambda,n}(\sigma^l\tilde x)|,\ \ \
|\psi_{\lambda,l+n}(\tilde y)-\psi_{\lambda,l}(\tilde y)-\psi_{\lambda,n}(\sigma^l\tilde y)|\lesssim1.
\end{align*}
By Claim 2, $\sigma^l\tilde x=z=\sigma^l\tilde y$, and hence $\underline{d}_{\n_a}(\tilde{x},\lambda)=\underline{d}_{\n_a}(\tilde{y},\lambda)$. However, since $\tilde{x}\in X$ and $\tilde{y}\in Y$, we also have $\underline{d}_{\n_a}(\tilde{x},\lambda)\leq d_1<d_2\leq \underline{d}_{\n_a}(\tilde{y},\lambda)$, which is a contradiction. Thus we conclude that $\n_a$ is  exact lower-dimensional.

The same proof shows that $\n_a$ is also  exact upper-dimensional.
\end{proof}

 \subsection{Proof of Proposition \ref{lip-fiber} and Lemma \ref{asym-psi-lambda-n} }\label{sec-proof-conti-d}\

 To prove these properties, we need certain control on the size of $m(u)$ (see \eqref{def-m-u}) which appears in Proposition \ref{bco-cor} and  bounds for the local upper and lower Lyapunov exponents. This motivates the following definitions.

 Define  a function $\vartheta: \Omega\to [0,\infty) $ as
\begin{equation}\label{def-var-theta}
  \vartheta(x):=1+(\Pi(x)_1-2)\chi_{\{\two\}}(\tT_{x_1})=
  1+\sum_{j=1}^{\infty}(j-2)\chi_{[(\two,1)_j]}(x).
\end{equation}
Define the  constant $\theta$ as the integration of $\vartheta$ w.r.t. $\n$. That is,
\begin{equation}\label{def-theta}
  \theta:=\int_{\Omega}\vartheta d\n=1+\sum_{j=1}^{\infty}(j-2)\n([(\two,1)_j]).
\end{equation}
Since  $\n([(\two,1)_1])<1$ and
$
\n([(\two,1)_j])\sim j^{-3}
$
by \eqref{meas-Gibbs}, we conclude that $0<\theta<\infty$.

\begin{lemma}\label{iota-Psi-a}
   There exists a   set $\F_4\subset \F_3$ with $\Gg(\F_4)=1$ such that for any $a\in \F_4$ there exists a $\n_a$-full measure set $X_a\subset \Omega_a$ such that for any $x\in X_a$, we have
   \begin{equation}\label{lim-iota}
   \lim_{n\to\infty}\frac{\sum_{j=1}^{n}\vartheta(\sigma^{j-1}x)}{n}= \theta.
  \end{equation}
   In particular, for any $(a,\lambda)\in \F_4\times[24,\infty)$ and  $x\in X_a$, we have
  \begin{equation}\label{up-bdd-Psi-a}
  \theta\log \frac{\lambda-8}{3}\le \underline\eL_\lambda(x)\le\overline\eL_\lambda(x)\le\theta\log [2(\lambda+5)]+3\log\kappa.
  \end{equation}
\end{lemma}

\begin{proof} Since $\n$ is ergodic and
$
\theta=\int_{\Omega}\vartheta d\n,
$
by ergodic theorem, there exists a $\n$-full measure set $\tilde \Omega\subset \Omega$ such that for any $x\in \tilde \Omega$, \eqref{lim-iota} holds. Since
$$
1=\n(\tilde \Omega)=\int_{\N^\N}\n_a(\tilde \Omega)d\Gg(a),
$$
there exists a $\Gg$-full measure set $\F_4\subset \F_3$ such that for any $a\in \F_4$, we have
$\n_a(\tilde \Omega)=1$. Since  $\n_a$ is  supported on $\Omega_a$,  if we define $X_a:=\tilde\Omega\cap\Omega_a$, then  the first statement holds.

Now fix  $(a,\lambda)\in \F_4\times [24,\infty)$ and  $x\in X_a$, then $\Pi(x)=a$. By Lemma $\ref{esti-band-length}$ and \eqref{Def-iota},
		\begin{align*}
|B^{\check a}_{\iota(x)|_{n+1}}(\lambda)|\leq&   4\tau_1^{-(n+1)}\prod_{1\le k\le n+1;\tT_{\iota(x)_k}=\mathbf{2}}\tau_1^{2-\check a_k}\le 4\tau_1^{-n}\prod_{1\le k\le n;\tT_{x_k}=\mathbf{2}}\tau_1^{2- a_k}\\
=&4\tau_1^{-n-\sum_{k=1}^n(a_k-2)\chi_{\{\two\}}(\tT_{x_k})}
=4\tau_1^{-\sum_{k=1}^{n}\vartheta(\sigma^{k-1}x)},
		\end{align*}
where $\tau_1=(\lambda-8)/3.$
Now by \eqref{lim-iota}, we have
$$
\underline\eL_\lambda(x)=\liminf_{n\to\infty}\frac{-\psi_{\lambda,n}(x)}{n}=
\liminf_{n\to\infty}\frac{-\log|B^{\check a}_{\iota(x)|_{n+1}}(\lambda)|}{n}\ge \theta\log \tau_1.
$$
Similarly, by Lemma $\ref{esti-band-length}$ and \eqref{Def-iota},
\begin{align*}
|B^{\check a}_{\iota(x)|_{n+1}}(\lambda)|\geq&   \tau_2^{-(n+1)}\prod_{k=1}^{n+1}\check a_k^{-3}\prod_{1\le k\le n+1;\tT_{\iota(x)_k}=\mathbf{2}}\tau_2^{2-\check a_k}\\
\geq&   \tau_2^{-(n+1)}\left(\prod_{k=1}^{n}a_k\right)^{-3}\prod_{1\le k\le n;\tT_{x_k}=\mathbf{2}}\tau_2^{2-a_k}\\
 =&\tau_2^{-1}\left(\prod_{k=1}^{n}a_k\right)^{-3}\tau_2^{-\sum_{k=1}^{n}\vartheta(\sigma^{k-1}x)},
		\end{align*}
where $\tau_2=2(\lambda+5).$
Now by \eqref{lim-iota} and \eqref{L-K}, we have
$$
\overline\eL_\lambda(x)=\limsup_{n\to\infty}\frac{-\psi_{\lambda,n}(x)}{n}=
\limsup_{n\to\infty}\frac{-\log|B^{\check a}_{\iota(x)|_{n+1}}(\lambda)|}{n}\le \theta\log \tau_2+3\log\kappa.
$$
		Then the result follows.
\end{proof}

\begin{proof}[Proof of Proposition \ref{lip-fiber}]
Let $\F_4$ be the $\Gg$-full measure set in Lemma \ref{iota-Psi-a}. Fix $a\in \F_4$ and $24\le \lambda_1<\lambda_2$. Since $a\in \F_3$, by Proposition \ref{exact-H-P-dim}, there exist $Y_{a,i}\subset \Omega_a, \ i=1,2$ with $\n_a(Y_{a,i})=1$ and
\begin{equation*}
  \underline{d}_{\n_a}({x},\lambda_i)=\underline{d}(a,\lambda_i),\ \ x\in Y_{a,i},\ \  i=1,2.
\end{equation*}
Fix $x\in X_a\cap Y_{a,1}\cap Y_{a,2}$, where $X_a$ is as in Lemma \ref{iota-Psi-a}. By \eqref{dim-cylinder}, \eqref{up-bdd-Psi-a} and  \eqref{lyapunov-new}, we have
\begin{align*}
  |\underline{d}(a,\lambda_1)- \underline{d}(a,\lambda_2)|=& |\underline{d}_{\n_a}({x},\lambda_1)-\underline{d}_{\n_a}({x},\lambda_2)|=
  \gamma\frac{|\overline\eL_{\lambda_1}(x)-\overline\eL_{\lambda_2}(x)|}
  {\overline\eL_{\lambda_1}(x)\overline\eL_{\lambda_2}(x)}\\
  \le&\frac{\gamma}{(\theta\log (\lambda_1-8)/3)^2}\limsup_n\frac{|\psi_{\lambda_1,n}(x)-
  \psi_{\lambda_2,n}(x)|}{n}.
\end{align*}
By \eqref{def-Psi-lambda}, Proposition \ref{bco-cor} and \eqref{0}, we have
\begin{eqnarray*}
  {|\psi_{\lambda_1,n}(x)-
  \psi_{\lambda_2,n}(x)|}=\left|\log\frac{|B_{\iota(x)|_{n+1}}^{\check{a}}(\lambda_1)|}
  {|B_{\iota(x)|_{n+1}}^{\check{a}}(\lambda_2)|}\right|\le
  \log \eta+\left|\log\frac{|B_{\iota(x)|_0}^{\check a}(\lambda_1)|}{|B_{\iota(x)|_0}^{\check a}(\lambda_2)|}\right|=\log \eta,
\end{eqnarray*}
where by \eqref{def-eta},
$$
\eta=C_1\exp(C_2(\lambda_1+\lambda_2)+C_3m(\iota(x)|_{n+1})\lambda_1|\lambda_1- \lambda_2|).
$$
By \eqref{def-m-u}, \eqref{Def-iota} and \eqref{def-var-theta},
$$
m(\iota(x)|_{n+1})=m(x|_n)=\sum_{j=1}^{n}\vartheta(\sigma^{j-1}x).
$$
Now by \eqref{lim-iota}, we have
$$
\limsup_n\frac{|\psi_{\lambda_1,n}(x)-
  \psi_{\lambda_2,n}(x)|}{n}\le C_3\theta\lambda_1 |\lambda_1-\lambda_2|.
$$
It is seen that there exists an absolute constant $C_4>0$ such that for any $\lambda\in [24,\infty)$,
$$
\frac{\gamma C_3\theta}{(\theta\log (\lambda-8)/3)^2}\le \frac{C_4}{(\log\lambda)^2}.
$$
So the first inequality of \eqref{Lip-of-d-a-lambda} holds. The proof of the second inequality is the same.
\end{proof}

 As another consequence of Lemma \ref{iota-Psi-a}, we get Lemma \ref{asym-psi-lambda-n}.

\begin{proof}[Proof of Lemma \ref{asym-psi-lambda-n}] Fix $\lambda\ge24$. Define
\begin{equation*}\label{Y-lambda}
  \Omega(\lambda):=\{x\in \Omega: \frac{\psi_{\lambda,n}(x)}{n}\to (\Psi_\lambda)_\ast(\n)\}.
\end{equation*}
Since $\n$ is ergodic, by Kingman's ergodic theorem, $\n(\Omega(\lambda))=1$. Let $\tilde \Omega\subset \Omega$ be the $\n$-full measure set in the proof of Lemma \ref{iota-Psi-a}. Take $x\in \Omega(\lambda)\cap \tilde\Omega$. At first,  we have
$-(\Psi_\lambda)_\ast(\n)=\overline \eL_\lambda(x)=\underline \eL_\lambda(x)$. Then by \eqref{up-bdd-Psi-a},
 \begin{equation*}
-\theta\log[2(\lambda+5)]-3\log\kappa\le(\Psi_\lambda)_\ast(\n)\le -\theta\log \frac{\lambda-8}{3}.
\end{equation*}
Consequently, Lemma \ref{asym-psi-lambda-n} follows.
\end{proof}

\section{Proof of  two geometric  lemmas}\label{sec-proof-geo-lem}

In this section, we prove gap lemma  and tail lemma. Since the proof is quite involved, we explain the basic idea of the proof first.

The basic tool is
a structure property about a spectral band $B_w^a(\lambda)$  and its sub-bands $\{B_{we}^a(\lambda)\}.$ It is known that the generating polynomial $h_w$ of $B_w^a(\lambda)$ is a diffeomorphism from $B_w^a(\lambda)$ to $[-2,2]$. It is shown in \cite[Proposition 3.1]{FLW} that the family $\{B_{we}^a(\lambda)\}$ is sent by $h_w$ into a  family $\mathcal I_p$ of disjoint sub-intervals of $[-2,2]$, where $p$ is certain integer depending on $a$ and $w$ (See Figure \ref{figure-gap} for some geometric intuitions). Since $h_w$ has bounded variation (see Proposition \ref{bv}), the length ratio between a gap in $B_w^a(\lambda)$ and $B_w^a(\lambda)$ is the same order with the length of the related gap of $\mathcal I_p.$ Hence,  to prove the gap lemma, we only need to bound the lengths of gaps of $\mathcal I_p$ from below.

Now assume $a,b, u,v$ satisfy the assumption of the tail lemma. At first, it is not hard to construct a natural bijection  $\tau_{uv}:X_u^a(\lambda)\to X_v^b(\lambda)$.
Fix two different points $E, \hat E\in X_u^a(\lambda)$, then they have different codings. Assume $E\in B^a_{uwe}(\lambda), \hat E\in B^a_{uw\hat e}(\lambda)$. By \cite[Proposition 3.1]{FLW}, the generating polynomial $h_{uw}$  of $B_{uw}^a(\lambda)$ sends $E, \hat E$ to different intervals $I,\hat I$ of $\mathcal I_{\tilde p}$ for some $\tilde p$. Write $E_\ast=\tau_{uv}(E), \hat E_\ast=\tau_{uv}(\hat E)$. Then the generating polynomial $h_{vw}$  of $B_{vw}^b(\lambda)$ will send $E_\ast, \hat E_\ast$ to the same intervals $I,\hat I$. Now by using the bounded variation and covariation of trace polynomials (see Propositions \ref{bv} and \ref{bco}), to get the Lipschitz property of $\tau_{uv}$, we only need to bound the ratio
$$
\frac{\max\{|x-\hat x|:x\in I,\hat x\in \hat I\}}{\min\{|x-\hat x|:x\in I,\hat x\in \hat I\}}
$$
from above for any pair $I,\hat I\in \mathcal I_p.$

We will define this special family $\mathcal I_p$ along with an auxiliary family $\mathcal J_p$,  and  obtain these two bounds related to $\mathcal I_p$ and $\mathcal J_p$ in Section \ref{family-I}. Then we apply them to prove two geometric lemmas.

Let us start with some definitions and  basic  properties for a general family.

\subsection{Two quantities related to a family of disjoint intervals}\

Let $\mathscr C$  be the family of all  non-degenerate compact intervals in  $\R$. Assume $I,J\in \mathscr C$ are disjoint. If $I$ is on the left of $J$,  we write $I\prec J$.

Assume $I,J\in \mathscr C$ are disjoint. Define
\begin{equation*}
\begin{cases}
\ d(I,J):=\min\{|x-y|: x\in I, y\in J\};\\
D(I,J):=\max\{|x-y|: x\in I, y\in J\};
\end{cases}  \ \ \ \text{ and } \ \ \ r(I,J):=\frac{D(I,J)}{d(I,J)}.
\end{equation*}
Assume $\mathcal I=\{I_1,\cdots, I_n\}\subset \mathscr C$ is a disjoint family. Define
\begin{equation*}
d(\mathcal I):=\min\{d(I_i,I_j): i\ne j\}\ \ \text{ and }\ \ r(\mathcal I):=\max\{r(I_i,I_j): i\ne j\}.
\end{equation*}

The following lemma is simple but useful, we will omit the easy proof.

\begin{lemma}\label{ratio}

(1) Assume $J_1\prec J_2$ and $I_i\subset J_i, i=1,2$. Then $r(I_1,I_2)\le r(J_1,J_2).$

(2) Assume $I\prec J$. Then
\begin{equation*}
d(-I,-J)=d(I,J);\ \ \ D(-I,-J)=D(I,J)\ \ \text{ and }\ \ r(-I,-J)=r(I,J).
\end{equation*}

(3) Assume $\mathcal I=\{I_1,\cdots, I_k\}$  and  $I_1\prec I_2\prec \cdots\prec  I_k$. Then
\begin{equation*}
d(\mathcal I)=\min\{d(I_l, I_{l+1}): 1\le l< k\}\ \ \text{ and }\ \ r(\mathcal I)=\max\{r(I_l, I_{l+1}): 1\le l< k\}.
\end{equation*}

(4) Assume $\mathcal I=\{I_1,\cdots, I_k\}$ and $\mathcal J=\{J_1,\cdots, J_k\}$ are families of disjoint compact intervals and $I_i\subset J_i$ for $i=1,\cdots, k$. Then $r(\mathcal I)\le r(\mathcal J)$.
\end{lemma}

Geometrically, $d(\mathcal I)$ is the length of the smallest gap of $\mathcal I$. The meaning of  $r(\mathcal I)$ is less obvious. Combine with Lemma \ref{ratio}(3),  $r(\mathcal I)$ gives a uniform upper bound on the length ratios between bands and gaps with common endpoint.

\subsection{$\mathcal I_p$ and $\mathcal J_p$}\label{family-I}\

At first we define the family $\mathcal I_p$ mentioned above.

Define the Chebyshev polynomials $\{S_p(x): p\ge0\}$ recursively as follows:
\begin{equation*}
S_0(x):=0, \ \ \ S_1(x):=1, \ \ \  S_{p+1}(x):=xS_{p}(x)-S_{p-1}(x)\ \ (p\geq1).
\end{equation*}
It is well known that
\begin{equation}\label{chebyshev}
S_p(2\cos\theta)=\frac{\sin p\theta}{\sin\theta} , \ \ (\theta\in[0,\pi]).
\end{equation}
Essentially follow \cite{FLW}, for each $p\ge2$ and $1\leq l< p$, we define
\begin{equation}\label{I-p,l}
I_{p,l}:=\left\{2\cos\theta: \left|\theta-\frac{l\pi}{p}\right|\le\frac{0.1\pi}{p}\ \ \text{and }\ \ |S_{p}(2\cos\theta)|\le \frac{1}{4}\right\}\subset [-2,2].
\end{equation}
(see \cite[Definition 2]{FLW}. However, we warn that our definition is different from \cite{FLW}. The relation is that $I_{p,l}$ in our definition is  $I_{p-1,l}$ in their definition.)

For each $p\ge2$, define
\begin{equation}\label{I-q}
\mathcal I_p:=\{I_{p,l}:1\le l<p\}\cup \{I_{p+1,s}: 1\le s<p+1\}.
\end{equation}
It is proven in \cite{Q1} that the bands in $\mathcal I_p$
are disjoint and ordered as follows:
\begin{equation*}
I_{p+1,p}\prec I_{p,p-1}\prec I_{p+1,p-1}\prec \cdots\prec I_{p+1,2}\prec I_{p,1}\prec I_{p+1,1}.
\end{equation*}

We intend to obtain a lower bound for $d(\mathcal I_p)$ and a upper bound for $r(\mathcal I_p)$. However it is not easy to compute them. Instead,
we  will consider another family $\mathcal J_p$ of intervals obtained by fattening each interval of $\mathcal I_p$ slightly.
Now we give the precise definition.

For $p\ge2$ and $1\le l<p$, define
  \begin{equation}\label{epsilon-p-l}
    \epsilon_{p,l}:=
    \begin{cases}
      0.1, & \mbox{ if }\ \  2\le p\le4,\\
      \min \left\{0.1,\frac{l+0.1}{3.92p},\frac{p-l+0.1}{3.92p}\right\},& \mbox{ if }\ \   p\ge5.
    \end{cases}
  \end{equation}
  Define the interval $J_{p,l}$ as
\begin{equation}\label{def-varepsilon}
J_{p,l}:= \left\{2\cos\theta: \left|\theta-\frac{l\pi}{p}\right|\le\frac{\epsilon_{p,l}\pi}{p}\right\}\subset [-2,2].
\end{equation}
For each $p\ge2$, define
\begin{equation}\label{J-p}
\mathcal J_p:=\{ J_{p,l}:1\le l<p\}\cup \{ J_{p+1,l} :1\le l<p+1\}.
\end{equation}

We have

\begin{proposition}\label{r-I-p}
For each $p\ge2$, the following hold:

(1) For any $1\le l<p$, we have
\begin{equation*}\label{symmetric}
J_{p,l}=-J_{p,p-l}\ \ \ \text{and}\ \ \ I_{p,l}\subset J_{p,l}.
\end{equation*}

(2) The intervals in $\mathcal J_p$
are disjoint and ordered as follows:
\begin{equation*}\label{J}
J_{p+1,p}\prec J_{p,p-1}\prec J_{p+1,p-1}\prec \cdots\prec J_{p+1,2}\prec J_{p,1}\prec J_{p+1,1}.
\end{equation*}

(3)  $r(\mathcal{J}_p)\le 40$ and $d(\mathcal J_p)\ge \frac{1}{20p^3}.$

\end{proposition}

\begin{proof}
(1)	Note that $\epsilon_{p,l}=\epsilon_{p,p-l}$, then $J_{p,l}=-J_{p,p-l}$  follows from \eqref{def-varepsilon}.

If $2\leq p\le 4$, then $\varepsilon_{p,l}=0.1$. Then $I_{p,l}\subset J_{p,l}$  follows from \eqref{I-p,l} and \eqref{def-varepsilon}.	

Now assume $p\ge5$. For any $1\le l<p$, assume $\delta\in [-0.1,0.1]$ is such that $2\cos\frac{l+\delta}{p}\pi\in I_{p,l}$. Then by \eqref{I-p,l} and  \eqref{chebyshev},
we have
\begin{eqnarray*}
\frac{1}{4} \ge \left|S_p(2\cos\frac{l+\delta}{p}\pi)\right|
=\left|\frac{\sin\delta \pi}{\sin \frac{l+\delta}{p}\pi}\right|
\ge\left|\frac{(\frac{\sin 0.1\pi}{0.1\pi})|\delta \pi| }{\sin \frac{l+\delta}{p}\pi}\right|\ge\left|\frac{0.98|\delta \pi| }{\sin \frac{l+\delta}{p}\pi}\right|.
\end{eqnarray*}
Since  $1\le l<p$ and $\delta\in [-0.1,0.1]$,  we have
$$
|\delta|\le \frac{\sin \frac{l+\delta}{p}\pi}{4\times 0.98\pi}\le
\min\left\{
\frac{l+0.1}{3.92p},
\frac{p-l+0.1}{3.92p}\right\}.
$$
Combine with the definition of $\epsilon_{p,l}$, we conclude that $I_{p,l}\subset J_{p,l}$.

(2) It is equivalent to show that  $J_{p+1,l+1}\prec J_{p,l}\prec J_{p+1,l}$ for all $1\leq l<p$. Since $2\cos(\cdot)$ is strictly  decreasing on $[0,\pi]$, it is sufficient  to show
\begin{equation*}\label{condi-disj}
\frac{l+\epsilon_{p+1,l}}{p+1}<\frac{l-\epsilon_{p,l}}{p}\ \  \text{and }\ \ \frac{l+\epsilon_{p,l}}{p}<\frac{l+1-\epsilon_{p+1,l+1}}{p+1}.
\end{equation*}
They are equivalent to
\begin{equation*}\label{condi-disj-1}
	p\epsilon_{p+1,l}+(p+1)\epsilon_{p,l}<l\ \ \text{ and }\ \
p\epsilon_{p+1,l+1}+(p+1)\epsilon_{p,l}<p-l.
\end{equation*}
By \eqref{epsilon-p-l}, if $2\le p\le 4,$ we have
$$
p\epsilon_{p+1,l}+(p+1)\epsilon_{p,l}\le (2p+1)\times 0.1\le 0.9\le 0.9l.
$$
If $p\ge5$, then we have
\begin{equation*}
  p\epsilon_{p+1,l}+(p+1)\epsilon_{p,l}\le \frac{l+0.1}{3.92}\left(\frac{p+1}{p}+\frac{p}{p+1}\right)\le 0.6l.
\end{equation*}
We can do similar estimation for the second inequality. As a result, we get
\begin{equation}\label{condi-disj-2}
	p\epsilon_{p+1,l}+(p+1)\epsilon_{p,l}\le0.9l\ \ \text{ and }\ \
p\epsilon_{p+1,l+1}+(p+1)\epsilon_{p,l}\le0.9(p-l).
\end{equation}
So statement  (2) holds.

(3)	At first, we show $r(\mathcal{J}_p)\le 40$. By the statement  (2) and Lemma \ref{ratio}(3), we only need to show that for $1\le l< p$,
\begin{equation*}\label{bd-ratio}
r(J_{p+1,l+1},J_{p,l}),\ \ r(J_{p,l}, J_{p+1,l})\le40.
\end{equation*}
By the statement (1) and Lemma \ref{ratio}(2),
\begin{align*}
r(J_{p+1,l+1},J_{p,l})=r( J_{p,l}, J_{p+1,l+1})=r(- J_{p,p-l},- J_{p+1,p-l})=r( J_{p,p-l},J_{p+1,p-l}).
\end{align*}
Thus we only need to show that for $1\le l<p,$
\begin{equation}\label{bd-ratio}
 r(J_{p,l}, J_{p+1,l})\le40.
\end{equation}
By \eqref{def-varepsilon}, we have
\begin{align*}
d( J_{p,l}, J_{p+1,l})
=&2\cos \frac{(l+\epsilon_{p+1,l})\pi}{p+1}-2\cos \frac{(l-\epsilon_{p,l})\pi}{p}=4\sin( x_{p,l}+ y_{p,l})\sin(\hat{ x}_{p,l}-\hat{ y}_{p,l})\\
D( J_{p,l}, J_{p+1,l})
=&2\cos \frac{(l-\epsilon_{p+1,l})\pi}{p+1}-2\cos \frac{(l+\epsilon_{p,l})\pi}{p}=4\sin( x_{p,l}- y_{p,l})\sin(\hat{ x}_{p,l}+\hat{ y}_{p,l}),
\end{align*}
where
\begin{align*}
x_{p,l}:=\frac{l(2p+1)}{p(p+1)}\frac{\pi}{2};\ \ \ \ & y_{p,l}:=\frac{p\epsilon_{p+1,l}-(p+1)\epsilon_{p,l}}{p(p+1)}\frac{\pi}{2}\\
\hat x_{p,l}:=\frac{l}{p(p+1)}\frac{\pi}{2};\ \ \ \ & \hat  y_{p,l}:=\frac{p\epsilon_{p+1,l}+(p+1)\epsilon_{p,l}}{p(p+1)}\frac{\pi}{2}.
\end{align*}
So we have
\begin{equation}\label{ratio-J-p}
r( J_{p,l}, J_{p+1,l})=\frac{D( J_{p,l}, J_{p+1,l})}{d( J_{p,l}, J_{p+1,l})}=\frac{\sin( x_{p,l}- y_{p,l})}
{\sin( x_{p,l}+ y_{p,l})}\cdot\frac{\sin(\hat x_{p,l}+\hat y_{p,l})}{\sin(\hat x_{p,l}-\hat y_{p,l})}.
\end{equation}

\noindent{\bf Claim:} Given $k>1$. If $x>0$ and $|y|\le \frac{\pi}{2k}$ are such that $k|y|\le x\le\pi-k|y|$, then
\begin{equation*}
\frac{k-1}{k+1}\leq\frac{\sin(x+ y)}{\sin(x- y)}\leq \frac{k+1}{k-1}.
\end{equation*}

\noindent $\lhd$
If $y=0,$ the claim holds trivially. Now assume $y\ne0$, then $|y|\in (0,\frac{\pi}{2k})$.
Since $0<k|y|\le x\le\pi-k| y|<\pi$,  we have
 $|\tan x|\geq\tan k |y|.$
It is ready to check that
$$\tan(ky)\geq k\tan y,\ \forall\ y\in [0,\frac{\pi}{2k}).$$
So we have
$
	\left|\frac{\tan x}{\tan y}\right|\geq k.
$	
Notice that
$$\frac{\sin(x+ y)}{\sin(x- y)}=\frac{\tan x+\tan y}{\tan x-\tan y}
=\frac{\frac{\tan x}{\tan y}+1}{\frac{\tan x}{\tan y}-1}
=\xi(\frac{\tan x}{\tan y}),$$
where $\xi(z)=(z+1)/(z-1)$. Since $\xi$ is decreasing on $(-\infty,1)$ and $(1,\infty)$, we have
$$
\frac{k-1}{k+1}=\xi(-k)\le\frac{\sin(x+ y)}{\sin(x- y)}=\xi(\frac{\tan x}{\tan y})\le \xi(k)=\frac{k+1}{k-1}.
$$
So the claim holds.
\hfill $\rhd$

By \eqref{condi-disj-2}, we have
\begin{equation}\label{ypl}
| y_{p,l}|\le|\hat  y_{p,l}|=\frac{p\epsilon_{p+1,l}+(p+1)\epsilon_{p,l}}{p(p+1)}\frac{\pi}{2}
\le\frac{0.9l}{p(p+1)}\frac{\pi}{2}.
\end{equation}
So we have
\begin{equation*}
\frac{\hat x_{p,l}}{|\hat y_{p,l}|} \ge\frac{10}{9}\ \ \ \text{and}\ \ \			\frac{ x_{p,l}}{| y_{p,l}|}\ge2p+1\ge5.
\end{equation*}
On the other hand, since $p\ge2$ and $l<p$, by \eqref{ypl},
\begin{align*}
  \frac{\pi-\hat{x}_{p,l}}{|\hat{y}_{p,l}|} \ge& \frac{2p(p+1)-l}{0.9l}\ge5,\\
  \frac{\pi-{x}_{p,l}}{|{y}_{p,l}|} \ge& \frac{2p(p+1)-l(2p+1)}{0.9l}\ge \frac{2p(p+1)-(p-1)(2p+1)}{0.9l}\ge3.
\end{align*}
So we have
\begin{equation}\label{lemma-use}
3| y_{p,l}|\le  x_{p,l}\le \pi-3| y_{p,l}|\ \ \text{and}\ \ 	\frac{10}{9}|\hat  y_{p,l}|\le \hat  x_{p,l}\le \pi-\frac{10}{9}|\hat  y_{p,l}|.
\end{equation}
Now by \eqref{ratio-J-p}, the claim and \eqref{lemma-use}, we obtain \eqref{bd-ratio}.

Next we show $d(\mathcal J_p)\ge \frac{1}{20p^3}.$ By Lemma \ref{ratio}(3), we only need to estimate $d(I,J)$ for two consecutive bands $I$ and $J$ in $\mathcal J_p$. For any $1\le l<p$,
\begin{align*}
d(J_{p,l},J_{p+1,l}) \ge&\frac{D(J_{p,l},J_{p+1,l})}{r(\mathcal J_p)}\ge  \frac{D(J_{p,l},J_{p+1,l})}{40}
\ge\frac{2\cos \frac{l\pi}{p+1} -2\cos \frac{l\pi}{p}}{40}\\
=&\frac{1}{10}\sin\left[\frac{1}{p}+\frac{1}{p+1}\right]\frac{l\pi}{2}\sin\frac{l\pi}{2p(p+1)}
\ge\frac{1}{10}\sin\frac{\pi}{2p}\sin \frac{\pi}{4p^2}
\ge\frac{1}{20p^3}.
\end{align*}
Similarly , we have
\begin{align*}
d(J_{p,l},J_{p+1,l+1})\ge  \frac{D(J_{p,l},J_{p+1,l+1})}{40}\ge\frac{2\cos \frac{l\pi}{p} -2\cos \frac{(l+1)\pi}{p+1}}{40}\ge\frac{1}{20p^3}.
\end{align*}
So the result follows.
\end{proof}

\subsection{Proof of  Lemma \ref{geo-lem-gap} (gap lemma)}\

In this part, we write $\B_n:=\B^a_n(\lambda),  B_w:=B^a_w(\lambda)$ for simplicity.

\begin{proof}
Given $B_w\in \mathcal{B}_n$, we
study the gaps of order $n$ which are contained in $B_w$.
Write $q:=a_{n+1}$ and
$$
e_j:=(\one,j)_q,\ \  1\le j\le q+1; \ \  e_{q+2}:=(\two,1)_q;\ \ e_{q+2+j}:=(\three,j)_q,\ \  1\le j\le q.
$$

If  $\tT_w=\one$, by Proposition \ref{basic-struc}(3),  there
exists a unique band $B_{we_{q+2}}\in \mathcal{B}_{n+1}$, which is contained in $B_w$. There is
no gap of order n in $B_w$ in this case.

Now assume $\tT_w=\two$. by Proposition \ref{basic-struc}(4), there are $2q+1$ bands of order $n+1$ in $B_w$,
which are disjoint and ordered as follows (see Figure \ref{figure-gap}):
$$B_{we_{1}}\prec B_{we_{q+3}}\prec B_{we_{2}}\prec B_{we_{q+4}}\prec B_{we_{3}}
\prec \cdots\prec B_{we_{2q+2}}\prec B_{we_{q+1}}.$$
\begin{figure}
		\includegraphics[scale=0.45]{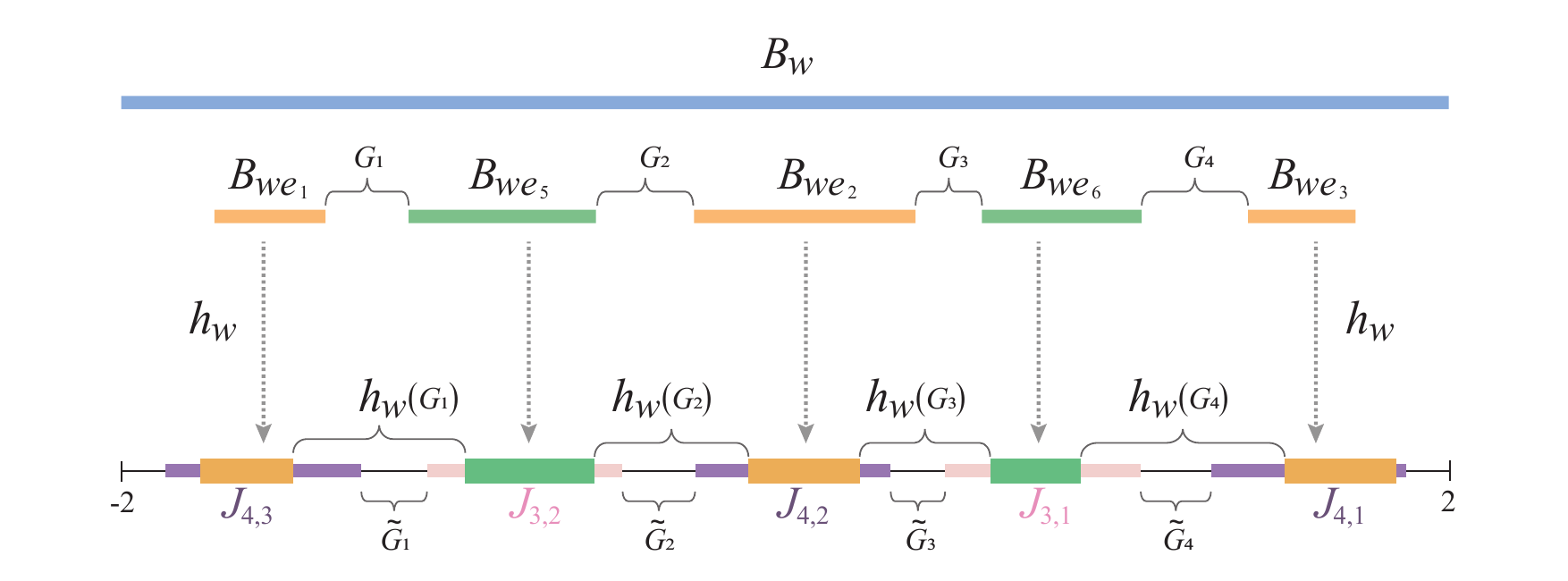}
		\caption{The case $q=2$, $\tT_w=\two$ and $h_w'|_{B_w}>0$}\label{figure-gap}
 \end{figure}
There are $2q$ gaps of order n in $B_w$. We list them from left to right as
$$
G_1\prec G_2\prec \cdots\prec G_{2q}.
$$
By \cite[Proposition 3.1]{FLW} and Proposition \ref{r-I-p}(1), if $h'_w|_{B_w}<0$, then
 $$
\begin{cases}
h_w(B_{we_{l}})\subset I_{q+2,l}\subset J_{q+2,l}, &1\leq l\leq q+1,\\
h_w(B_{we_{q+2+l}})\subset I_{q+1,l}\subset J_{q+1,l}, &1\leq l\leq q.
\end{cases}
$$
If $h'_w|_{B_w}>0$, then
 $$
\begin{cases}
h_w(B_{we_{q+2-l}})\subset I_{q+2,l}\subset J_{q+2,l}, &1\leq l\leq q+1,\\
h_w(B_{we_{2q+3-l}})\subset I_{q+1,l}\subset J_{q+1,l}, &1\leq l\leq q.
\end{cases}
$$

By Proposition \ref{r-I-p}(2),	there are  $2q$ gaps for the set
$\bigcup_{J\in \mathcal J_{q+1}} J\subset [-2,2]$. We list them from left to right as
$$
\tilde G_1<\tilde G_2<\cdots <\tilde G_{2q}.
$$
By Proposition \ref{r-I-p}(3), we have the following lower bound  for the length of $\tilde{G}_j$:
\begin{equation}\label{min-gap}
\min\{|\tilde G_{j}|:1\le j\le 2q\}\ge \frac{1}{20(q+1)^3}\ge\frac{1}{160 q^3}.
\end{equation}

Since $h_w: B_w\to [-2,2]$ is a diffeomorphism,
there exists  $x_*\in B_w$   such that
\begin{equation}\label{point}
|h^{'}_w(x_*)||B_w|=4.
\end{equation}
If we define
$$
\delta_j:=\begin{cases}
j, & \mbox{if } h_w'|_{B_w}>0, \\
2q+1-j, & \mbox{if }  h_w'|_{B_w}<0.
\end{cases}
$$
Then we further have
$h_w(G_j)\supset \tilde{G}_{\delta_j}, 1\le j\le 2q.$
Fix a gap $G_j$
in $B_w$, then there exists $x_j\in G_j$ such that
\begin{equation}\label{point2}
|h_w^{'}(x_j)||G_j|=|h_w(G_j)|\geq|\tilde{G}_{\delta_j}|.
\end{equation}
By \eqref{min-gap}-\eqref{point2},   and Proposition \ref{bv}, there exists a constant
$C_1(\lambda)>0$ such that
$$\frac{|G_j|}{|B_w|}\geq\frac{|\tilde G_{\delta_j}|}{4}\cdot\frac{|h_w^{'}(x_*)|}{|h_w^{'}(x_j)|}\geq\frac{C_1(\lambda)}{q^3}.$$

If $\tT_w=\three$, the same proof as above shows that there exists a constant
$C_2(\lambda)>0$ such that for any gap $G$ in $B_w$,
$$\frac{|G|}{|B_w|}\geq\frac{C_2(\lambda)}{q^3}.$$
Let $C(\lambda):=\min\{C_1(\lambda), C_2(\lambda)\}$, the result follows.
\end{proof}

\subsection{Proof of Lemma \ref{geo-lem-tail} (tail lemma)}

\begin{proof}
Since $S^na=S^mb, u\in \Omega^a_{n+k}, v\in \Omega^b_{m+k}$ and $\tT_u=\tT_v$, by the definition of the symbolic space $\Omega^a$ (see \eqref{Omega^alpha}), we have
$$
ux\in [u]^a \Leftrightarrow vx\in [v]^b.
$$
Thus we can define a bijection   $\Upsilon_{uv}:[u]^a\to [v]^b$ as
$\Upsilon_{uv}(ux):=vx.$
Then we define $\tau_{uv}:X^a_u(\lambda)\to X^b_v(\lambda)$ as
\begin{equation*}\label{def-tau-uv}
\tau_{uv}:=\pi^b_{\lambda}\circ \Upsilon_{uv} \circ (\pi^a_\lambda)^{-1}.
\end{equation*}
It is seen that $\tau_{uv}$ is a bijection. Next we show that $\tau_{uv}$ is bi-Lipschitz.

Fix any   $E,\hat E\in X^a_u(\lambda)$ with $E\ne \hat E$. If we write
$
(\pi^a_\lambda)^{-1}(E)=ux$ and $(\pi^a_\lambda)^{-1}(\hat E)=u\hat x,
$
then $x\ne \hat x$. Write $w:=x\wedge \hat x$, by \eqref{admissible-T-A}, it is seen that $\tT_w=\two $ or $\three.$ Set
$$
s:=a_{n+|w|+1};\ \ \ \ \  q:=
\begin{cases}
s+1, & \mbox{if }\  \tT_w=\two,\\
s,&\mbox{if }\ \tT_w=\three.
\end{cases}
$$
Again by \eqref{admissible-T-A},  there exist $e, \hat e\in \A_s$ with $e\ne \hat e$ and $e,\hat e\ne (\two,1)_s$ such that
$$
E\in B^a_{uwe}(\lambda)\ \ \ \text{ and }\ \ \ \hat E \in B^a_{uw\hat e}(\lambda).
$$
By \cite[Proposition 3.1]{FLW}, there exist two disjoint bands $I,\hat I\in \mathcal I_q$ (see \eqref{I-q}) such that
$$
h_{uw}(B_{uwe}^a(\lambda))\subset I \ \  \text{and} \ \ h_{uw}(B_{uw\hat e}^a(\lambda))\subset \hat I.$$
There exists $\tilde E\in B_{uw}^a(\lambda)$ such that
$$
h_{uw}(E)-h_{uw}(\hat E)=h_{uw}'(\tilde E)(E-\hat E).
$$

Write $E_\ast=\tau_{uv}(E)$ and $\hat E_\ast=\tau_{uv}(\hat E).$ We have
$$
(\pi^b_\lambda)^{-1}(E_\ast)=vx;\ \ \  (\pi^b_\lambda)^{-1}(\hat E_\ast)=v\hat x.
$$
Thus for the same $e,\hat e$ and $I,\hat I$, we have
\begin{eqnarray*}
E_\ast\in B^b_{vwe}(\lambda);\ \  \hat E_\ast \in B^b_{vw\hat e}(\lambda)\ \ \text{ and }\ \
h_{vw}(B_{vwe}^b(\lambda))\subset I; \ \   h_{vw}(B_{vw\hat e}^b(\lambda))\subset \hat I.
\end{eqnarray*}
There exists $\tilde E_\ast\in B_{vw}^b(\lambda)$ such that
$$
h_{vw}(E_\ast)-h_{vw}(\hat E_\ast)=h_{vw}'(\tilde E_\ast)(E_\ast-\hat E_\ast).
$$
So by Proposition \ref{bv}, Lemma \ref{ratio}(4), Proposition \ref{bco} and Proposition \ref{r-I-p}(3), we have
\begin{align*}
\frac{|E-\hat E|}{|E_\ast-\hat E_\ast|}&=\frac{|h_{uw}(E)-h_{uw}(\hat E)|}{|h_{vw}(E_\ast)-h_{vw}(\hat E_\ast)|}\frac{|h_{vw}'(\tilde E_\ast)|}{|h_{uw}'(\tilde E)|}\le\frac{D(I,\hat I)}{d(I,\hat I)}\cdot C(\lambda)^2\cdot\frac{|B_{uw}^a(\lambda)|}{|B_{vw}^b(\lambda)|}\\
&\le r(\mathcal I_q)\cdot C(\lambda)^2\cdot\left(\frac{|B_{uw}^a(\lambda)|}{|B_{u}^a(\lambda)|}/
\frac{|B_{vw}^b(\lambda)|}{|B_v^b(\lambda)|}\right)\frac{|B_{u}^a(\lambda)|}{|B_{v}^b(\lambda)|}\\
&\le r(\mathcal{J}_q)C(\lambda)^2\eta(\lambda)\frac{|B_{u}^a(\lambda)|}{|B_{v}^b(\lambda)|}\le40 C(\lambda)^2\eta(\lambda)\frac{|B_{u}^a(\lambda)|}{|B_{v}^b(\lambda)|}.
\end{align*}

By the same argument, we have
$$
\frac{|E-\hat E|}{|E_\ast-\hat E_\ast|}\ge \frac{1}{40C(\lambda)^2\eta(\lambda) }\frac{|B_{u}^a(\lambda)|}{|B_{v}^b(\lambda)|}.
$$
So we conclude that $\tau_{uv}$ is bi-Lipschitz with the Lipschitz constant
$$
L=40C(\lambda)^2\eta(\lambda)\max\left\{\frac{|B_{u}^a(\lambda)|}{|B_{v}^b(\lambda)|}, \frac{|B_{v}^b(\lambda)|}{|B_{u}^a(\lambda)|}\right\}.
$$
Hence the result follows.
\end{proof}

\noindent{\bf Acknowledgement}.
The authors thank the referees for many valuable suggestions, especially about the structure of paper, the numbering of the equations, geometric and intuitive explanations, which clarify some ambiguities
in mathematics and greatly improve the exposition. They also thank them for pointing out references \cite{BBBRT,BBL,KKL,Luna}, which are closely related to this paper.
Qu was supported by the National Natural Science Foundation of China, No. 11790273, No. 11871098 and No. 12371090.
\begin{appendices}
\section{A table of natations}\label{sec-app}
Due to the complex nature of the spectral properties of Sturmian Hamiltonians, in this paper we need to introduce a lot of notations. For the reader's convenience, we include  a table  of   notations  in this appendix.

\smallskip

\begin{longtable}{ll}
\hline
$\Sigma_{\alpha,\lambda},\Sigma_{a,\lambda}, \NN_{\alpha,\lambda},\NN_{a,\lambda}$&  spectrum, DOS,  see \eqref{def-dos}, \eqref{convention}\\
\hline
$ \II\supset \tilde \II\supset\hat \II $& full measure sets of frequencies, see \eqref{def-I}, \eqref{def-tilde-I}, \eqref{def-hat-I}\\
\hline
$(\II,T,G)$& irrationals in $[0,1]$, Gauss map, Gauss measure, see Section \ref{sec-continued-fra}\\
\hline
$(\N^\N,S,\Gg)$& full shift over $\N$, Gauss measure, see Section \ref{sec-continued-fra} \\
\hline
$\Theta: \II\to \N^\N$& symbolic representation of irrationals $\II$, see \eqref{def-Theta}\\
\hline
$ q_n(\alpha), q_n(a), q_n(\vec a)$\ \ & denominator of $n$-th convergent of $\alpha$, see  \eqref{def-p-q-n}, \eqref{def-q-n-a}, Remark \ref{q-n-vec-a}\\
\hline
$\B^a_n(\lambda), \mathcal G_n^a(\lambda)$& the set of spectral bands, the set of  gaps, see \eqref{def-B-n},  \eqref{def-G-n}\\
\hline
$ B^a_w(\lambda), X^a_w(\lambda)$& spectral band, basic set, see  Section \ref{sec-coding-map}\\
\hline
$\T,\T_0,\A_n, \A, e\to \hat e$& alphabets, admissible relation, see Section \ref{sec-Omega^a}, \eqref{def-A}\\
\hline
$\tT_e,\iI_e,\ell_e,\tT_{w}$& type, index, level of $e$, type of  $w$, see \eqref{type-index-level}, \eqref{type}\\
\hline
$A_{nm}, A, \hat A_n$& incidence matrices, see \eqref{A_ij}, \eqref{A-and-A-nm}, \eqref{hat-A-n}\\
\hline
$\Omega^a, \Omega^a_n,  [w]^a$& coding space, admissible words, cylinder, see Section \ref{sec-Omega^a} \\
\hline
$\Omega_a, \Omega_{\vec a,n}, \Omega_{a,n}, [w]_a$& fiber, admissible words, cylinder, see Section \ref{sec-fiber-Omega}\\
\hline
$\Omega,\sigma$&  global symbolic space, shift map, see \eqref{def-Omega}, \eqref{def-shift}\\
\hline
$\Pi:\Omega\to \N^\N$ & projection from $\Omega$ to $\N^\N$, see \eqref{def-Pi}\\
\hline
$\check a,\widehat \Omega, \iota:\Omega\to\widehat \Omega$&  see \eqref{hat-Omega}, \eqref{Def-iota}\\
\hline
$\Xi(\tT,\vec a), \Xi(\tT,\vec a,\tT')$& two sets of admissible words, see \eqref{def-Xi-N}\\
\hline
$\Xi_{a,m}(w), \Xi_{a,m,\tT}(w)$& descendants of $w$, see \eqref{des-w}\\
\hline
$\pi^a_\lambda,\pi_{a,\lambda}, \rho_{a,\lambda}$& coding maps, metric on $\Omega_a$, see \eqref{pi^alpha}, \eqref{pi_a-lambda}, \eqref{rho-a-lambda} \\
\hline
$\QQ(a,\lambda,t,n)$& logarithm of partition function, see \eqref{def-Q}
\\
\hline
$\overline{\rm P}_{a,\lambda}(t), \underline{\rm P}_{a,\lambda}(t), \mathbf{P}_{\lambda}(t)$& upper, lower, relativized  pressure functions, see \eqref{def-lu-pre}, \eqref{relative-pre} \\
\hline
$\phi_{n},\Phi$ & entropic potential,  see \eqref{phi-poten}  \\
\hline
$\n, \n_a$&Gibbs measure for $\Phi$, fiber of $\n$,  see Proposition \ref{Gauss-measure-Omega}, Lemma \ref{imp} \\
\hline
$\psi_{\lambda,n}, \Psi_{\lambda}, (\Psi_\lambda)_\ast(\n),$& geometric potential, Lyapunov exponent, see \eqref{def-Psi-lambda}, \eqref{Psi-ast-n}
\\
\hline
$s_\ast(a,\lambda), s^\ast(a,\lambda)$& pre-dimensions of $\Sigma_{a,\lambda}$, see \eqref{pre-dim}\\
\hline
$ \underline{D}(a,\lambda),\overline{D}(a,\lambda)$& ``zeros" of the pressure functions, see \eqref{def-d-D}\\
\hline
$ \underline{d}(a,\lambda),\overline{d}(a,\lambda)$& Hausdorff and packing dimensions of $\n_a$, see \eqref{def-d-a-lambda}\\
\hline
$ D(\lambda),d(\lambda)$& almost sure dimension of $\Sigma_{a,\lambda}$, $\NN_{a,\lambda}$, see Proposition \ref{as-pressure}(3), \eqref{def-dim-dos}\\
	\hline
	$ \rho, \varrho$& the asymptotic constants of $D(\lambda)$ and  $d(\lambda)$, see \eqref{def-rho}, \eqref{def-var-rho}  \\
\hline
$\F_1\supset \F(\lambda)\supset\F_2\supset \F_3$&  $\Gg$-full measure sets, see Proposition \ref{F-proposition},  \eqref{def-F-lambda}, \eqref{def-tilde-F}, Lemma \ref{imp}\\
\hline
$\F_3\supset \F_4\supset\widehat \F(\lambda)\supset \widehat\F$& $\Gg$-full measure sets, see
Lemma \ref{imp}, Propositions  \ref{lip-fiber} and \ref{exact-dim-fix-lambda}, \eqref{def-hat-F} \\
\hline
$\vartheta, \theta$& see \eqref{def-var-theta}, \eqref{def-theta}\\
\hline
$ \mathcal I_p, \mathcal J_p$& two families of sub-intervals of $[-2,2],$ see  \eqref{I-q}, \eqref{J-p}\\
\hline
\end{longtable}

\end{appendices}



\begin{thebibliography}{9999}


\bibitem{BBBRT}
Band R., Beckus S.,  Biber B.,  Raymond L.,  Thomas Y., {\it A review of a work by Raymond: Sturmian Hamiltonians with a large coupling constant -- periodic approximations and gap labels,} arXiv:2409.10920.

\bibitem{BBL} Band R.,  Beckus S.,  Loewy R., {\it The Dry Ten Martini Problem for Sturmian Hamiltonians,} arXiv:2402.16703.


\bibitem{BIST}  Bellissard J.,  Iochum B.,  Scoppola E. and  Testart D.,
\emph{ Spectral properties of one dimensional quasi-crystals},
Commun. Math. Phys. {\bf 125}(1989), 527-543.

\bibitem{Bi} Billingsley P., {\em Probability and measure}. Third edition. Wiley Series in Probability and Mathematical Statistics. A Wiley-Interscience Publication. John Wiley \& Sons, Inc., New York, 1995.

   \bibitem{Bowen}  Bowen R., {\em Equilibrium states and the ergodic theory of Anosov diffeomorphisms.} Second revised edition. With a preface by David Ruelle. Edited by Jean-Ren\'e Chazottes. Lecture Notes in Mathematics, 470. Springer-Verlag, Berlin, 2008.


\bibitem{C} Cantat S., \emph{  Bers and H\'enon, Painlev\'e and Schr\"odinger}, Duke Math. J. 149 (2009), 411-460.

\bibitem{CQ} Cao J.,  Qu Y.-H., \emph{Almost sure multifratal formalism for the density of states of Sturmian Hamiltonian}, in preparation.


\bibitem{CL}  Carmona R. and  Lacroix J.,  \emph{ Spectral theory of random Schr\"odinger operators.} Probability and its Applications,  Birkh\"auser Boston, Inc., Boston, MA, 1990.

\bibitem{Ca}  Casdagli M., \emph{ Symbolic dynamics for the renormalization map of a quasiperiodic Schr\"odinger equation},  Comm. Math. Phys. 107 (1986), no. 2, 295-318.


\bibitem{Da07} Damanik D., {\em Strictly ergodic subshifts and associated operators.} Spectral theory and mathematical physics: a Festschrift in honor of Barry Simon's 60th birthday, 505-538, Proc. Sympos. Pure Math., 76, Part 2, Amer. Math. Soc., Providence, RI, 2007.

\bibitem{Da17}    Damanik D., {\em  Schr\"odinger operators with dynamically defined potentials}. Ergodic Theory Dynam. Systems 37 (2017), no. 6, 1681-1764.

\bibitem{DEG15} Damanik D., Embree M., Gorodetski A., {\em Spectral properties of Schr\"odinger operators arising in the study of quasicrystals}. Mathematics of aperiodic order, 307-370, Progr. Math., 309, Birkh\"auser/Springer, Basel, 2015.

\bibitem{DEGT} Damanik D., Embree M., Gorodetski A., and Tcheremchantsev S., \emph{ the fractal dimension of the spectrum of the
Fibonacci Hamiltonian}, Commun. Math. Phys. 280:2(2008), 499-516.

\bibitem{DG} Damanik D. and  Gorodetski A.,  \emph{ Hyperbolicity of the trace map for the weakly coupled Fibonacci Hamiltonian}, Nonlinearity 22 (2009), 123-143.

\bibitem{DG2} Damanik D. and  Gorodetski A., \emph{ Spectral and quantum dynamical properties of the weakly coupled Fibonacci Hamiltonian}, Commun. Math. Phys. 305 (2011), 221-277.

\bibitem{DG3}Damanik D. and  Gorodetski A.,  \emph{ The density of states measure of the weakly coupled Fibonacci Hamiltonian}, Geom. Funct. Anal. 22 (2012), no. 4, 976-989.

\bibitem{DG4}Damanik D. and  Gorodetski A.,  \emph{H\"older continuity of the integrated density of states for the Fibonacci Hamiltonian},  Comm. Math. Phys. 323 (2013), no. 2, 497-515.

\bibitem{DG15} Damanik D. and Gorodetski A., {\em Almost sure frequency independence of the dimension of the spectrum of Sturmian Hamiltonians}. Comm. Math. Phys. 337 (2015), no. 3, 1241-1253.


\bibitem{DGLQ} Damanik D.,  Gorodetski A.,  Liu Q.-H. and  Qu Y.-H.,
\emph{Transport Exponents of Sturmian Hamiltonians}, Journal of Functional Analysis. 269 (2015), no. 5, 1404-1440.

\bibitem{DGY} Damanik D., Gorodetski A., Yessen W., {\em The Fibonacci Hamiltonian}. Invent. Math. 206 (2016), no. 3, 629-692.

\bibitem{DKL} Damanik D.,  Killip R.,  Lenz D.,
{\em Uniform spectral properties of one-dimensional quasicrystals,III. $\alpha$-continuity},  Commun. Math. Phys.  212,(2000),191-204.

\bibitem{EW2011} Einsiedler M.,  Ward T.,  {\em Ergodic theory with a view towards number theory.} Graduate Texts in Mathematics, 259. Springer-Verlag London, Ltd., London, 2011.



\bibitem{Fal97}  Falconer K., {\em techniques in fractal geometry}, John Wiley\& Sons, 1997.

\bibitem{FLW} Fan S.,  Liu Q.-H., Wen Z.-Y.,
\emph{ Gibbs like measure for spectrum of a class of quasi-crystals},
Ergodic Theory Dynam. Systems, {\bf 31}(2011), 1669-1695.

\bibitem{FS}Feng D.-J., Shu L., {\em Multifractal analysis for disintegrations of Gibbs measures and conditional Birkhoff averages}. Ergodic Theory Dynam. Systems 29 (2009), no. 3, 885-918.

\bibitem{Gi}  Girand A., \emph{Dynamical Green Functions and Discrete Schr\"odinger Operators with Potentials Generated by Primitive Invertible Substitution},
Nonlinearity,  27 (2014) 527-543.

\bibitem{IY} Iommi G., Yayama Y., {\em Almost-additive thermodynamic formalism for countable Markov shifts}. Nonlinearity 25 (2012), no. 1, 165-191.

    \bibitem{IK} Iosifescu M.,  Kraaikamp C., {\em Metrical theory of continued fractions.}
Mathematics and its Applications, 547. Kluwer Academic Publishers, Dordrecht, 2002.

\bibitem{JL}  Jitomirskaya S. and Last Y., \emph{ Power-law subordinacy and singular spectra. II. Line operators}, Commun. Math. Phys. 211 (2000), 643-658.

\bibitem{JZ} Jitomirskaya S., Zhang S.-W., {\em Quantitative continuity of singular continuous spectral measures and arithmetic criteria for quasiperiodic Schr\"odinger operators.} J. Eur. Math. Soc. (JEMS) 24 (2022), no. 5, 1723-1767.

\bibitem{YK}Khinchin Y., {\em Zur metrischen Kettenbrucheorie}, Comp. Math.3(1936),276-285.

\bibitem{KKL}  Killip R.,  Kiselev A.,  Last Y.,  {\it Dynamical upper bounds on wavepacket spreading.}  Amer. J. Math. 125 (2003), no. 5, 1165-1198.

\bibitem{KKT} Kohmoto M.,  Kadanoff L. P.,  Tang C., {\it Localization problem in one dimension: mapping and escape}, {\em Phys. Rev. Lett.} {\bf 50} (1983), 1870-1872.

\bibitem{LS} L\'evy P., {\em Sur le development en fraction continue d'un nombre choisi au hazar}, Comp. Math3(1936),286-303.

\bibitem{LPW07} Liu Q.-H.,  Peyri\`ere J. and  Wen Z.-Y.,
\emph{ Dimension of the spectrum of one-dimensional discrete Schr\"odinger operators with Sturmian potentials}, Comptes Randus Mathematique, {\bf 345:12}(2007), 667--672.

\bibitem{LQW}   Liu Q.-H.,   Qu Y.-H.,   Wen Z.-Y.,  \emph{ The fractal dimensions of the spectrum of Sturm Hamiltonian},  Adv. Math. 257 (2014), 285-336.

\bibitem{LW} Liu Q.-H.,    Wen Z.-Y.,
\emph{ Hausdorff dimension of spectrum of one-dimensional Schr\"odinger operator with Sturmian potentials}, { Potential Analysis} {\bf 20:1}(2004), 33--59.

\bibitem{LW05} Liu Q.-H.,    Wen Z.-Y., \emph{ On dimensions of multitype Moran sets},
Math. Proc. Camb. Phyl. Soc. {\bf 139:3}(2005), 541--553.




\bibitem{Luna}
  Luna A., {\it
On the spectrum of Sturmian Hamiltonians of bounded type in a small coupling regime}, arXiv:2408.01637.

\bibitem{Mat}Mattila P., {\em Geometry of sets and measures in Euclidean spaces},
Cambridge Stud. Adv. Math., 44 Cambridge University Press, Cambridge, 1995.

\bibitem{Mei}  Mei M., \emph{  Spectra of Discrete Schr\"odinger Operators with Primitive Invertible Substitution Potentials}, J. Math. Phys. 55 (2014), no. 8, 082701, 22 pp.

\bibitem{Munger}   Munger P., {\em Frequency dependence of H\"older continuity for quasiperiodic Schr\"odinger operators}. J. Fractal Geom. 6 (2019), no. 1, 53-65.

\bibitem{OPRSS}  Ostlund S.,  Pandit R.,  Rand D.,  Schellnhuber H., Siggia E., \emph{ One-dimensional Schr\"odinger equation with an almost periodic potential}, {  Phys. Rev. Lett.} 50 (1983), 1873-1877.

\bibitem{Pa}Parry W., {\em Entropy and generators in ergodic theory}. W. A. Benjamin, Inc., New York-Amsterdam 1969.

\bibitem{P}  Pollicott M.,  \emph{Analyticity of dimensions for hyperbolic surface diffeomorphisms},  Proc. Amer. Math. Soc. 143 (2015), no. 8, 3465-3474.

\bibitem{Q1} Qu Y.-H., {\em The spectral properties of the strongly coupled Sturm Hamiltonian of eventually constant type}. Ann. Henri Poincar\'e 17 (2016), no. 9, 2475-2511.

\bibitem{Q2}Qu Y.-H., {\em Exact-dimensional property of density of states measure of Sturm Hamiltonian}. Int. Math. Res. Not. IMRN 2018, no. 17, 5417-5454.

\bibitem{R}  Raymond L.,
{\em A constructive gap labelling for the discrete schr\"odinger
operater on a quasiperiodic chain}.(Preprint,1997)

\bibitem{Ro}Rokhlin V. A., {\em On the fundamental ideas of measure theory}. Amer. Math. Soc. Translation 1952, (1952). 1-52.


\bibitem{Sa}Sarig O.,  {\em Lecture Notes on Thermodynamic Formalism for Topological Markov shifts}, Penn. State Uni., available 2009.


\bibitem{Su87}  S\"ut\"o A.,
\emph{ The spectrum of a quasiperiodic Schr\"odinger operator},  Comm. Math. Phys. 111 (1987), no. 3, 409-415.

\bibitem{Su}  S\"ut\"o A.,
\emph{  Singular continuous spectrum on a Cantor set of zero Lebesgue measure for the Fibonacci Hamiltonian},  J. Stat. Phys. 56 (1989), 525-531.

\bibitem{T}  Toda M.,
{\em Theory of Nonlinear Lattices}, Number 20 in Solid-State Sciences,
Springer-Verlag, second enlarged edition, 1989. Chap. 4.

\bibitem{Walters} Walters P., {\em An introduction to ergodic theory.} Graduate Texts in Mathematics, 79. Springer-Verlag, New York-Berlin, 1982.

\bibitem{Young} Young L.-S.,
{\em Dimension, entropy and Lyapunov exponents.}
Ergodic Theory Dynam. Systems 2 (1982), no. 1, 109-124.

\end{thebibliography}
\end{document}